\documentclass[12pt,leqno]{article}
\usepackage{comment}
\usepackage{amsmath,amssymb,amsfonts,esint,mathtools}
\usepackage{bm}
\usepackage[dvipdfmx]{graphicx}

\mathtoolsset{showonlyrefs=true}
\numberwithin{equation}{section}
\usepackage{amsthm}

\theoremstyle{plain}
\newtheorem{thm}{Theorem}[section]

\newtheorem{lem}{Lemma}[section]
\newtheorem{cor}{Corollary}[section]
\theoremstyle{definition}
\newtheorem{dfn}{Definition}[section]
\newtheorem{rem}{Remark}[section]

\newcommand{\R}{\mathbb{R}}
\newcommand{\N}{\mathbb{N}}
\newcommand{\Z}{\mathbb{Z}}

\allowdisplaybreaks

\renewcommand{\proofname}{\indent\sc Proof.}

\title{\uppercase{on very weak solutions \\ of certain elliptic systems \\ with double phase growth}} 
\author{
\bigskip \\
\textsc{Yoshiki Kaiho}\\ 
Mathematical Institute, Tohoku University\\
6-3 Aramaki Aza-Aoba, Aoba-ku, Sendai 980-8578, Miyagi, Japan\\
\texttt{kaiho.yoshiki.p7@dc.tohoku.ac.jp}
}

\date{\today} 
\begin{document}
\maketitle

\renewcommand{\proofname}{Proof}

\begin{abstract}
In this paper, we prove a higher integrability result for very weak solutions of higher-order elliptic systems involving a double phase operator as the principal part. 
As a model case, we consider 
\begin{equation}
\int_{\Omega} \left( |D^m u|^{p-2}D^m u + a(x)|D^m u|^{q-2}D^m u \right) \cdot D^m \varphi  = 0
\quad
\text{for any }
\varphi \in C_c^{\infty}(\Omega),
\end{equation}
where $n,\,m \in \mathbb{N},\ n\ge 2,\,1 < p \le q < \infty,\,\Omega \subset \mathbb{R}^n$ is an open set and $a:\Omega \rightarrow [0,\infty)$ is a measurable function. 
We extend the result of Baasandorj, Byun and Kim \cite{BBK} with respect to the derivative order, the coefficient class and the growth conditions.
The proof is based on a construction of an appropriate test function by the Lipschitz truncation technique, a deduction of a reverse H\"older inequality and an application of Gehring's lemma.
To this end, we establish estimates for weighted mean value polynomials and sharp Sobolev--Poincar\'e-type inequalities for the double phase operator. 
\end{abstract}

\section{Introduction}\label{sec:Intro}
\subsection{Background}\label{subsec:bg}
This paper deals with the self-improving property for very weak solutions of higher-order nonlinear elliptic systems with double phase growth. The typical example is the following equation:
\begin{equation}\label{eq:model}
\int_{\Omega} \left( |D^m u|^{p-2}D^m u + a(x)|D^m u|^{q-2}D^m u \right) \cdot D^m \varphi  = 0
\quad
\text{for any }
\varphi \in C_c^{\infty}(\Omega),
\end{equation}
where $n,m \in \mathbb{N},\ n\ge 2,\,1 < p \le q < \infty,\,\Omega \subset \mathbb{R}^n$ is an open set and $a:\Omega \rightarrow [0,\infty)$ is a measurable function. In this paper, a very weak solution refers to a solution whose integrability is lower than that of a standard weak solution (see Definition \ref{def:vws} below). For instance, standard weak solutions to the $p$-Laplace equation belong to the Sobolev space $W^{1,p}$, while very weak solutions typically lie only in $W^{1,p-\varepsilon}$ for some $\varepsilon > 0$. 

A fundamental question is whether a very weak solution automatically improves to a weak solution. 
For the Laplacian, Weyl's lemma gives a stronger conclusion: any locally integrable function satisfying Laplace's equation in the sense of distributions agrees almost everywhere with a smooth harmonic function. 
For general uniformly elliptic linear equations, however, such an improvement is not automatic. 
Indeed, Serrin \cite{Se} constructed examples of very weak solutions that are not weak solutions. 
Nevertheless, if the integrability of a very weak solution is sufficiently close to that of a weak solution (that is, in the example of the preceding paragraph, if $\varepsilon>0$ is sufficiently small), a positive result is available for linear elliptic systems, as shown by Meyers and Elcrat \cite{ME}. 
A similar situation arises for the $p$-Laplacian. Self-improvement results under a suitable integrability assumption were obtained in \cite{IwS,Le}, and their parabolic counterparts were established, for example, in \cite{KL,Bo}. 
On the other hand, Colombo and Tione \cite{CR} recently showed that very weak solutions of the $p$-Laplace equation need not be weak solutions.

While the results mentioned above deal with operators of standard $p$-growth, over the last several decades, much attention has been directed toward operators with nonstandard growth in the context of regularity theory: the terminology nonstandard growth was introduced by Marcellini \cite{Ma} (for the recent developments, see the survey \cite{MRV} and the references therein). For such classes as well, the higher integrability of very weak solutions has been investigated in several settings: see \cite{BoZ} for the elliptic $p(x)$-Laplacian case, \cite{BoL} for the parabolic case, \cite{ByL} for the Orlicz space framework and \cite{BBK} for the double phase operator.

Let us provide further explanation regarding the double phase operator, which is the other main subject of this paper. 
We first consider the following functional:
\begin{equation}
F(u) = \int_{\Omega} \left(|Du|^p + a(x)|Du|^q \right)\,dx,
\end{equation}
where $1 < p < q < \infty$, $\Omega \subset \R^n$ is an open set, $u \in W^{1,1}(\Omega)$, and $a:\Omega \rightarrow [0,\infty)$ is a measurable function.
Since the function $a$ may vanish for some $x \in \Omega$, the nonlinearity of the integrand varies depending on the spatial variable.
Such functionals were first investigated in \cite{Zh} in the context of homogenization theory. The associated Euler-Lagrange equation of $F$ is given by
\begin{equation}
-\mathrm{div} \left( p|Du|^{p-2}Du + qa(x)|Du|^{q-2}Du \right) = 0
\quad
\text{in } 
\Omega.
\end{equation}
Analogous to the functional case, we regard this equation as a nonuniformly elliptic equation. For a more detailed discussion, we refer the reader to \cite[pp.1417--1420]{CoM2}.

For the double phase functional, fundamental regularity results were obtained in \cite{EsLM}. In that paper, the Lavrentiev phenomenon is addressed, and it is shown that the condition 
$$
\frac{q}{p} < 1 + \frac{\alpha}{n}
$$
suffices for the absence of this phenomenon, provided that $a \in C^{0,\alpha}$ for some $\alpha \in (0,1]$ (see also \cite{BCDeM}). That paper also establishes other basic estimates. Since then, a considerable number of regularity results for weak solutions have been established, including $C^{1,\beta}$-regularity, Calder\'{o}n--Zygmund estimates, and related topics (see \cite{EsLM, CoM, CoM2, BaCM}, for example).

Further developments concern nearly linear growth, variable exponents and parabolic settings\,(see \cite{DeM,RaT,KiKM}, for example).
In another direction, Baasandorj, Byun and Kim in \cite{BBK} obtained a result of the following type:\,let $\Omega$ be an open set, $1 < p \le q < \infty$, $\alpha \in (0,1]$, $0 \le a \in C^{0,\alpha}(\Omega)$ and $\frac{q}{p} < 1 + \frac{\alpha}{n}$.
Then there exists $\delta = \delta(n,p,q,a) \in (0,1)$ such that if $u \in W^{1,1}_{loc}(\Omega)$ satisfies 
\begin{equation}
\int_{\Omega_0} \left( |Du|^p + a(x)|Du|^q\right)^{\delta}\ dx < \infty
\quad
\text{for any } 
\Omega_0 \subset \subset \Omega
\end{equation}
and
\begin{equation}
\int_{\Omega} \left( |D u|^{p-2}D u + a(x)|D u|^{q-2}D u \right) \cdot D \varphi\ dx  = 0
\quad
\text{for any } 
\varphi \in C_c^{\infty}(\Omega),
\end{equation}
then
\begin{equation}
\int_{\Omega_0} \left( |Du|^p + a(x)|Du|^q \right)^{\frac{1}{\delta}}\ dx < \infty
\quad
\text{for any } 
\Omega_0 \subset \subset \Omega.
\end{equation}

Lewis \cite{Le} established self-improvement results for higher-order elliptic systems with standard $p$-growth and natural lower-order terms, whereas Baasandorj, Byun and Kim \cite{BBK} obtained analogous results for first-order double phase systems without lower-order terms. This comparison motivates us to ask whether the self-improvement property persists for higher-order double phase systems with natural lower-order terms. The aim of this paper is to extend the results of \cite{BBK} with respect to the derivative order, the class of the coefficient function $a$ (see Definition~\ref{def:Zalpha} below), and the growth conditions. In particular, we allow lower-order terms subject to integrability conditions analogous to those in \cite{Le} (see Remark~1.2 below). To this end, we construct suitable test functions and establish Sobolev--Poincar\'{e}-type inequalities adapted to the higher-order double phase setting. After stating our main theorem in Subsection~\ref{subsec:main-thm}, we outline the proof and describe the structure of the paper in Subsection~\ref{subsec:outline}.

\subsection{Main theorem}\label{subsec:main-thm}
Throughout this paper, we assume that $n,m,N \in \mathbb{N},\ n\ge 2$, and $\Omega \subset \mathbb{R}^n$ is an open set. We first provide the definition of the class $\mathcal{Z}^{\alpha}$. This class was introduced in \cite{BCDeM} in the context of the Lavrentiev phenomenon for double phase functionals.
\begin{dfn}\label{def:Zalpha}
Let $\alpha \in (0, \infty)$ and let $\Omega_0 \subset \mathbb{R}^n$ be an open set.
A measurable function $a:\Omega_0 \rightarrow [0,\infty)$ is said to belong to the class $\mathcal{Z}^{\alpha}(\Omega_0)$  
if there exists a constant $C > 0$ such that
\begin{equation}\label{eq:Zalpha}
  a(x) \le  C\left( a(y) + |x - y|^{\alpha} \right)
  \quad \text{for any}\ x, y \in \Omega_0.
\end{equation}
We denote the infimum of such constants $C$ by $[a]_{\alpha,\Omega_0}$ (or simply $[a]_{\alpha}$).
\end{dfn}

\begin{rem}
We first note that for $\alpha \in (0,1]$, $\mathcal{Z}^{\alpha}(\Omega_0)$ strictly includes the set of nonnegative $\alpha$-H\"older continuous functions (see \cite[Remark 1.2.1]{BCDeM}). 
\end{rem}

Here we assume
\begin{equation}\label{assmpt:exponents}
1 < p \le q < \infty,
\quad
\alpha \in (0,\infty),
\quad
a \in \mathcal{Z}^{\alpha}(\Omega)
\quad
\text{and}
\quad
\frac{q}{p} < 1 + \frac{\alpha}{n}.
\end{equation}

In what follows, for $\ell \in \{0,\ldots,m\}$, let 
\begin{equation}
S_{\ell} 
\coloneqq 
\{\sigma = (\sigma_1, \ldots, \sigma_n) \in \Z_{\ge 0}^n: |\sigma| =\sigma_1 + \cdots + \sigma_n= \ell\},
\quad
S 
\coloneqq 
\bigcup_{\ell=0}^{m} S_{\ell}
\end{equation}
denote the sets of all multi-indices of order $\ell$ and of order at most $m$, respectively. 
Furthermore, for a smooth $\R^N$-valued function $u=(u_1,\ldots,u_N)$, we define the $\ell$-th order derivative array
\begin{equation}
D^\ell u 
\coloneqq 
\{\partial_\sigma u_i \}_{\sigma \in S_{\ell}, 1 \le i \le N}
\end{equation}
and its Euclidean norm
\begin{equation}
|D^\ell u| 
\coloneqq \left(\sum_{\sigma \in S_{\ell}}\sum_{i=1}^{N} |\partial_{\sigma} u_i|^2 \right)^{\frac{1}{2}}.
\end{equation}
In particular, $D^1 u = Du$ denotes the standard gradient. We consider $u \in W^{m,1}_{loc}(\Omega;\R^N)$ such that
\begin{equation}\label{eq:main}
    \sum_{\sigma \in S} \int_{\Omega}  A_\sigma(x,u,\ldots,D^m u) \cdot \partial_\sigma \varphi = 0
    \quad
   \text{for any }
   \varphi \in C_c^\infty(\Omega;\R^N),
\end{equation}
where the center dot denotes the standard inner product in $\R^N$ and $A_\sigma:\Omega \times \mathbb{R}^{N \sharp S} \rightarrow \R^N$ 
is a Carath\'{e}odory vector field for each $\sigma \in S$.
We write
\begin{equation}
\xi = \{\xi_{\tau}\}_{\tau \in S} \in \R^{N \sharp S}
\quad
\text{with}
\quad
\xi_{\tau} \in \R^N
\quad
\text{and}
\quad
\xi_m = \{\xi_{\tau}\}_{\tau \in S_m}.
\end{equation}
We assume that there exists $\nu > 0$ such that for a.e.\,$x \in \Omega$ and all $\xi \in \R^{N \sharp S}$,
\begin{equation}\label{eq:coercivity}
    \nu^{-1} \sum_{\sigma \in S_m} A_\sigma(x,\xi) \cdot \xi_{\sigma} \geq  |\xi_m|^p + a(x)|\xi_m|^q - \left( f_p+a(x)f_q \right).
\end{equation}
Moreover,  for a.e.\,$x \in \Omega$, any $\ell \in \{0,\ldots,m\}$ and any $\sigma \in S_{\ell}$, it holds that
\begin{equation}\label{eq:growth}
    \left| A_\sigma(x,\xi) \right| \leq  g_{p,\ell}|\xi_m|^{p-1} + h_{p,\ell} +  a(x)^{\frac{1}{q}} \left( g_{q,\ell}a(x)^{\frac{q-1}{q}}|\xi_m|^{q-1} + h_{q,\ell} \right).
\end{equation}
Here $f_r,\,g_{r, \ell},\,h_{r,\ell}$ ($r \in \{p,q\}$) are nonnegative measurable functions on $\Omega$. 
In order to state the assumptions on these functions, we introduce notation.
For $t \in [1,\infty]$, let $t'$ be the H\"older conjugate of $t$, that is,
\begin{equation}\label{eq:Holder conjugate}
t' \coloneqq 
\left\{ 
  \begin{alignedat}{3}   
    \infty \hspace{25pt} &\text{if}\,t=1, \\
    \frac{t}{t-1} \hspace{12pt} &\text{if}\,1 < t < \infty, \\
    1                  \hspace{25pt} &\text{if}\,t=\infty.
  \end{alignedat} 
  \right.
\end{equation}
In addition, for each $\ell \in \{0, \ldots, m\}$ we define
\begin{equation}\label{eq:pell*}
  (t_{\ell})^* \coloneqq
  \left\{ 
  \begin{alignedat}{2}   
    \frac{nt}{n - \ell t} \hspace{12pt} &\text{if}\,\ell t < n, \\
    \infty                        \hspace{25pt} &\text{otherwise}.
  \end{alignedat} 
  \right.
\end{equation}
In particular, we denote $(t_1)^* = t^*$. With this notation, we assume the following integrability conditions:
\begin{equation}\label{assumpt:functions}
 \left(f_p + af_q \right) \in L^{\beta}_{loc}(\Omega),
 \quad
 g_{r,\ell} \in L^{s_{r,\ell}}_{loc}(\Omega)
 \quad
 \text{and}
 \quad
 h_{r,\ell} \in L^{t_{r,\ell}}_{loc}(\Omega),
\end{equation}
where
\begin{equation}\label{def:exponents}
    \beta > 1,
    \quad
    s_{r,\ell} > \frac{1}{1 - \frac{1}{(r_{m-\ell})^*} - \frac{1}{r'}},
    \quad
    t_{r,\ell} > \frac{1}{1 - \frac{1}{(r_{m-\ell})^*}}
\end{equation}
for each $\ell \in \{0,\ldots,m\}$ and $r \in \{p,q\}$.
Here we interpret $s_{r,m} = \infty$.
Moreover, we assume $g_{r,m} \equiv 1$ for brevity.

\begin{rem}\label{rem:growth}
The above condition \eqref{def:exponents} naturally arises to ensure the well-definedness of the system \eqref{eq:main}. We also remark that this condition is precisely the one used in Meyers--Elcrat \cite{ME} and Lewis \cite{Le}:
the former concerns higher integrability results for weak solutions of nonlinear elliptic systems with standard $p$-growth, while the latter deals with very weak solutions.
\end{rem}

\begin{dfn}\label{def:vws}
Let $u \in W^{m,1}_{loc}(\Omega;\R^N)$. If $u$ satisfies
\begin{equation}
\int_{\Omega_0}  \left( |D^m u|^p + a(x)|D^m u|^q \right) < \infty
\quad
\text{for any }
\Omega_0 \subset \subset \Omega
\end{equation}
and \eqref{eq:main}, then $u$ is called a weak solution of \eqref{eq:main}. On the other hand, if there exists $\delta \in (0,1)$ such that
\begin{equation}
\int_{\Omega_0}  \left( |D^m u|^p + a(x)|D^m u|^q \right)^{\delta} < \infty
\quad
\text{for any }
\Omega_0 \subset \subset \Omega
\end{equation}
and $u$ satisfies \eqref{eq:main},
then $u$ is called a very weak solution of \eqref{eq:main} with $\delta$.
\end{dfn}

\begin{rem}
Let us give a remark on the fourth condition in \eqref{assmpt:exponents}. In the regularity theory of the double phase operator, it is standard to assume $\frac{q}{p} \le 1 + \frac{\alpha}{n}$ in order to control the weighted $q$-th power term (see \cite[Proposition 3.1]{CoM}, for example).
However, for the same reason as explained in \cite[Remark 1.3]{BBK}, which treats the case $m=1$, we exclude the borderline case $\frac{q}{p} = 1 + \frac{\alpha}{n}$.
\end{rem}

\begin{thm}\label{thm:main}
Assume \eqref{assmpt:exponents}, \eqref{eq:coercivity}, \eqref{eq:growth}, \eqref{assumpt:functions} and \eqref{def:exponents}. Then there exists
\begin{equation}
\delta = \delta(n,m,N,p,q,\alpha,[a]_{\alpha},\nu,\beta,s_{p,\ell},s_{q,\ell},t_{p,\ell},t_{q,\ell}) \in (0,1)
\end{equation}
such that if $u \in W^{m,1}_{loc}(\Omega;\R^N)$ is a very weak solution of \eqref{eq:main} with this $\delta$, then it holds that
 \begin{equation}  
 \int_{\Omega_0}  \left( |D^m u|^p + a(x)|D^m u|^q \right)^{\frac{1}{\delta}} < \infty
 \quad
 \text{for any }
 \Omega_0 \subset \subset \Omega.
\end{equation}
In particular, $u$ is a weak solution of \eqref{eq:main}.
\end{thm}
By a suitable modification, we also have
\begin{cor}\label{cor:main}
Under the same assumptions as in Theorem \ref{thm:main} except that $\beta=1$,
there exists
\begin{equation}
\delta = \delta(n,m,N,p,q,\alpha,[a]_{\alpha},\nu,s_{p,\ell},s_{q,\ell},t_{p,\ell},t_{q,\ell}) \in (0,1)
\end{equation}
such that if $u \in W^{m,1}_{loc}(\Omega;\R^N)$ is a very weak solution of \eqref{eq:main} with this $\delta$, then it holds that
 \begin{equation}  
 \int_{\Omega_0}  \left( |D^m u|^p + a(x)|D^m u|^q \right) < \infty
 \quad
 \text{for any }
 \Omega_0 \subset \subset \Omega.
\end{equation}
In particular, $u$ is a weak solution of \eqref{eq:main}.
\end{cor}
\begin{rem}
If $a \equiv 0$, Theorem \ref{thm:main} agrees with the result of Lewis \cite{Le}.
On the other hand, if $m=1,\, 0 \le a \in C^{0,\alpha}(\Omega)\,(\alpha \in (0,1])$ and $g_{r,0} = h_{r,0} \equiv 0$ for each $r \in \{p,q\}$, then Corollary \ref{cor:main} coincides with the result of Baasandorj, Byun and Kim \cite{BBK}.
In this first-order setting without lower-order terms,
the corresponding result for coefficients in
$\mathcal{Z}^{\alpha}$ with $\alpha \in (0,1]$
also follows from \cite{BBK}.
This is because each such coefficient
is comparable to a nonnegative $C^{0,\alpha}$ coefficient,
and the same operator satisfies the corresponding
coercivity and growth conditions with modified constants.
Our main contributions beyond this reduction are the treatment
of higher-order systems, the inclusion of lower-order terms,
and the extension to coefficients in $\mathcal{Z}^{\alpha}$
with $\alpha > 1$.
\end{rem}

\subsection{Outline of proof}\label{subsec:outline}

Our proof consists of three steps: first, we construct a suitable test function via Lipschitz truncation. Next, by substituting this test function into the system \eqref{eq:main}, we deduce a reverse H\"older inequality. Finally, applying Gehring's lemma to the reverse H\"older inequality, we obtain the higher integrability result. For a concise discussion of the case $m=1$  and $a \equiv 0$, see \cite[pp.1518-1520]{Le} and \cite[Section 12.3]{KLV}.

In the context of very weak solutions, the Lipschitz truncation method was first introduced by Lewis \cite{Le}. Unlike weak solutions, very weak solutions do not allow us to use test functions involving the solutions themselves due to their lower integrability. To overcome this difficulty, we employ the Lipschitz truncation method. This method enables us to construct a truncated function whose gradient is bounded by a prescribed constant, while preserving weak differentiability. 

In the work of Lewis \cite{Le}, a general extension theorem is used for the construction of a truncated function. However, as explained in \cite[pp.8735-8736]{BBK} for the specific case $m=1$, in order to construct an appropriate truncated function for the double phase operator, we require the truncated function $v_{\lambda}$ to satisfy, for a given $\lambda > 0$
\begin{equation}\label{eq:principle}
|Dv_{\lambda}| \lesssim \lambda^\frac{1}{p},
\quad
a^{\frac{1}{q}}|Dv_{\lambda}| \lesssim \lambda^\frac{1}{q}.
\end{equation}
To the best of our knowledge, there is no general extension theorem preserving these two types of estimates simultaneously. Therefore, we must perform a direct construction.

Based on this principle, we need to generalize the Lipschitz truncation developed in \cite{BBK} to higher orders. In standard higher-order theory, 
one uses a mean value polynomial (\cite{Bo}) for the extension. 
However, the polynomial contains derivatives of $u$ up to order $m-1$. This implies we would need to control higher derivatives when estimating the lower derivatives of the Lipschitz truncation. This feature is not suitable for our growth condition (cf. \cite{BBK,Bo}). To avoid this issue, we replace the standard mean value polynomial with a weighted mean value polynomial (see Subsection 2.2), obtained by weighting the integral average with a cut-off function. With this modification, we lower the derivative order via integration by parts.

After constructing the appropriate Lipschitz truncation, our next concern is the deduction of a reverse H\"older inequality (see Lemma \ref{lem:Reverse-Holder}). In order to prove this inequality, a Sobolev--Poincar\'e inequality associated with the double phase operator is needed. To this end, we establish the following theorem:
\begin{thm}\label{thm:sp-0}
Let $1 < p \le q < n$, $\alpha \in (0,\infty)$ and $\frac{q}{p} \le 1 + \frac{\alpha}{n}$, and let $B=B_R \subset \R^n$ be an open ball of radius $R>0$ and $a \in \mathcal{Z}^{\alpha}(B)$. Moreover, assume  $u \in W^{1,p}(B)$ satisfies 
\begin{equation}
\int_{B} a|Du|^q  < \infty
\quad
\text{and}
\quad
u_{B} = 0.
\end{equation}
Then there exists a constant $c=c(n,p,q,\alpha,[a]_{\alpha})$ such that
\begin{equation}\label{eq:sp-0}
\left( \frac{1}{|B|}\int_{B} a^{\frac{q^*}{q}}\left| \frac{u}{R} \right|^{q*} \right)^{\frac{1}{q^*}} 
\le c \left( \frac{1}{|B|}\int_{B} a|Du|^{q} \right)^{\frac{1}{q}}
+ c R^{\frac{\alpha}{q}} \left( \frac{1}{|B|}\int_{B} |Du|^{p} \right)^{\frac{1}{p}}
\end{equation}
Here $q^* = \frac{nq}{n-q}$ and $u_B = \frac{1}{|B|}\int_{B} u$.
\end{thm}
We first note that this result coincides with the classical Sobolev--Poincar\'e inequality if $a \equiv 1$ because from its proof one can see that the second term can be omitted in this case. 
We also remark that the higher-order case is obtained by induction (Corollary\,\ref{cor:sp}).
In the standard approach (see \cite[Lemma 3.10]{BBK} and \cite[Theorem 2.13]{Ok}, for example), the fact that the exponent, which appears on the left-hand side of \eqref{eq:sp-0}, is lower than $p^*$ is essentially used. 
Therefore, we can not obtain Theorem \ref{thm:sp-0} by modifying their argument. 
Recently a very similar result was obtained in \cite[Theorem 3.6]{CD} in a more general form using Sobolev--Musielak spaces. 
However, in their result, the constant $c$ is dependent on the size of the domain. By carefully examining their argument, we can make the dependence on the size of the domain explicit. Instead, we provide an elementary proof specialized for the double phase operator. The key to the proof is a simple inequality involving the Riesz potential (see Lemma \ref{lem:beta}).

At the end of this subsection, we outline the structure of the paper.
Section 2 introduces notation and auxiliary lemmas used throughout the proof of the main theorem.
Section 3 establishes the Sobolev--Poincar\'e inequality for the double phase operator; we also prove Theorem \ref{thm:sp-0} there.
Section 4 is devoted to the proof of the main theorem.
Finally, we provide an elementary proof of Gehring's lemma in the Appendix.

\section{Preliminaries}

Throughout the paper, we denote by $c$ a general positive constant, possibly varying from line to line.
Moreover, relevant dependencies on parameters will be emphasized using parentheses,  
i.e., 
$$
c = c(n, p, q)
$$
means that $c$ depends on $n, p$ and $q$.

For $x_0 \in \R^n$ and $r>0$, the open ball of radius $r$ centered at $x_0$ is written as
\begin{equation}
B(x_0,r) \coloneqq \{ x \in \mathbb{R}^n : |x - x_0| < r \}.
\end{equation}
When the center or radius is irrelevant or clear from the context,  
we simply write $B_r$ or $B$.

Let $B \subset \mathbb{R}^n$ be a measurable subset with $0 < |B| < \infty$,
and let $f = (f_1,\ldots,f_k): B \to \mathbb{R}^k$ ($k \in \N$) be an integrable map. For a given weight function $\eta \in L^\infty(B) \setminus \{0\}$, we define the weighted average of $f$ over $B$ with respect to $\eta$ by
\begin{equation}\label{eq:weight-mean}
f_{B,\eta} \coloneq (f)_{B,\eta} \coloneq \left( \frac{1}{\|\eta\|_{L^1(B)}} \int_B f_1\eta,\ldots, \frac{1}{\|\eta\|_{L^1(B)}} \int_B f_k\eta\right).
\end{equation}
When $\eta \equiv 1$, we simply write $f_B$ or $(f)_B$.

Finally, we denote by $\omega_n$ the volume of the unit ball in $\R^n$. 
\subsection{Some properties of maximal function}\label{subsec:Mf}

Let $f \in L^1_{\text{loc}}(\mathbb{R}^n;\R^k)$ with $k \in \N$ and $0\le\beta<n$. 
The uncentered (fractional) maximal function of $f$ is defined, for every $x \in \R^n$, by
\begin{equation} \label{eq:Mfunct}
M_{\beta}(f)(x) \coloneqq \sup_{x \in B(x_0,r)} r^{\beta}\fint_{B(x_0,r)} |f|
\end{equation}
where the supremum is taken over all balls $B(x_0,r) \subset \mathbb{R}^n$ containing the point $x$.
Moreover, we denote $M_0 = M$. 
For the basic properties of the (fractional) maximal function, see \cite[Chapter 1]{KLV}.

Here we recall the following $L^p$-boundedness of the maximal function.
This statement follows from the $L^p$-boundedness for the centered maximal function (\cite[Theorem\,1.15]{KLV}).
\begin{lem}\label{lem:Mfunct}
Let $1 < s < \infty$. Then there exists $c=c(n,s)>0$ such that
\begin{equation}
    \int_{\mathbb{R}^n} M(f)^s \leq c \int_{\mathbb{R}^n} |f|^s
\end{equation}
for any $f \in L^s(\R^n)$.
Moreover, the constant $c(n,s)$ depends continuously on $s$.
\end{lem}

The following lemma concerns the composition of the maximal function with its fractional counterpart. 
It will be used in the proof of Lemma \ref{lem:Hedberg-2}.
\begin{lem}\label{lem:composite}
Let $0<\beta<n$. Then there exists $c=c(n,\beta)$ such that
\begin{equation}\label{eq:composite-conc}
M(M_{\beta}(f))(x) \le cM_{\beta}(f)(x)
\end{equation}
for any $x \in \R^n$ and $f \in L^1_{loc}(\R^n)$ with $M_{\beta}(f) \in L^1_{loc}(\R^n)$.
\end{lem}

\begin{proof}
This is a Euclidean special case of \cite[Theorem 8.1]{HeKNT}. 
We include a direct proof because the estimates obtained below will also be used in Remark~2.1.
Let us define
\begin{equation}
M^c(f)(x) \coloneqq \sup_{r>0} \fint_{B(x,r)} |f|
\quad
\text{and}
\quad
M^c_{\beta}(f)(x) \coloneqq \sup_{r>0} r^{\beta}\fint_{B(x,r)} |f|.
\end{equation}
Since it holds
\begin{equation}
M^c(f)(x) \le M(f)(x) \le 2^nM^c(f)(x)
\quad
\text{and}
\quad
M_{\beta}^c(f)(x) \le M_{\beta}(f)(x) \le 2^{n-\beta}M_{\beta}^c(f)(x),
\end{equation}
for any $x \in \R^n$, it is enough to verify there exists $c=c(n,\beta)>0$ such that
\begin{equation}
M^c(M^c_{\beta}(f))(x) \le cM^c_{\beta}(f)(x)
\end{equation}
for any $x \in \R^n$. For the rest of the proof,
we denote $M^c = M$ and $M_{\beta}^c = M_{\beta}$.

Fix any $x \in \R^n$ and $r>0$. It suffices to show that
\begin{equation}\label{eq:comosite-goal}
\fint_{B(x,r)} M_{\beta}(f)(y)\, dy \le c(n,\beta)M_{\beta}(f)(x).
\end{equation}
For the left-hand side of \eqref{eq:comosite-goal}, we have
\begin{align}\label{eq:composite-I's}
&\fint_{B(x,r)} M_{\beta}(f)(y)\, dy \\
&= \fint_{B(x,r)} \left( \sup_{s>0}s^\beta \fint_{B(y,s)} |f(z)|\,dz \right) dy \\
&\le \fint_{B(x,r)} \left( \sup_{s>r}s^\beta \fint_{B(y,s)} |f(z)|\,dz \right) dy
+ \fint_{B(x,r)} \left( \sup_{0 < s \le r}s^\beta \fint_{B(y,s)} |f(z)|\,dz \right) dy \\
&\eqqcolon I_1 + I_2.
\end{align}

Firstly, we estimate $I_1$.
If $y \in B(x,r)$ and $s > r$, then we get
\begin{equation}
B(y,s) \subset B(x,s+r) \subset B(x,2s)
\end{equation}
and
\begin{align}
s^\beta \fint_{B(y,s)} |f(z)|\,dz
&= \frac{s^\beta}{|B(y,s)|} \int_{B(y,s)} |f(z)|\,dz \\
&\le \frac{s^\beta}{|B(y,s)|} \int_{B(x,2s)} |f(z)|\,dz \\
&= \frac{2^{n-\beta}(2s)^\beta}{|B(x,2s)|} \int_{B(x,2s)} |f(z)|\,dz
\le 2^{n-\beta}M_{\beta}(f)(x).
\end{align}
Therefore, we obtain
\begin{equation}\label{eq:composite-I_1}
I_1 
\le 2^{n-\beta}\fint_{B(x,r)} M_{\beta}(f)(x)\,dy
= 2^{n-\beta}M_{\beta}(f)(x).
\end{equation}

Next we estimate $I_2$. For any $y \in \R^n$ and $j \in \N$ with $2^{-j}r < s \le 2^{-j+1}r$, we have
\begin{align}
&s^{\beta}\fint_{B(y,s)}|f(z)|\,dz \\
&\le \frac{(2^{-j+1}r)^{\beta}}{|B(y,2^{-j}r)|}\int_{B(y,2^{-j+1}r)}|f(z)|\,dz
= 2^n(2^{-j+1}r)^{\beta-n}(\omega_n)^{-1}\int_{B(y,2^{-j+1}r)}|f(z)|\,dz.
\end{align}
This inequality and the monotone convergence theorem yield that
\begin{align}\label{eq:composite-I_2}
I_2
&\le \fint_{B(x,r)} \sum_{j=1}^{\infty}\left( \sup_{2^{-j}r < s \le 2^{-j+1}r}s^\beta \fint_{B(y,s)} |f(z)|\,dz \right)\,dy \\
&\le 2^n(\omega_n)^{-1}\fint_{B(x,r)} \sum_{j=1}^{\infty}\left( (2^{-j+1}r)^{\beta-n}\int_{B(y,2^{-j+1}r)}|f(z)|\,dz \right)\,dy \\
&= 2^n(\omega_n)^{-1}\sum_{j=1}^{\infty}(2^{-j+1}r)^{\beta-n} \fint_{B(x,r)}
\left( \int_{B(y,2^{-j+1}r)}|f(z)|\,dz \right)\,dy
\eqqcolon I_2'.
\end{align}
Moreover, for any $j \in \N$ and $y \in B(x,r)$, we have $B(y,2^{-j+1}r) \subset B(x,2r)$ and
\begin{align}
\fint_{B(x,r)} 
\left( \int_{B(y,2^{-j+1}r)}|f(z)|\,dz \right)\,dy
&\le \fint_{B(x,r)} 
\left( \int_{B(x,2r)}|f(z)|\chi_{B(y,2^{-j+1}r)}(z)\,dz \right)\,dy \\
&= 2^n\fint_{B(x,2r)}|f(z)|
\left( \int_{B(x,r)}\chi_{B(z,2^{-j+1}r)}(y)\,dy \right)\,dz \\
&= 2^n\fint_{B(x,2r)}|f(z)||B(x,r) \cap B(z,2^{-j+1}r)|\,dz \\
&\le 2^n(2^{-j+1}r)^{n}\omega_n\fint_{B(x,2r)}|f(z)|\,dz
\end{align}
by Fubini's theorem.
Therefore, we get
\begin{align}\label{eq:composite-I_2'}
I_2'
&\le 4^n\sum_{j=1}^{\infty}(2^{-j+1}r)^{\beta}\fint_{B(x,2r)}|f(z)|\,dz \\
&= 4^n \cdot (2r)^{\beta} \fint_{B(x,2r)}|f(z)|\,dz \sum_{j=1}^{\infty} 2^{-j\beta}
\le 4^nM_{\beta}(f)(x)\sum_{j=1}^{\infty} 2^{-j\beta}. 
\end{align}
Since $\beta > 0$, 
it follows that $\sum_{j=1}^{\infty} 2^{-j\beta} = \frac{1}{2^{\beta}-1}$. Therefore, by \eqref{eq:composite-I_2} and \eqref{eq:composite-I_2'}, we obtain
\begin{equation}\label{eq:composite-I_2-2}
  I_2 \le I_2' \le \frac{4^n}{2^\beta - 1}M_{\beta}(f)(x).
\end{equation}
From \eqref{eq:composite-I's},  \eqref{eq:composite-I_1} and \eqref{eq:composite-I_2-2}, we derive the desired inequality \eqref{eq:comosite-goal}.
\end{proof}

\begin{rem}\label{rem:composite}
In Lemma \ref{lem:composite}, we assume $f \in L^1_{loc}(\R^n)$ with $M_{\beta}(f) \in L^1_{loc}(\R^n)$. This condition is satisfied if $f \in L^t(\R^n)$ for some $t \in [1,\infty)$ and $\beta \le \frac{n}{t}$. Indeed, fix $x \in \R^n$ and $B(x,r)$. It is sufficient to show that
\begin{equation}
  \fint_{B(x,r)}M_{\beta}(f)(y)dy < \infty.
\end{equation}
Moreover, it is enough to verify that the terms $I_1$ and $I_2$ in \eqref{eq:composite-I's} are finite.

First, for $I_1$, Jensen's inequality yields
\begin{align}
  I_1 
  &\le
  \fint_{B(x,r)} \left\{ \sup_{s>r}s^\beta \left( \fint_{B(y,s)} |f(z)|^t\,dz \right)^{\frac{1}{t}} \right\} dy \\
  &= {\omega_n}^{-\frac{1}{t}} \fint_{B(x,r)} \left\{ \sup_{s>r}s^{\beta-\frac{n}{t}} \left( \int_{B(y,s)} |f(z)|^t\,dz \right)^{\frac{1}{t}} \right\} dy
  \le {\omega_n}^{-\frac{1}{t}}r^{\beta-\frac{n}{t}}\|f\|_{L^t(\R^n)}.
\end{align}
In the last inequality, we used $f \in L^t(\R^n)$ and $\beta \le \frac{n}{t}$.

On the other hand, for $I_2$, from \eqref{eq:composite-I_2} and \eqref{eq:composite-I_2'}, it follows that
\begin{equation}
  I_2 \le I_2' \le c\sum_{j=1}^{\infty}(2^{-j+1}r)^{\beta}\fint_{B(x,2r)}|f(z)|\,dz < \infty.
\end{equation}
Therefore, we conclude $M_{\beta}(f) \in L^1_{loc}(\R^n)$.
\end{rem}

For the rest of this subsection, we prove two lemmas: Lemma \ref{lem:Hedberg-1} is a slightly modified version of the well-known Hedberg type estimates (\cite{He}). 
On the other hand, Lemma \ref{lem:Hedberg-2} provides the inequality with the product of $a \in \mathcal{Z^\alpha}$ and a maximal function. 
Both lemmas are used in Subsection 4.3. In what follows, we fix $B = B_R$. For the sake of readability, we introduce the following two assumptions:
\begin{equation}\label{assmpt:He-1}
\eta \in L^{\infty}(B),
\quad
0 \le \eta \le 1,
\quad
\|\eta\|_{L^1(B)} \ge \kappa|B|.
\end{equation}
for some $\kappa = \kappa(n) > 0$ and
\begin{equation}\label{assmpt:He-2}
\alpha \in (0,\infty), \quad
a \in \mathcal{Z}^{\alpha}(B).
\end{equation}
In addition, we write for any $f \in L^t(B)\,(t \in (1,\infty])$,
\begin{equation}
M_B(f) = M(f \chi_B)
\quad
\text{and}
\quad
M_B^\ell(f) = M_B^{\ell - 1}(M_Bf)
\quad
\text{for}
\quad
\ell \ge 2.
\end{equation} 
Let us begin with the following lemma.
\begin{lem}\label{lem:Hedberg-1}
Assume \eqref{assmpt:He-1}. Moreover, suppose $u \in W^{\ell,t}(B)\,(t \in (1,\infty])$ and $(D^k u)_{B,\eta} = 0$ for any $k \in \{0,\ldots,\ell-1\}$. Then there exists $c=c(n,\ell)$ such that
\begin{equation}\label{eq:Hedberg-1}
|u(x)| \le cR^{\ell}M_B^{\ell}(|D^\ell u|)(x)
\quad
\text{for a.e.\,}
x \in B.
\end{equation}
\end{lem}
\begin{proof}
We first consider the case $\ell=1$. 
From \cite[Theorem 3.5]{KLV}, it follows that
\begin{align}\label{eq:He-1-2}
|u(x)-u_B| 
&\le cRM(|Du|\chi_B)(x) \\
&= c(n)RM_B(|Du|)(x)
\quad
\text{for a.e.\,}
x \in B.
\end{align}
Then, by the triangle inequality and the assumption $u_{B,\eta}=0$, we have
\begin{align}\label{eq:He-1-1}
|u(x)| 
&\le |u(x)-u_B| + |u_B-u_{B,\eta}| \\
&\le c(n)RM_B(|Du|)(x) + |u_B-u_{B,\eta}|
\quad
\text{for a.e.\,}
x \in B.
\end{align}
For the second term in \eqref{eq:He-1-1}, 
using \eqref{assmpt:He-1} and Poincar\'{e}'s inequality, we deduce
\begin{align}\label{eq:He-1-3}
|u_B-u_{B,\eta}|
&\le \frac{1}{\|\eta\|_{L^1(B)}}\int_{B}|u(y)-u_B|\eta\, dy \\
&\le c\fint_{B}|u(y)-u_B|\, dy
\le cR\fint_{B}|Du(y)|\, dy
\le cRM_B(|Du|)(x)
\end{align}
for any $x \in B$.
Therefore, from \eqref{eq:He-1-1} and \eqref{eq:He-1-3}, we obtain
\begin{equation}
|u(x)| \le c(n)RM_B(|Du|)(x)
\quad
\text{for a.e.\,}
x \in B.
\end{equation}
For general $\ell \in \N$, the assertion follows by induction.
This completes the proof.
\end{proof}

Next we show the following lemma.
\begin{lem}\label{lem:Hedberg-2}
Assume \eqref{assmpt:He-2} and $f \in L^t(B)\,(t \in (1,\infty))$.
Moreover, suppose $R \le 1$, $\ell \in \N$, $q \in (1,\infty)$ and $\beta \in \left(0,\min\left\{\frac{\alpha}{q}, \frac{n}{t}\right\} \right]$. Then there exists $c=c(n,\ell, q,\alpha,[a]_{\alpha},\beta)$ such that
\begin{equation}
a(x)^{\frac{1}{q}}M_B^{\ell}(f)(x)
\le cM_B^{\ell}(a^{\frac{1}{q}}f)(x)
+ cM_{\beta}(M^{\ell-1}(f\chi_B))(x)
\quad
\text{for any }
x \in B.
\end{equation}
Here $M^{\ell}(f)=M(M^{\ell-1}(f))$ for $\ell \ge 2$.
\end{lem}

\begin{proof}
We prove by induction on $\ell$. Let us consider the case $\ell=1$. 
Fix any $x \in B$.
Then we have
\begin{align}
a(x)^{\frac{1}{q}}M_B(f)(x)
&= a(x)^{\frac{1}{q}}\left(\sup_{x \in B(x_0,r)} \fint_{B(x_0,r)} |f(y)|\chi_B(y) \,dy  \right)\\
&= \sup_{x \in B(x_0,r)} \fint_{B(x_0,r)} a(x)^{\frac{1}{q}}|f(y)|\chi_B(y) \,dy
\eqqcolon I.
\end{align}
Here fix any ball $B(x_0,r)$ containing $x$. Then from \eqref{assmpt:He-2}, $R \le 1$, and $0 < \beta \le \frac{\alpha}{q}$, it follows that
\begin{align}
a(x)^{\frac{1}{q}}
&\le [a]_{\alpha}^{\frac{1}{q}} \left( a(y)^{\frac{1}{q}} + |x-y|^{\frac{\alpha}{q}} \right)\\
&= [a]_{\alpha}^{\frac{1}{q}}\left( a(y)^{\frac{1}{q}} + |x-y|^{\frac{\alpha}{q}-\beta}|x-y|^{\beta} \right)\\
&\le [a]_{\alpha}^{\frac{1}{q}}\left( a(y)^{\frac{1}{q}} + (2R)^{\frac{\alpha}{q}-\beta}(2r)^{\beta} \right)
\le c(q,\alpha,[a]_{\alpha},\beta)\left( a(y)^{\frac{1}{q}} + cr^{\beta} \right)
\end{align}
for any $y \in B(x_0,r) \cap B$.
Therefore, we obtain
\begin{align}
I
&\le c\left(\sup_{x \in B(x_0,r)} \fint_{B(x_0,r)} a(y)^{\frac{1}{q}}|f(y)|\chi_B(y) \,dy \right)
+ c \left(\sup_{x \in B(x_0,r)} \fint_{B(x_0,r)} r^{\beta}|f(y)|\chi_B(y) \,dy \right) \\
&\le cM(a^{\frac{1}{q}}f\chi_{B})(x) + cM_{\beta}(f\chi_{B})(x)
=cM_B(a^{\frac{1}{q}}f)(x) + cM_{\beta}(f\chi_{B})(x),
\end{align}
where $c=c(q,\alpha,[a]_{\alpha},\beta)$. This completes the proof of the case $\ell = 1$.

Next we show the case $\ell = k+1$, assuming the assertion holds for some $k \in \N$.
We fix $x \in B$.
Since $M_B^k(f) \in L^t(B)$ by Lemma \ref{lem:Mfunct}, the case $\ell=1$ implies
\begin{align}\label{eq:Hedberg-2-I's}
a(x)^{\frac{1}{q}}M_B^{k+1}(f)(x)
&= a(x)^{\frac{1}{q}}M_B \left( M_B^{k}(f) \right)(x) \\
&\le cM_B \left( a^{\frac{1}{q}}M_B^{k}(f) \right)(x)
+ cM_\beta \left( M_B^{k}(f)\chi_B\right)(x) \\
&\eqqcolon c(q,\alpha,[a]_{\alpha},\beta) (I_1 + I_2).
\end{align}

Since $M_B^k(f) \le M^k(f\chi_B)$ in $\R^n$, we obtain
\begin{equation}\label{eq:Hedberg-2-I_2}
I_2 \le M_\beta( M^k(f\chi_B))(x).
\end{equation}

For $I_1$, the induction hypothesis yields
\begin{align}\label{eq:Hedberg-2-I_1}
I_1
&= M\left(a^{\frac{1}{q}}M_B^{k}(f)\chi_B\right)(x)\\
&\le M \left( cM_B^{k}(a^{\frac{1}{q}}f)\chi_B
+ cM_{\beta}(M^{k-1}(f\chi_B))\chi_B \right)(x) \\
&\le cM \left( M_B^{k}(a^{\frac{1}{q}}f)\chi_B \right)(x)
+ cM\left(M_{\beta}(M^{k-1}(f\chi_B) \right)(x) \\
&\le cM_B^{k+1} \left( a^{\frac{1}{q}}f \right)(x) 
+ cM_{\beta}\left(M^{k-1}(f\chi_B) \right)(x) \\
&\le  c(n,k,q,\alpha,[a]_{\alpha},\beta) \left( M_B^{k+1} \left( a^{\frac{1}{q}}f \right)(x) 
+ M_{\beta}\left(M^{k}(f\chi_B) \right)(x) \right).
\end{align}
In the second-to-last inequality, we used Remark \ref{rem:composite}, Lemma \ref{lem:composite} and $\beta \le \frac{n}{t}$.

Combining \eqref{eq:Hedberg-2-I's}, \eqref{eq:Hedberg-2-I_2} and \eqref{eq:Hedberg-2-I_1}, we obtain the conclusion.
\end{proof}

\subsection{Weighted mean value polynomial}\label{subsec:wmvp}
In this subsection, we fix $B=B_R$ and a function $0 \le \eta \in L^{\infty}(B) \setminus \{0\}$.
For $u \in W^{m,1}(B)$, there exists a unique polynomial $P=P(B,\eta,u)$ such that for each $\sigma \in S \setminus S_m$,
\begin{equation}\label{eq:wmep-P}
    \deg P \le m-1
    \quad
    \text{and}
    \quad
    (\partial_\sigma u)_{B,\eta} = (\partial_\sigma P)_{B,\eta}
    \quad
    \text{for each }
    \sigma \in S \setminus S_m.
\end{equation}
We refer to $P$ as the weighted mean value polynomial with respect to ($u,B,\eta$).

Here we take another ball $\mathcal{B}$ such that
\begin{equation}\label{eq:mathcalB}
c_1^{-1}R \le \mathrm{radius\,of}\,\mathcal{B} \le c_1R
\quad 
\text{for some }
 c_1=c_1(n)>0
\quad
\text{and} 
\quad
B \cap \mathcal{B} \neq \emptyset.
\end{equation}
We fix any $x_0 \in \mathcal{B}$.

For Lemmas \ref{lem:Pi} and \ref{lem:Pioscilation}, we fix some notation for multi-indices.
Let us write for $\sigma \in S$,
\begin{equation}
|\sigma|= \sigma_1+\cdots+\sigma_n,
\quad
\sigma ! = \sigma_1!\times \cdots \times\sigma_n !,
\quad
x^{\sigma} = x_1^{\sigma_1}\times \cdots \times x_n^{\sigma_n}.
\end{equation}
Moreover, for $\sigma, \tau \in S$, if $\tau \ge \sigma$, i.e., $\tau_i \ge \sigma_i$ for $i \in \{1,\cdots,n\}$, then we denote $\tau - \sigma \in S$ by
\begin{equation}
\tau - \sigma = \left(\tau_1-\sigma_1,\cdots, \tau_n-\sigma_n\right).
\end{equation}

For each $\sigma \in S \setminus S_m$, let $a_\sigma$ be the coefficient of $(x-x_0)^\sigma$ 
in the expansion of $P$ around the center $x_0$. 
The following properties are known in the case $\eta \equiv 1$ (see \cite[Lemma 3.4]{Bo}).
Although the difference is not essential, we provide a full proof of Lemma \ref{lem:Pi} for the sake of completeness.
\begin{lem}\label{lem:Pi}
The following statements hold:
\begin{description}
  \item[(i)]For each $\sigma \in S \setminus S_m$,
  \begin{align}
  a_{\sigma} 
  &= \frac{1}{\sigma!}\left\{(\partial_\sigma u)_{B,\eta} 
  - \sum_{\tau \in S \setminus S_m,\tau > \sigma}\frac{\tau!}{(\tau - \sigma)!}a_{\tau}\left((x - x_0)^{\tau - \sigma}\right)_{B,\eta} \right\} \\
  &= \frac{1}{\sigma!}\left(\partial_{\sigma}u - \sum_{\tau \in S \setminus S_m,\tau > \sigma}\frac{\tau!}{(\tau - \sigma)!}a_{\tau}(x - x_0)^{\tau-\sigma}\right)_{B,\eta}.
  \end{align}
  \item[(ii)] $|a_\sigma| \leq c(m,n)\sum_{\mu = \ell}^{m - 1} R^{\mu - \ell}|(D^{\mu}u)_{B,\eta}|$ 
  $\mathrm{for\ each}$ $\sigma \in S_\ell$ $\mathrm{and}$ $\ell \in \{0,\ldots,m-1\}$.
  \item[(iii)] $|D^\ell P| \leq c(m,n)\sum_{\mu = \ell}^{m - 1}R^{\mu - \ell}|(D^{\mu}u)_{B,\eta}|$ 
  $\mathrm{in}$ $B \cup \mathcal{B}$ $\mathrm{for\ each\ } \ell \in \{0,\ldots, m\}$.
\end{description}
\end{lem}

\begin{proof}
\textbf{(i)}:\,fix $\sigma \in S \setminus S_m $.
By \eqref{eq:wmep-P} and the definition of $a_{\sigma}$, we have
\begin{align}
    a_{\sigma} 
    &= (a_{\sigma})_{B,\eta} \\
    &= \frac{1}{\sigma!}\left(\partial_{\sigma}P - \sum_{\tau \in S \setminus S_m,\tau > \sigma}\frac{\tau!}{(\tau - \sigma)!}a_{\tau}(x - x_0)^{\tau-\sigma}\right)_{B,\eta} \\
    &= \frac{1}{\sigma!}\left\{(\partial_\sigma P)_{B,\eta} 
  - \sum_{\tau \in S \setminus S_m,\tau > \sigma}\frac{\tau!}{(\tau - \sigma)!}a_{\tau}\left((x - x_0)^{\tau - \sigma}\right)_{B,\eta} \right\} \\
   &= \frac{1}{\sigma!}\left\{(\partial_\sigma u)_{B,\eta} 
  - \sum_{\tau \in S \setminus S_m,\tau > \sigma}\frac{\tau!}{(\tau - \sigma)!}a_{\tau}\left((x - x_0)^{\tau - \sigma}\right)_{B,\eta} \right\} \\
    &= \frac{1}{\sigma!}\left(\partial_{\sigma}u - \sum_{\tau \in S \setminus S_m,\tau > \sigma}\frac{\tau!}{(\tau - \sigma)!}a_{\tau}(x - x_0)^{\tau-\sigma}\right)_{B,\eta}.
\end{align}
\noindent \textbf{(ii)}:\,argue by induction on $\ell$.
If $\ell = m - 1$, then 
we notice that
$a_\sigma = \frac{1}{\sigma!}\left(\partial_{\sigma} u \right)_{B,\eta}$ for each $\sigma \in S_{m - 1}$.
Then the claim is verified in this case.
Suppose that there exists $\ell_0 \in \{1,\ldots,m-1\}$ such that  the assertion holds for any $\ell \in \{\ell_0,\ldots,m-1\}$.
Then we will prove the assertion for $\ell = \ell_0 - 1$. 
Fix any $\sigma \in S_{\ell_0 - 1}$.
Firstly, \textbf{(i)} implies
\begin{equation}\label{eq:Pi-2-1}
   |a_\sigma| 
   \leq |(\partial_\sigma u)_{B,\eta}| 
  + \sum_{\tau \in S \setminus S_m,\tau > \sigma}\frac{\tau!}{(\tau - \sigma)!}|a_{\tau}|\left|\left((x - x_0)^{\tau - \sigma}\right)_{B,\eta}\right|.
\end{equation}
Moreover, from \eqref{eq:mathcalB} and the induction hypothesis, it follows that
\begin{align}\label{eq:Pi-2-2}
    &\sum_{\tau \in S \setminus S_m,\tau > \sigma}\frac{\tau!}{(\tau - \sigma)!}|a_{\tau}|\left|\left((x - x_0)^{\tau - \sigma}\right)_{B,\eta}\right| \\
    &= \sum_{\tau \in S \setminus S_m,\tau > \sigma}\frac{\tau!}{(\tau - \sigma)!}|a_{\tau}|\left|\frac{1}{\|\eta\|_{L^1(B)}}\int_{B}(x - x_0)^{\tau - \sigma}\eta \right| \\
    &\le \sum_{\tau \in S \setminus S_m,\tau > \sigma}\frac{\tau!}{(\tau - \sigma)!}|a_{\tau}|\frac{1}{\|\eta\|_{L^1(B)}}\int_{B}|x - x_0|^{|\tau| - |\sigma|}|\eta|  \\
    &\le c\sum_{\tau \in S \setminus S_m,\tau > \sigma}|a_{\tau}|R^{|\tau| - |\sigma|} \\
    &\leq c\sum_{\tau \in S \setminus S_m,\tau > \sigma}\sum_{\mu = |\tau|}^{m - 1} R^{\mu - |\tau|}|(D^{\mu}u)_{B,\eta}|\cdot R^{|\tau| - |\sigma|} 
    \leq c(m,n)\sum_{\mu = \ell_0}^{m - 1} R^{\mu - \ell_0 + 1}|(D^{\mu}u)_{B,\eta}|.
\end{align}
Combining \eqref{eq:Pi-2-1} and \eqref{eq:Pi-2-2}, the assertion holds for $\ell_0 - 1$.

\noindent \textbf{(iii)}:\,
the case $\ell = m$ is obvious because $D^mP \equiv 0$.
We fix $\ell \in \{0,\ldots,m - 1\}$ and $\sigma \in S_\ell$.
Using \textbf{(ii)} and \eqref{eq:mathcalB},
we have in $B \cup \mathcal{B}$,
\begin{align}
    |\partial_\sigma P|
    &= \left|\sum_{\tau \in S \setminus S_m,\tau \geq \sigma}\frac{\tau !}{(\tau - \sigma)!}a_{\tau}(x - x_0)^{\tau - \sigma}\right| \\
    &\leq \sum_{\tau \in S \setminus S_m,\tau \geq \sigma}\frac{\tau !}{(\tau - \sigma)!}|a_{\tau}||x - x_0|^{|\tau| - |\sigma|} \\
    &\leq c\sum_{\tau \in S \setminus S_m,\tau \geq \sigma}\sum_{\mu = |\tau|}^{m - 1}R^{\mu - |\tau|}|(D^{\mu}u)_{B,\eta}|\cdot (c_1R)^{|\tau| - |\sigma|} 
    \leq c\sum_{\mu = \ell}^{m - 1} R^{\mu - \ell}|(D^{\mu}u)_{B,\eta}|.
\end{align}
The proof is completed.
\end{proof}

\begin{rem}\label{rem:integration by parts}
In addition, we assume 
\begin{equation}\label{eq:wmvp-2}
    \eta \in C_c^{\infty}(B),\,
    \|\eta\|_{L^1(B)} \ge \kappa|B|,\,
    \ |D^\ell \eta| \leq c(m,n)R^{-\ell}
    \,
    \text{for each }
    \ell \in \{0,\ldots,m\}
\end{equation} 
for some $\kappa = \kappa(n) > 0$.
Then, using Lemma \ref{lem:Pi} \textbf{(iii)} and integration by parts, we get
\begin{align}\label{eq:integration by parts}
|D^\ell P| 
&\leq c(m,n)\sum_{\mu = \ell}^{m - 1}R^{\mu - \ell}|(D^{\mu}u)_{B,\eta}| \\
&\le  c(m,n)\sum_{\mu = \ell}^{m - 1}R^{\mu - \ell}|(D^{\mu}u \cdot \eta)_{B}| \\
&\le c(m,n)\sum_{\mu = \ell}^{m - 1}R^{\mu - \ell}\fint_{B}|D^{\ell}u||D^{\mu-\ell} \eta| 
\le c(m,n)\fint_{B}|D^\ell u|
\quad
\text{in }
B \cup \mathcal{B}
\end{align}
for each $\ell \in \{0,\ldots,m\}$.
That is, the $\ell$-th derivative of $P$ is estimated by the $\ell$-th derivative of $u$.
As mentioned in Subsection~\ref{subsec:outline}, this estimate is the primary reason for adopting the weighted mean value polynomial over the standard version.
\end{rem}

The next lemma is needed in Subsection 4.3 and Subsection 4.4.
\begin{lem}\label{lem:useful-formula}
Let $u \in W^{m,t}(B)\,(t \in (1,\infty])$. Moreover, assume
\begin{equation}
1 \ge \eta \in C_c^{\infty}(B),
\quad
\|\eta\|_{L^1(B)} \ge \kappa|B|,
\quad|D^\ell \eta| \leq c(m,n)R^{-\ell}
\quad
\text{for each }\ell \in \{0,\ldots,m\}
\end{equation} 
for some $\kappa = \kappa(n) > 0$.
Then there exists $c=c(n,m)>0$ such that 
\begin{equation}
\sum_{k = 0}^{\ell}\left| \frac{D^k u(x) - D^k P(x)}{R^{\ell - k}} \right|
\le cM_{B}^{\ell+1}\left(|D^\ell u|\right)(x)
\end{equation}
for a.e.\,$x \in B$ and each $\ell \in \{0,\ldots,m\}$. Here 
\begin{equation}
M_B(u) = M(u \chi_B),
\quad
M_B^\ell(u) = M_B^{\ell - 1}(M_B u)
\quad
\text{for }
\ell \ge 2.
\end{equation}
\end{lem}

\begin{proof}
Fix any $\ell \in \{0,\ldots,m\}$ and $k \in \{0,\ldots,\ell\}$. 
By Lemma \ref{lem:Hedberg-1} and \eqref{eq:wmep-P}, we have
\begin{equation}
\left| \frac{D^k u(x) - D^k P(x)}{R^{\ell - k}} \right|
\le cM_{B}^{\ell}\left(|D^\ell u - D^\ell P|\right)(x)
\quad
\text{for a.e.\,}
x \in B.
\end{equation}
We fix $x$ such that the above inequality holds.
Since $M_{B_{2R}}$ is sublinear, we get
\begin{equation}
M_{B}^{\ell}\left(|D^\ell u - D^\ell P|\right)(x)
\le M_{B}^{\ell}\left(|D^\ell u|\right)(x) + M_{B}^{\ell}\left(|D^\ell P|\right)(x).
\end{equation}
Moreover, \eqref{eq:integration by parts} implies
\begin{equation}
|D^\ell P|(x)
\le c\fint_{B} |D^\ell u|(y)\chi_B(y)\,dy 
\le cM_{B}(|D^\ell u|)(x).
\end{equation}
Therefore, it follows that
\begin{equation}
\left| \frac{D^k u(x) - D^k P(x)}{R^{\ell - k}} \right| \le cM_{B}^{\ell+1}\left(|D^\ell u|\right)(x).
\end{equation}
The proof is completed.
\end{proof}

\subsection{Gehring's lemma}\label{subsec:Ge}
Finally, we shall need a refined form of a theorem of Gehring \cite{Ge}.
\begin{lem}\label{lem:Gehring}
Let $\Omega \subset \mathbb{R}^n$ be an open set, and let $f_1 \in L^1_{loc}(\Omega)$ and $f_2 \in L^{1 + \varepsilon_0}_{loc}(\Omega)$ for some $\varepsilon_0 > 0$, where $f_1$ and $f_2$ are nonnegative.
In addition, assume there exist $\kappa \in (0,1),\,A> 0$, $R_0>0$ and $\theta \in (0,1)$ such that
\begin{equation}\label{eq:lem-Gehring-reverse-Holder}
    \fint_{B_R} f_1 \leq A \left( \fint_{B_{2R}} f_1^\kappa \right)^\frac{1}{\kappa} 
    + \fint_{B_{2R}} f_2  + \theta \fint_{B_{2R}} f_1,
\end{equation}
whenever $B_{2R} \subset \subset \Omega$ and $R \le R_0$.
Then there exists $c^* = c^*(n,A,\varepsilon_0,\theta)>0$ such that 
if $ 0 < \varepsilon \leq \min \left\{ \frac{1-\kappa}{c^*}, \varepsilon_0 \right\}$, then
$f_1 \in L^{1 + \varepsilon}_{loc}(\Omega)$. 
Moreover, there exists $c = c(n,A,\varepsilon_0,\theta)$ such that for any $0 < \varepsilon \leq \min \left\{ \frac{1-\kappa}{c^*}, \varepsilon_0 \right\}$, it holds
\begin{equation}
    \left( \fint_{B_{R}} f_1^{1 + \varepsilon} \right)^\frac{1}{1 + \varepsilon} 
    \leq c\fint_{B_{2R}} f_1 
    +    c\left( \fint_{B_{2R}} f_2^{1 + \varepsilon} \right)^\frac{1}{1 + \varepsilon} 
\end{equation}
whenever $B_{2R} \subset \subset \Omega$ and $R \le R_0$. Furthermore, the constant $c^*$ can be chosen to be monotone increasing with respect to $\varepsilon_0$.
\end{lem}
Although the validity of this lemma has already been established (see \cite[Section 6.4]{Gi} for example), we will give a new proof, which is based on some covering argument (see \cite{KL, CoM2}) in the Appendix. The proof seems elementary in comparison with the available ones.

\section{Sobolev--Poincar\'e inequality for the double phase operator}\label{sec:SP}

The purpose of this section is to establish a Sobolev--Poincar\'e inequality associated with the double phase operator;\,we also provide the proof of Theorem~\ref{thm:sp-0}.
To this end, we first prove the result for the case $q < n$ (Theorem \ref{thm:sp}). 

For the rest of this section, we fix $B=B_R$. Moreover, for $\gamma \in (0,n)$ and $f \in L^1(B;\R^k)\,(k \in \N)$, we set
\begin{equation}
I_{\gamma}(f)(x) \coloneq \int_{B} \frac{|f(y)|}{|x-y|^{n-\gamma}}\, dy.
\end{equation}
$I_{\gamma}$ is the Riesz potential restricted on $B$. For the basic properties of the Riesz potential, see \cite[Chapter 1]{KLV} for example.

We begin with the following auxiliary lemma.
\begin{lem}\label{lem:beta}
Let $1 \le p \le q < n$ and set 
\begin{equation}\label{eq:beta}
\beta \coloneqq n\left(\frac{1}{p}-\frac{1}{q}\right)+1.
\end{equation}
Then the following assertions hold:
\begin{description} 
\item[(i)] $1 \le \beta< \frac{n}{p}$. Moreover, $\frac{np}{n-\beta p} = \frac{nq}{n-q}$.
\item[(ii)] $1+\frac{\alpha}{q}-\beta = \frac{n}{q}\left(1+\frac{\alpha}{n}-\frac{q}{p} \right)$ for $\alpha \in \R$.
\item[(iii)]Let $\alpha \in (0,\infty)$, $a \in \mathcal{Z}^{\alpha}(B)$ and $\frac{q}{p} \le 1+\frac{\alpha}{n}$. Then there exists $c=c(n,p,q,\alpha, [a]_{\alpha})>0$ such that
\begin{equation}
a(x)^{\frac{1}{q}}I_{1}(f)(x)
\le cI_1(a^{\frac{1}{q}}f)(x)
+ cR^{1+\frac{\alpha}{q}-\beta}I_{\beta}(f)(x)
\end{equation}
for any $x \in B$ and $f \in L^1(B)$.
\end{description}
\end{lem}

\begin{proof}
We only prove \textbf{(iii)}.
Fix any $x \in B$ and $f \in L^1(B)$. By the definition of $\mathcal{Z}^{\alpha}(B)$ (see Definition \ref{def:Zalpha}), we have
\begin{align}\label{eq:beta-3-1}
&a(x)^{\frac{1}{q}}I_{1}(f)(x) \\
&= \int_{B} \frac{a(x)^{\frac{1}{q}}|f(y)|}{|x-y|^{n-1}}\ dy \\
&\le [a]_{\alpha}^{\frac{1}{q}}\int_{B} \frac{\left( a(y) + |x-y|^{\alpha} \right)^{\frac{1}{q}} |f(y)|}{|x-y|^{n-1}}\ dy \\
&\le [a]_{\alpha}^{\frac{1}{q}}\int_{B} \frac{a(y)^{\frac{1}{q}}|f(y)|}{|x-y|^{n-1}}\ dy
+ [a]_{\alpha}^{\frac{1}{q}}\int_{B}\frac{|x-y|^{\frac{\alpha}{q}}|f(y)|}{|x-y|^{n-1}}\ dy
\eqqcolon  [a]_{\alpha}^{\frac{1}{q}}(J_1 + J_2).
\end{align}

For $J_1$, we notice that
\begin{equation}\label{eq:beta-3-2}
J_1 = I_1(a^{\frac{1}{q}}f)(x).
\end{equation}

For $J_2$, since $1+\frac{\alpha}{q}-\beta \ge 0$ by \textbf{(ii)} and $\frac{q}{p} \le 1 + \frac{\alpha}{n}$, we derive
\begin{align}\label{eq:beta-3-3}
J_2 
&= \int_{B} \frac{|x-y|^{1+\frac{\alpha}{q}-\beta}|f(y)|}{|x-y|^{n-\beta}}\ dy \\
&\le  (2R)^{1+\frac{\alpha}{q}-\beta}\int_{B} \frac{|f(y)|}{|x-y|^{n-\beta}}\ dy
= (2R)^{1+\frac{\alpha}{q}-\beta}I_{\beta}(f)(x).
\end{align}

From \eqref{eq:beta-3-1}, \eqref{eq:beta-3-2} and \eqref{eq:beta-3-3}, we obtain the desired inequality.
\end{proof}

Here we recall the strong-type estimate of the Riesz potential. This lemma immediately follows from the whole space version \cite[Theorem 1.36]{KLV}.
\begin{lem}\label{lem:strong-type-c}
Let $1 < r < \infty$ and $f \in L^r(B)$. Moreover assume $\gamma \in \left(0,\frac{n}{r}\right)$. Then there exists $c=c(n,r,\gamma) >0$ such that
\begin{equation}
\|I_{\gamma}(f)\|_{L^{\frac{nr}{n-\gamma r}}(B)}
\le c(n,r,\gamma)\|f\|_{L^r(B)}.
\end{equation}
\end{lem}

For convenience, we introduce the following two assumptions: 
\begin{equation}\label{assmpt:sp-1}
1 < p \le q < \infty,\quad
\alpha \in (0,\infty),\quad
a \in \mathcal{Z}^{\alpha}(B),\quad
\frac{q}{p} \le 1 + \frac{\alpha}{n}
\end{equation}
and
\begin{equation}\label{assmpt:sp-2}
\eta \in L^{\infty}(B),
\quad
0 \le \eta \le 1,
\quad
\|\eta\|_{L^1(B)} \ge \kappa|B|
\end{equation}
for some $\kappa = \kappa(n) > 0$.
The function $\eta$ will serve as a weight in what follows. 

Thanks to Lemma \ref{lem:beta}, we prove the following strong type estimate:
\begin{lem}\label{lem:strong-type}
Assume \eqref{assmpt:sp-1} and $q < n$. Moreover suppose $f \in L^p(B)$ satisfies $\int_{B} a|f|^q  < \infty$.
Then there exists $c=c(n,p,q,\alpha,[a]_{\alpha})$ such that
\begin{equation}
\|a^{\frac{1}{q}}I_{1}(f)\|_{L^{q^{*}}(B)} \le c\|a^{\frac{1}{q}}f\|_{L^q(B)} + cR^{1+\frac{\alpha}{q}-\beta}\|f\|_{L^p(B)}.
\end{equation}
Here $q^* = \frac{nq}{n-q}$ and $\beta$ is the number defined in \eqref{eq:beta}.
\end{lem}

\begin{proof}
By Lemma \ref{lem:beta} \textbf{(iii)}, we have
\begin{equation}\label{eq:stong-type-0}
\|a^{\frac{1}{q}}I_{1}(f)\|_{L^{q^{*}}(B)}
\le c\|I_{1}(a^{\frac{1}{q}}f)\|_{L^{q^{*}}(B)} 
+ cR^{1+\frac{\alpha}{q}-\beta}\|I_{\beta}f\|_{L^{q^{*}}(B)},
\end{equation}
where $c=c(n,p,q,\alpha,[a]_{\alpha})$.
Since $a^{\frac{1}{q}}f \in L^q(B)$, it follows that
\begin{equation}\label{eq:stong-type-1}
\|I_{1}(a^{\frac{1}{q}}f)\|_{L^{q^{*}}(B)}
\le c(n,q)\|a^{\frac{1}{q}}f\|_{L^{q}(B)}
\end{equation}
by Lemma \ref{lem:strong-type-c}.
Moreover, using Lemma \ref{lem:beta} \textbf{(i)} and Lemma \ref{lem:strong-type-c}, we have
\begin{equation}\label{eq:stong-type-2}
\|I_{\beta}f\|_{L^{q^{*}}(B)} 
= \|I_{\beta}f\|_{L^{\frac{np}{n-p\beta}}(B)} 
\le c(n,p,q)\|f\|_{L^{p}(B)}.
\end{equation}
Combining \eqref{eq:stong-type-0}, \eqref{eq:stong-type-1} and \eqref{eq:stong-type-2}, we obtain the desired inequality.
\end{proof}

Now we state the Sobolev--Poincar\'e inequality for the double phase operator for the case $q < n$. A version of this result in a more general setting can already be found in \cite{CD}. 
However, we emphasize that the dependence on the size of the domain is explicit and the proof is elementary\,(see also Subsection~\ref{subsec:outline}).
\begin{thm}\label{thm:sp}
Assume \eqref{assmpt:sp-1},\,\eqref{assmpt:sp-2} and $q < n$. Moreover, suppose $u \in W^{1,p}(B)$ satisfies 
\begin{equation}
\int_{B} a|Du|^q  < \infty
\quad
\text{and}
\quad
u_{B,\eta} = 0.
\end{equation}
Then there exists $c=c(n,p,q,\alpha, [a]_{\alpha})>0$ such that
\begin{equation}
\left( \fint_{B} a^{\frac{q^*}{q}}\left| \frac{u}{R} \right|^{q*} \right)^{\frac{1}{q^*}} 
\le c\left( \fint_{B} a|Du|^{q} \right)^{\frac{1}{q}}
+ cR^{\frac{\alpha}{q}} \left( \fint_{B} |Du|^{p} \right)^{\frac{1}{p}}.
\end{equation}
Here $q^* = \frac{nq}{n-q}$.
\end{thm}
\begin{rem}\label{rem:sp-p}
Note that $p > 1$. This condition necessitates an indirect approach in the proof of Lemma \ref{lem:SP}. 
\end{rem}

\begin{proof}
We first prove
\begin{equation}\label{eq:sp-1}
|u(x)| \le c(n)I_1(Du)(x)
\quad
\text{for a.e.\,}
x \in B.
\end{equation}
Using $u_{B,\eta}=0$, 
we have $|u(x)| \le |u(x)-u_B| + |u_B-u_{B,\eta}|$.
\cite[Lemma 7.16]{GT} implies $|u(x)-u_B| \le c(n)I_1(|Du|)(x)$ for a.e.\,$x \in B$.
On the other hand, from \eqref{assmpt:sp-2} and Poincar\'{e}'s inequality, it follows that 
for any $x \in B$
\begin{align}
|u_B-u_{B,\eta}|
&\le \left|\frac{1}{\|\eta\|_{L^1(B)}}\int_B (u_B - u(y))\eta \,dy \right| \\
&\le c(n)\fint_{B}|u-u_B| 
\le c(n)R\fint_{B}|Du|
\le c(n)I_1(Du)(x).
\end{align}
Therefore, we obtain \eqref{eq:sp-1}.

Lemma \ref{lem:strong-type} and \eqref{eq:sp-1} imply 
\begin{align}\label{eq:sp-2}
&\|a^{\frac{1}{q}}u\|_{L^{q^{*}}(B)} \\
&\le c(n)\|a^{\frac{1}{q}}I_{1}(Du)\|_{L^{q^{*}}(B)}
\le c\|a^{\frac{1}{q}}Du\|_{L^q(B)} + cR^{1+\frac{\alpha}{q}-\beta}\|Du\|_{L^p(B)},
\end{align}
where $c=c(n,p,q,\alpha,[a]_{\alpha})$.

Finally, multiplying both sides of \eqref{eq:sp-2} by $ \frac{1}{R}\left( \frac{1}{|B|} \right)^{\frac{1}{q^*}}$, we obtain
\begin{align}\label{eq:sp-3}
\left( \fint_B a^{\frac{q^*}{q}}\left|\frac{u}{R}\right|^{q*} \right)^{\frac{1}{q^*}} 
&\le c\omega_n^{\frac{1}{n}}\left(\fint_{B} a|Du|^{q} \right)^{\frac{1}{q}} 
+ c\omega_n^{\frac{1}{p}-\frac{1}{q^*}}R^{\frac{\alpha}{q}}\left( \fint_{B} |Du|^{p} \right)^{\frac{1}{p}}.
\end{align}
Here we used
\begin{equation}
\left( \frac{1}{|B|} \right)^{\frac{1}{q^*}-\frac{1}{p}}R^{1+\frac{\alpha}{q}-\beta}
= (\omega_nR^n)^{\frac{1}{p}-\frac{1}{q^*}}R^{\frac{\alpha}{q}-n\left(\frac{1}{p}-\frac{1}{q}\right)}
=(\omega_n)^{\frac{1}{p}-\frac{1}{q^*}}R^{1+\frac{\alpha}{q}}.
\end{equation}
Estimate \eqref{eq:sp-3} is the desired one and the proof is completed.
\end{proof}

\begin{proof}[Proof of Theorem \ref{thm:sp-0}]
The conclusion follows from applying Theorem \ref{thm:sp} with $\eta \equiv 1$.
\end{proof}

Here we established the Sobolev--Poincar\'e inequality for the case $q < n$. 
Next, we turn to the proof of the general Sobolev--Poincar\'e inequality, which is the goal of this section.
Before entering the proof, we present the following lemma, which will also be used in Subsection~4.1.

\begin{lem}\label{lem:exponents-0}
Let $1 \le p \le q < \infty$ and $\alpha \in (0,\infty)$. If $\frac{q}{p} \le(\text{resp.}<)\,1 + \frac{\alpha}{n}$, then it holds
\begin{equation}\label{eq:exponents-0} 
\frac{\alpha}{q} - n\left( \frac{1}{(p_{\ell})^{*}} - \frac{1}{(q_{\ell})^{*}} \right) \ge (\text{resp.}>)\,0
\end{equation}
for each $\ell \in \{0,\ldots,m\}$. Here for $r \in \{p,q\}$,
\begin{equation}\label{eq:pell*-2} 
  (r_{\ell})^* \coloneqq
  \left\{ 
  \begin{alignedat}{2}   
    \frac{nr}{n - \ell r} \hspace{12pt} &\text{if}\ \ell r < n, \\
    \infty                        \hspace{25pt} &\text{otherwise}.
  \end{alignedat} 
  \right. 
\end{equation}
In particular, we denote $(r_1)^* = r^*$.
\end{lem}

\begin{proof}
Firstly, let us consider the case $\ell q < n$. Then it follows
\begin{align}\label{eq:1}
&\frac{\alpha}{q} - n\left( \frac{1}{(p_{\ell})^{*}} - \frac{1}{(q_{\ell})^{*}} \right) \\
&=\frac{\alpha}{q} - n\left(\frac{n - \ell p}{np}  -  \frac{n - \ell q}{nq}\right) \\
&=\frac{\alpha}{q} - n\left( \frac{1}{p} - \frac{1}{q} \right)
 =\frac{n}{q}\left(1+\frac{\alpha}{n}-\frac{q}{p}\right) \ge (\text{resp.}>) \,0.
\end{align}

Next, we will consider the case $\ell p < n$ and $\ell q \ge n$.
Then we have
\begin{align}
&\frac{\alpha}{q} - \frac{n}{(p_{\ell})^{*}} \\
&= \frac{\alpha}{q} - \frac{n-\ell p}{p} \\
&= \frac{\alpha}{q} + \ell - \frac{n}{p} \\
&= \frac{n}{q}\left(\frac{\ell q}{n}+\frac{\alpha}{n}-\frac{q}{p}\right) 
\ge \frac{n}{q} \left( 1+\frac{\alpha}{n}-\frac{q}{p} \right) \ge (\text{resp.}>)\,0
\end{align}
by $\ell q \ge n$. If $\ell p \ge n$, \eqref{eq:exponents-0} is obvious. The proof is completed.
\end{proof}

Using Theorem \ref{thm:sp} and Lemma \ref{lem:exponents-0}, we obtain the following corollary:
\begin{cor}\label{cor:sp}
Assume \eqref{assmpt:sp-1} and \eqref{assmpt:sp-2}.
Moreover suppose $\ell \in \N$ and $u \in W^{\ell,p}(B)$ satisfies
\begin{equation}
\int_B a|D^\ell u|^q < \infty
\quad
\text{and}
\quad
(D^k u)_{B,\eta} = 0
\quad
\text{for each}
\quad
k \in \{0,\ldots,\ell-1\}.
\end{equation}
If $\ell q < n$, then there exists $c=c(n,\ell,p,q,r,\alpha,[a]_{\alpha})$ such that
\begin{equation}\label{eq:cor-sp-c}
\left(\fint_{B} a^{\frac{r}{q}}\left|\frac{u}{R^\ell}\right|^r \right)^{\frac{1}{r}}
\le c\left(\fint_{B} a|D^\ell u|^q \right)^{\frac{1}{q}} 
+ cR^{\frac{\alpha}{q}}\left(\fint_{B} |D^\ell u|^p \right)^{\frac{1}{p}}
\end{equation}
for any $r \in [1,(q_{\ell})^*]$. If $\ell q \ge n$, then the same conclusion holds for any $r \in [1,(q_\ell)^*)$. For the definition of $(q_\ell)^*$, see \eqref{eq:pell*-2}.
In particular, if $\ell = 1$ and $r=q$, it holds that
\begin{equation}\label{eq:cor-sp-p}
\left(\fint_{B} a\left|\frac{u}{R}\right|^q \right)^{\frac{1}{q}}
\le c\left(\fint_{B} a|Du|^q \right)^{\frac{1}{q}} 
+ cR^{\frac{\alpha}{q}}\left(\fint_{B} |Du|^p \right)^{\frac{1}{p}}.
\end{equation}
\end{cor}

\begin{proof}
We first suppose $\ell q < n$. It suffices to show the case $r = (q_{\ell})^*$. This is because the other case follows from Jensen's inequality. Argue by induction on $\ell$.
The case $\ell = 1$ coincides with Theorem \ref{thm:sp}.
Assume $\ell q < n$ for some $\ell \ge 2$ and the conclusion \eqref{eq:cor-sp-c} holds for $\ell-1$. Since we have $(\ell - 1)q < n$, the induction hypothesis yields
\begin{equation}\label{eq:cor-sp-1}
\left(\fint_{B} a^{\frac{(q_{\ell-1})^*}{q}}\left|\frac{Du}{R^{\ell-1}}\right|^{(q_{\ell-1})^*} \right)^{\frac{1}{(q_{\ell-1})^*}}
\le c\left(\fint_{B} a|D^\ell u|^q \right)^{\frac{1}{q}} 
+ cR^{\frac{\alpha}{q}}\left(\fint_{B} |D^\ell u|^p \right)^{\frac{1}{p}},
\end{equation}
where $c = c(n,\ell,p,q,\alpha,[a]_{\alpha})$.
On the other hand, from $\ell q < n$ and Lemma \ref{lem:exponents-0}, it follows that
\begin{equation}
(q_{\ell-1})^* < n,
\quad
\frac{(q_{\ell-1})^*}{(p_{\ell-1})^*} \le 1 + \frac{(q_{\ell-1})^*\alpha}{nq}.
\end{equation}
Moreover, Sobolev's embedding theorem, our assumption and \eqref{eq:cor-sp-1} yield
\begin{equation}
 u \in W^{1,(p_{\ell-1})^*}(B),\quad
 u_{B,\eta}= 0,\quad
 \int_B a^{\frac{(q_{\ell-1})^*}{q}}|Du|^{(q_{\ell-1})^*}  < \infty.
\end{equation}
Therefore, applying Theorem \ref{thm:sp} with 
$$
(p,q,a,\alpha) = \left( (p_{\ell-1})^*,(q_{\ell-1})^*,a^{\frac{(q_{\ell-1})^*}{q}},\frac{\alpha(q_{\ell-1})^*}{q} \right)
$$ 
and the classical Poincar\'e inequality, we deduce
\begin{align}\label{eq:cor-sp-2}
&\left(\fint_{B} a^{\frac{(q_\ell)^*}{q}}\left|\frac{u}{R^\ell}\right|^{(q_\ell)^*} \right)^{\frac{1}{(q_\ell)^*}} \\
&= \left(\fint_{B} [a^{\frac{(q_{\ell-1})^*}{q}}]^{\frac{((q_{\ell-1})^*)^*}{(q_{\ell-1})^*}} \left|\frac{u}{R^\ell}\right|^{((q_{\ell-1})^*)^*} \right)^{\frac{1}{((q_{\ell-1})^*)^*}} \\
&\le c\left(\fint_{B} a^{\frac{(q_{\ell-1})^*}{q}} \left|\frac{Du}{R^{\ell-1}}\right|^{(q_{\ell-1})^*} \right)^{\frac{1}{(q_{\ell-1})^*}}
+ cR^{\frac{\alpha}{q}}\left(\fint_{B} \left|\frac{Du}{R^{\ell-1}} \right|^{(p_{\ell-1})^*} \right)^{\frac{1}{(p_{\ell-1})^*}} \\
&\le c\left(\fint_{B} a^{\frac{(q_{\ell-1})^*}{q}} \left|\frac{Du}{R^{\ell-1}}\right|^{(q_{\ell-1})^*} \right)^{\frac{1}{(q_{\ell-1})^*}}
+ cR^{\frac{\alpha}{q}}\left(\fint_{B} |D^{\ell}u|^{p} \right)^{\frac{1}{p}},
\end{align}
where $c=c(n,p,q,\alpha,[a]_{\alpha})$.
Here we also used
\begin{equation}
( (q_{\ell-1})^{*} )^* = (q_{\ell})^*
\quad
\text{and}
\quad
[a^{\frac{(q_{\ell-1})^*}{q}}]_{\frac{\alpha(q_{\ell-1})^*}{q}} \le c(n,\ell,q) [a]_{\alpha}^{\frac{(q_{\ell-1})^*}{q}}.
\end{equation}
Combining \eqref{eq:cor-sp-1} and \eqref{eq:cor-sp-2}, we obtain the desired inequality for the case $\ell q < n$.

Next, let us consider the case $\ell q \ge n$. Fix $r \in [1,\infty)$.
Suppose first $\ell p \ge n$ or $r \le (p_{\ell})^{*}$. 
Observe that there exists $c=c(q,\alpha,[a]_{\alpha})$ such that
\begin{equation}
a^{\frac{1}{q}} \le ca_m^{\frac{1}{q}} + cR^{\frac{\alpha}{q}}
\quad
\text{in }
B=B_R,
\end{equation}
where $a_m = \inf_{y \in B}a(y)$ since $a \in \mathcal{Z}^{\alpha}(B)$.
Then, Minkowski's inequality, Jensen's inequality and the standard Sobolev--Poincar\'e inequality yield that
\begin{align}
\left(\fint_{B} a^{\frac{r}{q}}\left|\frac{u}{R^\ell}\right|^r \right)^{\frac{1}{r}}
&\le c\left(\fint_{B} a_m^{\frac{r}{q}}\left|\frac{u}{R^\ell}\right|^r \right)^{\frac{1}{r}}
+ cR^{\frac{\alpha}{q}}\left(\fint_{B} \left|\frac{u}{R^\ell}\right|^r \right)^{\frac{1}{r}} \\
&\le c\left(\fint_{B} a_m^{\frac{p}{q}}|D^\ell u|^p \right)^{\frac{1}{p}}
+ cR^{\frac{\alpha}{q}}\left(\fint_{B} |D^\ell u|^p \right)^{\frac{1}{p}} \\
&\le c\left(\fint_{B} a|D^\ell u|^q \right)^{\frac{1}{q}}
+ cR^{\frac{\alpha}{q}}\left(\fint_{B} |D^\ell u|^p \right)^{\frac{1}{p}},
\end{align}
where $c=c(n,\ell,p,q,r,\alpha,[a]_{\alpha})$. Therefore, the conclusion is verified in this case.

Finally, we suppose $(p_{\ell})^{*} < r$ under the condition $\ell q \ge n$. Note that $p < q$ and we can take $s \in (p,q)$ such that $r = (s_{\ell})^{*}$. Moreover, we have
\begin{equation}
\ell s < n,
\quad
\frac{s}{p} < 1 + \frac{\alpha s}{nq}
\quad
\text{and}
\quad
\int_{B} \left( |D^\ell u|^p + a^{\frac{s}{q}}|D^\ell u|^s \right) < \infty.
\end{equation}
Therefore, applying \eqref{eq:cor-sp-c} with $(p,q,a,\alpha) = \left( p,s,a^{\frac{s}{q}},\frac{\alpha s}{q} \right)$ and Jensen's inequality, we obtain a constant $c=c(n,\ell,p,q,r,\alpha,[a]_{\alpha})$ such that
\begin{align}
\left(\fint_{B} a^{\frac{r}{q}}\left|\frac{u}{R^\ell}\right|^r \right)^{\frac{1}{r}}
&= \left(\fint_{B} [a^{\frac{s}{q}}]^{\frac{(s_{\ell})^{*}}{s}}\left|\frac{u}{R^\ell}\right|^{(s_{\ell})^{*}} \right)^{\frac{1}{(s_{\ell})^{*}}} \\
&\le c\left(\fint_{B} a^{\frac{s}{q}}|D^\ell u|^s \right)^{\frac{1}{s}} 
+ cR^{\frac{\alpha}{q}}\left(\fint_{B} |D^\ell u|^p \right)^{\frac{1}{p}} \\
&\le  c\left(\fint_{B} a|D^\ell u|^q \right)^{\frac{1}{q}} 
+ cR^{\frac{\alpha}{q}}\left(\fint_{B} |D^\ell u|^p \right)^{\frac{1}{p}}.
\end{align}
Here we also used $r = (s_{\ell})^{*}$ and $[a^{\frac{s}{q}}]_{\frac{\alpha s}{q}} \le [a]_{\alpha}^{\frac{s}{q}}$.
The proof is completed.
\end{proof}

\section{Proof of Theorem\,\ref{thm:main} and Corollary\,\ref{cor:main}}\label{sec:pf-of-main}
In this section, we prove Theorem\,\ref{thm:main} and Corollary\,\ref{cor:main}.
For convenience, we abbreviate the expression for vector-valued functions, e.g., we denote $L^1(\Omega;\R^N)$ by $L^1(\Omega)$.
Moreover, we also omit the dependence on the dimension $N$.
We make the same assumptions as in Theorem \ref{thm:main}, that is, we assume \eqref{assmpt:exponents}, \eqref{eq:coercivity}, \eqref{eq:growth}, \eqref{assumpt:functions} and \eqref{def:exponents} in the following. 

We divide the proof into five subsections:\,in Subsection~4.1, we set the exponents, which are needed for the definition of an appropriate truncated function. 
Next, in Subsection~4.2, we construct the truncated function using a Whitney-type covering lemma. 
In Subsection~4.3, we provide estimates for the derivatives of the truncated function and verify its admissibility as a test function for \eqref{eq:main}.
After the above preparation, we deduce a Caccioppoli inequality and reverse H\"older inequality in Subsection~4.4.
Finally, in Subsection~4.5, we derive the conclusion, by applying Gehring's lemma. 

\subsection{Setting of exponents}
For each $\ell \in \{0,\ldots,m\}$ and $r \in \{p,q\}$, we define $\widehat{s_{r,\ell}},\,\widehat{t_{r,\ell}},\,\gamma_{r,\ell}$ and $\delta_0$. 
Recall \eqref{eq:Holder conjugate}, \eqref{eq:pell*}, \eqref{assumpt:functions} and \eqref{def:exponents}.
Firstly, we set
\begin{equation}\label{eq:hats-m}
    \widehat{s_{r,m}} \coloneqq \infty,
    \quad
    \widehat{t_{r,m}} \coloneqq r'\ (< t_{r,m}),
    \quad
    \gamma_{r,m} \coloneqq r.
\end{equation}

Next, let us consider the case $\ell \in \{0,\ldots,m-1\}$. Then, by \eqref{def:exponents}, \eqref{assmpt:exponents} and Lemma~\ref{lem:exponents-0}, we have
\begin{gather}
   \frac{1}{s_{r,\ell}} < 1 - \frac{1}{(r_{m-\ell})^*} - \frac{1}{r'},
   \quad
   \frac{1}{t_{r,\ell}} < 1 - \frac{1}{(r_{m-\ell})^*},
   \quad
   \frac{\alpha}{q} - n\left( \frac{1}{(p_{m-\ell})^*} - \frac{1}{(q_{m-\ell})^*} \right) > 0.
\end{gather}
Therefore, we can take $\gamma_{r,\ell} \in (r,(r_{m-\ell})^*)$ such that $\gamma_{p,\ell} \le \gamma_{q,\ell}$ and
\begin{gather}\label{eq:hats-1}
    \frac{1}{s_{r,\ell}} < 1 - \frac{1}{\gamma_{r,\ell}} - \frac{1}{r'},
    \quad
    \frac{1}{t_{r,\ell}} < 1 - \frac{1}{\gamma_{r,\ell}},
    \quad
    \frac{\alpha}{q} - n\left( \frac{1}{\gamma_{p,\ell}} - \frac{1}{\gamma_{q,\ell}} \right) > 0.
\end{gather}
Here, we define $\widehat{s_{r,\ell}} \in (1,s_{r,\ell})$ and $\widehat{t_{r,\ell}} \in (1,t_{r,\ell})$ by
\begin{equation}
  \frac{1}{\widehat{s_{r,\ell}}} \coloneqq 1 - \frac{1}{\gamma_{r,\ell}} - \frac{1}{r'},
  \quad
  \frac{1}{\widehat{t_{r,\ell}}} \coloneqq 1 - \frac{1}{\gamma_{r,\ell}}.
\end{equation}
From these definitions, for each $\ell \in \{0,\ldots,m\}$ and $r \in \{p,q\}$, it follows that
\begin{equation}\label{eq:Holder}
\frac{1}{\widehat{s_{r,\ell}}} + \frac{1}{\gamma_{r,\ell}} + \frac{1}{r'} = 1,
\quad
\frac{1}{\widehat{t_{r,\ell}}} + \frac{1}{\gamma_{r,\ell}} = 1.
\end{equation}
Moreover, we can choose $\delta_0 \in \left( \frac{1}{p}, 1 \right)$ such that
\begin{equation}\label{eq:delta_0-0}
\frac{1}{\delta_0} < \beta,
\quad
\frac{\widehat{t_{r,\ell}}}{\delta_0} < t_{r,\ell}
\quad 
\text{for each }
\ell \in \{0,\ldots,m\}
\text{ and }
r \in \{p,q\},
\end{equation}
\begin{equation}\label{eq:delta_0-1}
    \frac{\widehat{s_{r,\ell}}}{\delta_0} < s_{r,\ell},
    \quad
     \frac{\gamma_{r,\ell}}{\delta_0} < ((r\delta_0)_{m-\ell})^{*}
    \quad
    \text{for\ each }
    \ell \in \{0,\ldots,m-1\}
    \text{ and }
    r \in \{p,q\},
\end{equation}
\begin{equation}\label{eq:delta_0-2}
   \frac{\alpha}{q} - n\left( \frac{1}{\gamma_{p,\ell}\delta_0} - \frac{\delta_0}{\gamma_{q,\ell}} \right) > 0
   \quad
    \text{for each }
    \ell \in \{0,\ldots,m\}.
\end{equation}

\subsection{Construction of truncated function}\label{subsec:construction}
Let us take $\delta \in \left[\frac{1 + \delta_0}{2}, 1 \right)$ to be determined later. 
Moreover, let $u \in W^{m,1}_{loc}(\Omega)$ be a very weak solution of \eqref{eq:main} with this $\delta$ (see Definition \ref{def:vws}).
Fix $\Omega_0 \subset\subset \widetilde{\Omega_0} \subset \subset \Omega$. For $z \in \R^k$ ($k \in \N$), we set
\begin{equation}\label{def:H_ell}
H_\ell (x,z) \coloneqq |z|^{\gamma_{p,\ell}} + a(x)^{\frac{\gamma_{q,\ell}}{q}}|z|^{\gamma_{q,\ell}}.
\end{equation}

In the first half of this subsection, we define the ``good set''$E(\lambda)$ (see \eqref{eq:Elambda}). 
This set is a lower level set of some technical l.s.c.\,function, on which $u$ belongs to $W^{m,\infty}$. 

Firstly, by the fact $u$ is a very weak solution with $\delta$, \eqref{eq:delta_0-2} (with $\ell = m$) and $\delta > \delta_0$, we have
\begin{equation}\label{eq:H_m}
H_{m}(x,D^m u)=|D^m u|^p + a(x)|D^m u|^q \in L^{\delta}(\widetilde{\Omega_0})
\end{equation}
and 
\begin{align}\label{eq:delta_0-4}
\frac{q}{p} < 1 + \frac{\alpha \delta}{n} 
&\impliedby \frac{q}{p} < 1 + \frac{\alpha \delta_0}{n} \\
&\iff \frac{\alpha}{q}-n\left(\frac{1}{p\delta_0}-\frac{1}{q\delta_0}\right) > 0
\impliedby \frac{\alpha}{q}-n\left(\frac{1}{p \delta_0}-\frac{\delta_0}{q}\right) > 0.
\end{align}
Moreover, by \eqref{eq:delta_0-1},\,\eqref{eq:H_m},\,\eqref{eq:delta_0-4} and Corollary \ref{cor:sp} (with $(p,q,a,\alpha)=(p\delta_0,q\delta_0,a^{\delta_0},\alpha \delta_0)$), 
we have for each $\ell \in \{0,\ldots,m-1\}$,
\begin{equation}\label{eq:conv-H_ell-0}
D^\ell u
\in
L^{\frac{\gamma_{p,\ell}}{\delta_0}}(\Omega_0),
\quad
a^{\frac{1}{q}}D^\ell u
\in
L^{\frac{\gamma_{q,\ell}}{\delta_0}}(\Omega_0).
\end{equation}
In particular, for each $\ell \in \{0,\ldots,m-1\}$
\begin{equation}\label{eq:conv-H_ell}
H_\ell(x,D^\ell u)
\in
L^{\frac{1}{\delta_0}}(\Omega_0).
\end{equation}
Here we can take $R_0 \in \left(0,\frac{1}{2}\right)$ such that for any $R \in (0,R_0]$ and $\ell \in \{0,\ldots,m\}$, it holds that
\begin{equation}\label{eq:R_0}
R^{\frac{\alpha}{q}-n\left(\frac{1}{\gamma_{p,\ell}\delta_0}-\frac{1}{\gamma_{q,\ell}\delta_0}\right)}\left( \|M^{\ell+1}(|D^\ell u|\chi_{\Omega_0})\|_{L^{\gamma_{p,\ell}\delta_0}(\Omega_0)}^{1-\frac{\gamma_{p,\ell}}{\gamma_{q,\ell}}} + 1\right)
\le 1.
\end{equation}
Indeed, for each $\ell \in \{0,\ldots,m\}$, it follows that $D^{\ell}u \in L^{\gamma_{p,\ell}\delta_0}(\Omega_0)$ and $D^{\ell}u \chi_{\Omega_0} \in L^{\gamma_{p,\ell}\delta_0}(\R^n)$ from \eqref{eq:conv-H_ell-0}.
Moreover, by $\gamma_{p,\ell}\delta_0 > 1$ (note $\gamma_{p,\ell}\delta_0 \ge p\delta_0 > 1$) and Lemma \ref{lem:Mfunct}, we obtain $M^{\ell + 1}(|D^\ell u|\chi_{\Omega_0}) \in L^{\gamma_{p,\ell}\delta_0}(\R^n)$. Therefore, $\|M^{\ell+1}(|D^\ell u|\chi_{\Omega_0})\|_{L^{\gamma_{p,\ell}\delta_0}(\Omega_0)}^{1-\frac{\gamma_{p,\ell}}{\gamma_{q,\ell}}}$ is finite. Since \eqref{eq:delta_0-2} implies $\frac{\alpha}{q}-n\left(\frac{1}{\gamma_{p,\ell}\delta_0}-\frac{1}{\gamma_{q,\ell}\delta_0}\right) > 0$, we can choose $R_0 \in \left(0,\frac{1}{2}\right)$ such that \eqref{eq:R_0} holds for any $R \in (0,R_0]$ and $\ell \in \{0,\ldots,m\}$.

For each $\ell \in \{0,\ldots,m\}$, we define
\begin{equation}
\beta_\ell \coloneqq  n\left( \frac{1}{\gamma_{p,\ell}\delta_0} - \frac{\delta_0}{\gamma_{q,\ell}} \right).
\end{equation} 
Then, \eqref{eq:delta_0-2} implies 
\begin{equation}\label{eq:conv-M_beta-0}
\beta_{\ell} \in \left( 0, \min\left\{\frac{n}{\gamma_{p,\ell}\delta_0},\frac{\alpha}{q}\right\}\right),
\,
\frac{n\gamma_{p,\ell}\delta_0}{n-\beta_\ell \gamma_{p,\ell}\delta_0} = \frac{\gamma_{q,\ell}}{\delta_0},
\,
\gamma_{p,\ell}\delta_0 > 1,
\,
D^\ell u \in L^{\gamma_{p,\ell}\delta_0}(\Omega_0)
\end{equation}
for each $\ell \in \{0,\ldots,m\}$.
Indeed, for each fixed $\ell \in \{0,\ldots,m\}$, it follows $ \beta_\ell > 0$ from $\gamma_{p,\ell} \le \gamma_{q,\ell}$ (recall the definition of $\gamma_{p,\ell}$ and $\gamma_{q,\ell}$).
Moreover, by the properties of the fractional maximal function \cite[Theorem 1.42]{KLV},
we obtain
\begin{equation}\label{eq:conv-M_beta}
M_{\beta_{\ell}} \left( M^{\ell+1}\left( D^\ell u\chi_{\Omega_0} \right) \right)^{\gamma_{q,\ell}}
\in 
L^{\frac{1}{\delta_0}}(\Omega_0).
\end{equation}
On the other hand, \eqref{eq:delta_0-0} and \eqref{assumpt:functions} imply
\begin{equation}\label{eq:f}
f_p + a(x)f_q,
\quad
h_{r,\ell}^{\widehat{t_{r,\ell}}} \in L^{\frac{1}{\delta_0}}(\Omega_0)
\quad
\text{for each }
\ell \in \{0,\ldots,m\}
\text{ and }
r \in \{p,q\}.
\end{equation}
In addition, \eqref{eq:delta_0-1} and \eqref{assumpt:functions} yield
\begin{equation}\label{eq:g}
g_{r,\ell}^{\widehat{s_{r,\ell}}} \in L^{\frac{1}{\delta_0}}(\Omega_0)
\quad
\text{for each }
\ell \in \{0,\ldots,m-1\}
\text{ and }
r \in \{p,q\}.
\end{equation}
Here let us define
\begin{align}\label{eq:F_0} 
F_0 &\coloneqq
    \sum_{\ell = 0}^{m - 1} \left( g_{p,\ell}^{\widehat{s_{p,\ell}}} 
    + g_{q,\ell}^{\widehat{s_{q,\ell}}}  \right)
     + \sum_{\ell= 0}^{m} \left( h_{p,\ell}^{\widehat{t_{p,\ell}}} + h_{q,\ell}^{\widehat{t_{q,\ell}}} + M_{\beta_{\ell}} \left( M^{\ell+1}\left( D^\ell u \cdot \chi_{\Omega_0} \right) \right)^{\gamma_{q,\ell}} \right).
\end{align}
Note that 
from \eqref{eq:conv-M_beta}, \eqref{eq:f}, \eqref{eq:g} and \eqref{eq:F_0}, it follows that $F_0 \in L^{\frac{1}{\delta_0}}(\Omega_0)$.
Then we set
\begin{equation}\label{eq:g_j}
g \coloneqq \left\{ \sum_{\ell = 0}^{m} M^{\ell+1} \left( H_{\ell}(x,D^\ell u)^{\delta_0}\chi_{B_{2R}} \right) 
+ (f_p + a(x)f_q)^{\delta_0}
+ F_0^{\delta_0} \right\}\chi_{B_{2R}} \in L^{\frac{\delta}{\delta_0}}(\R^n),
\end{equation}
\begin{equation}\label{eq:G}
G \coloneqq M(g)^{\frac{1}{\delta_0}},\,
\end{equation}

In the following, we fix any $B_{2R} \subset \subset \Omega_0$ and $R \le R_0$.
We first provide an upper estimate for $\fint_{B_{2R}}G^{\delta}$.
\begin{lem}\label{lem:G}
There exists $c = c(n,m,\delta_0)>0$ such that
\begin{equation}\label{eq:G-0}
    \fint_{B_{2R}} G^\delta \leq c\fint_{B_{2R}} \left( H_{m}(x,D^m u)^{\delta} + F^{\delta} \right),
\end{equation}
where
\begin{equation}\label{eq:F}
F \coloneqq F_0 + 1 + f_p + a(x)f_q + \sum_{\ell = 0}^{m-1} M^{\ell+1} \left( H_{\ell}(x,D^\ell u)^{\delta_0}\chi_{\Omega_0} \right)^{\frac{1}{\delta_0}} \in L^{\frac{1}{\delta_0}}(\Omega_0).
\end{equation}
\end{lem}
\begin{proof}
Applying Lemma \ref{lem:Mfunct} with $s = \frac{\delta}{\delta_0}$, we obtain
\begin{align}
    &\fint_{B_{2R}} G^\delta 
    = \fint_{B_{2R}} M(g)^\frac{\delta}{\delta_0} 
    \leq c_1 \left(n,\frac{\delta}{\delta_0} \right) \fint_{B_{2R}}g^\frac{\delta}{\delta_0} \\
    &\le  c_1 \left(n,\frac{\delta}{\delta_0} \right) \fint_{B_{2R}} \left( \sum_{\ell = 0}^{m} M^{\ell+1} \left( H_{\ell}(x, D^\ell u)^{\delta_0} \chi_{B_{2R}} \right) +(f_p + a(x)f_q)^{\delta_0} + F_0^{\delta_0} \right)^{\frac{\delta}{\delta_0}}.
\end{align}
Note that $1 < \frac{1 + \delta_0}{2\delta_0} \le \frac{\delta}{\delta_0} < \frac{1}{\delta_0}$ by $\delta \in \left[\frac{1+\delta_0}{2}, 1 \right)$, and that the constant $c(n,s)$ in Lemma \ref{lem:Mfunct} is continuous with respect to $s \in (1,\infty)$.
Therefore, we can take $c = c(n,\delta_0)$ such that
\begin{equation}\label{eq:delta_0-argument}
c_1 \left(n,\frac{\delta}{\delta_0} \right) \le c(n,\delta_0).
\end{equation}
Moreover, since
\begin{equation}
(a+b)^{\frac{\delta}{\delta_0}} 
\le 2^{\frac{\delta}{\delta_0}}(a^{\frac{\delta}{\delta_0}} + b^{\frac{\delta}{\delta_0}})
\le 2^{\frac{1}{\delta_0}}(a^{\frac{\delta}{\delta_0}} + b^{\frac{\delta}{\delta_0}})
\quad
\text{for }
a,b \ge 0,
\end{equation}
we derive
\begin{align}\label{eq:G-1}
\fint_{B_{2R}} G^\delta 
&\le c\fint_{B_{2R}} M^{m+1} \left( H_{m}(x,D^m u)^{\delta_0}\chi_{B_{2R}} \right)^{\frac{\delta}{\delta_0}} \\
&+ c\fint_{B_{2R}} \left(  \sum_{\ell=0}^{m-1} M^{\ell+1} \left( H_{\ell}(x,D^\ell u)^{\delta_0}\chi_{B_{2R}} \right)^{\frac{\delta}{\delta_0}} +(f_p + a(x)f_q)^{\delta}+ F_0^{\delta} \right) \\
&= c(n,m,\delta_0)(I_1 + I_2).
\end{align}
Here we set
\begin{equation}
  I_1 \coloneqq \fint_{B_{2R}} M^{m+1} \left( H_{m}(x,D^m u)^{\delta_0}\chi_{B_{2R}} \right)^{\frac{\delta}{\delta_0}}
\end{equation}
and
\begin{equation}
I_2 \coloneqq \fint_{B_{2R}} \left(  \sum_{\ell=0}^{m-1} M^{\ell+1} \left( H_{\ell}(x,D^\ell u)^{\delta_0}\chi_{B_{2R}} \right)^{\frac{\delta}{\delta_0}}+ (f_p + a(x)f_q)^{\delta} + F_0^{\delta} \right).
\end{equation}

For the estimate of $I_1$, applying Lemma \ref{lem:Mfunct} ($m+1$ times) with $s=\frac{\delta}{\delta_0}$, we have
\begin{equation}\label{eq:G-2}
I_1 \le c\fint_{B_{2R}} H_{m}(x,D^m u)^{\delta}.
\end{equation}
Note that we can make the constant $c$, which appears in \eqref{eq:G-2}, independent of $\delta$ by the same argument as in \eqref{eq:delta_0-argument}.

Next we evaluate $I_2$.
From \eqref{eq:F}, it follows
\begin{equation}
\sum_{\ell=0}^{m-1} M^{\ell+1} \left( H_{\ell}(x,D^\ell u)^{\delta_0}\chi_{B_{2R}} \right)^{\frac{\delta}{\delta_0}} + (f_p + a(x)f_q)^{\delta} + F_0^{\delta} \le (m+1)F^{\delta}.
\end{equation}
Therefore, we have
\begin{equation}\label{eq:G-3}
I_2 \le (m+1)\fint_{B_{2R}} F^{\delta}.
\end{equation}
Combining \eqref{eq:G-1}, \eqref{eq:G-2} and \eqref{eq:G-3}, we obtain \eqref{eq:G-0}.
Moreover, \eqref{eq:conv-H_ell}, \eqref{eq:f} and \eqref{eq:F_0} imply $F \in L^{\frac{1}{\delta_0}}(\Omega_0)$. The proof is completed.
\end{proof}

With the above preparation in hand, we turn our attention to the definition of the ``good set''$E(\lambda)$ and the proof of its properties.
Here set
\begin{equation}\label{eq:Lambda0}
    \Lambda_0 \coloneqq  2^{\frac{n}{\delta_0}}\left( \fint_{B_{2R}} G^{\delta} \right)^{\frac{1}{\delta}}+1.
\end{equation}
For $\lambda > 0$, we define
\begin{equation}\label{eq:Elambda}
    E(\lambda) \coloneqq \{x \in \mathbb{R}^n : G(x) \leq \lambda\}.
\end{equation}
Then, $E(\lambda)^c$ is an open set by the lower semicontinuity of the maximal function. Moreover, we obtain the following lemma.
\begin{lem}\label{lem:Elambda}
For any $\lambda > \Lambda_0$, we have $E(\lambda)^c \subset B_{4R}$.
\end{lem}
\begin{proof}
It suffices to show $\R^n\setminus B_{4R} \subset E(\lambda)$. Fix $x \in \mathbb{R}^n \setminus B_{4R}$. Then we have
\begin{equation}
    G(x) 
    = [M(g)(x)]^{\frac{1}{\delta_0}}
    = \left(\sup_{x \in B(x_0,r)} \frac{1}{|B(x_0,r)|} \int_{B(x_0,r)} g \right)^{\frac{1}{\delta_0}}.
\end{equation}
Since $x \in \mathbb{R}^n \setminus B_{4R}$, if
$x \in B(x_0,r)$ and $B(x_0,r) \cap B_{2R} \neq \emptyset$, then $r \ge R$. Indeed, take $y \in B(x_0,r) \cap B_{2R}$. Then we have
$2r > |x-x_0| + |x_0 - y| \ge |x-y| > 2R$. This implies $r \ge R$. Combining this inequality and the fact $g \equiv 0$ outside $B_{2R}$ (see \eqref{eq:g_j}), we obtain
\begin{equation}
\left(\sup_{x \in B(x_0,r)} \frac{1}{|B(x_0,r)|} \int_{B(x_0,r)} g \right)^{\frac{1}{\delta_0}}
    \leq \left( 2^n\fint_{B_{2R}} g \right)^{\frac{1}{\delta_0}} 
    \leq 2^{\frac{n}{\delta_0}}\left( \fint_{B_{2R}} M(g) \right)^{\frac{1}{\delta_0}}.
\end{equation}
Moreover, by Jensen's inequality and \eqref{eq:Lambda0}, we obtain
\begin{align}
     2^{\frac{n}{\delta_0}}\left( \fint_{B_{2R}} M(g) \right)^{\frac{1}{\delta_0}} 
    &\leq  2^{\frac{n}{\delta_0}}\left( \fint_{B_{2R}} M(g)^\frac{\delta}{\delta_0} \right)^{\frac{1}{\delta}} \\
    &=  2^{\frac{n}{\delta_0}}\left( \fint_{B_{2R}} G^\delta \right)^\frac{1}{\delta}
    \le \Lambda_0 < \lambda.
\end{align}
Therefore, $\R^n \setminus B_{4R} \subset E(\lambda)$. This completes the proof.
\end{proof}

By Lemma \ref{lem:Elambda}, $E(\lambda)^c$ is a bounded open set.
Then, using a Whitney-type covering lemma \cite[Lemma C.1]{DRW}, 
we can construct a countable family of open balls $\{B_i\} = \{B(x_i,r_i)\}_{i \in \mathbb{N}}$ satisfying the following properties:
\begin{description}
  \item[(W1)] $\bigcup_{i \in \mathbb{N}} \frac{1}{2}B_i = E(\lambda)^c$.
  \item[\textbf{(W2)}] $r_i \leq R$ for any $i \in \mathbb{N}$.
  \item[\textbf{\textbf{(W3)}}] $8B_i \subset E(\lambda)^c$ and $16B_i \cap E(\lambda) \neq \emptyset$.
  \item[\textbf{(W4)}] If $B_i \cap B_j \neq \emptyset$ for any $i,j \in \mathbb{N}$, then $\frac{1}{2}r_j \leq r_i \leq 2r_j$.
  \item[(W5)]$\{\frac{1}{4}B_i \}_{i \in \mathbb{N}}$ is disjoint.
\end{description}
In addition, we define, for any $i \in \mathbb{N}$,
\begin{equation}\label{eq:Ai}
    A_i \coloneqq \left\{ j \in \mathbb{N} : \frac{3}{4}B_i \cap \frac{3}{4}B_j \neq \emptyset \right\}.
\end{equation}
\begin{description}
    \item[\textbf{(W6)}] There exists $c(n)>0$ such that $\sharp{A_i} < c(n)$ for any $i \in \mathbb{N}$.
    \item[\textbf{(W7)}] There exists $c(n)>0$ such that $\left|B_i \cap \frac{3}{4}B_j \right| \geq c(n)^{-1} \max \{ |B_i|,|B_j| \}$ for any $i \in \mathbb{N}$ and $j \in A_i$.
\end{description}
Furthermore, with the covering $\{\frac{3}{4}B_i\}$, we can construct a partition of unity $\{\psi_i\} \subset C_c^\infty(\mathbb{R}^n)$ in a standard way (see \cite[Chapter 6]{St} or \cite[pp.10-11]{DRW}, for example) such that
\begin{description}
    \item[\textbf{(P1)}] There exists $\kappa = \kappa(n) > 0$ such that $\kappa\chi_{\frac{1}{2}B_i} \le \psi_i \le 1$ and $\psi_i \in C_c^{\infty}(\frac{3}{4}B_i)$ for every $i \in \mathbb{N}$.
    \item[(P2)] $|D^{\ell} \psi_i| \leq c(n,m)r_i^{-\ell}$ in $\mathbb{R}^n$ for every $i \in \mathbb{N}$ and $\ell \in \{0,\ldots,m\}$.
    \item[\textbf{(P3)}] $\sum_{j \in A_i} \psi_j = 1$ in $\frac{3}{4}B_i$ for every $i \in \mathbb{N}$.
\end{description}

In the following, we define the truncated function $v_\lambda$. Firstly, let us take $\eta \in C^\infty_c(B_{2R})$ such that
\begin{equation}\label{eq:eta}
    \chi_{B_R} \leq \eta \leq \chi_{B_{2R}},
    \quad
    |D^\ell \eta| \leq c(m,n)R^{-\ell}
    \quad
    \text{for each }
    \ell \in \{0,\ldots,m\}.
\end{equation} 
Then we define $v \in W^{m,p\delta}_{0}(B_{2R})$ by
\begin{equation}\label{eq:v}
v \coloneqq (u-P)\eta.
\end{equation}
Here $P$ is a unique polynomial satisfying the following properties:
\begin{equation}\label{eq:P}
    \deg P \le m - 1,
    \quad
    (\partial_\sigma u)_{B_{2R},\eta} 
    = (\partial_\sigma P)_{B_{2R},\eta}
    \quad
    \text{for each }
    \sigma \in S \setminus S_m.
\end{equation}
Finally, let $v_{\lambda}$ be the measurable function on $\R^n$ defined by
\begin{equation}\label{eq:vlambda} 
  v_\lambda  \coloneqq
  \left\{ 
  \begin{alignedat}{2}   
   v \quad &
   \text{in}\,E(\lambda), \\
   \sum_{i \in \mathbb{N}}\bar{P_i}\psi_i \quad
   &\text{in}\,E(\lambda)^c.
  \end{alignedat} 
  \right.
\end{equation}
where for $i \in \mathbb{N}$, we define
\begin{equation}
\bar{P_i}  \coloneqq
  \left\{ 
  \begin{alignedat}{2}   
   P_i
   \quad
   &\text{if}\,\frac{3}{4}B_i \subset B_{2R}, \\
   0
   \quad
   &\text{if}\,\frac{3}{4}B_i \not\subset B_{2R}.
  \end{alignedat} 
  \right.
\end{equation}
Moreover, 
$P_i$ is a unique polynomial such that
\begin{equation}\label{eq:Pi}
    \deg P_i \le m - 1,
    \quad
    (\partial_\sigma v)_{\frac{3}{4}B_i,\psi_i} 
    = (\partial_\sigma P_i)_{\frac{3}{4}B_i,\psi_i}
    \quad
    \text{for each }
    \sigma \in S \setminus S_m.
\end{equation}
From the same argument as \cite[Lemma\,24]{DiKS} and \cite[Lemma\,3.2]{BuS}, it follows that $v_{\lambda} \in W^{m,p\delta}_0(B_{2R})$.
Moreover, note that $D^{\ell} v_{\lambda} = D^{\ell} v$ a.e.\,on $E(\lambda)$ for each $\ell = \{0,\ldots,m\}$ by \cite[Lemma\,7.7]{GT}.
\begin{rem}
\textbf{(W2)} easily follows from the fact $E(\lambda)^c \subset B_{4R}$ and the proof of \cite[Lemma C.1]{DRW}.
\end{rem}

\begin{rem}
We adopt \textbf{(W7)} in place of the following stronger condition found in \cite[section 2.3]{DSSV}:
\begin{description}
\item[\textbf{(W7)}'] There exists $c(n)>0$ such that $\left|\frac{3}{4}B_i \cap \frac{3}{4}B_j \right| \geq c(n)^{-1} \max \{ |B_i|,|B_j| \}$ for any $i \in \mathbb{N}$ and $j \in A_i$.
\end{description}
This is because the proof of \textbf{(W7)'} is not explicitly provided therein, and \textbf{(W7)} is sufficient for our purpose.
Moreover, \textbf{(W7)} follows from a simple geometric observation and the property \textbf{(W4)}.
\end{rem}

\subsection{Estimates of the truncated function}
The purpose of the first half of this subsection is to estimate the derivative of $v_\lambda$. 
To this end, we will establish a further estimate concerning weighted mean value polynomials.
For the moment, we fix $i \in \mathbb{N}$ and $j \in A_i$.
There exists a unique polynomial $Q_i$ such that
\begin{equation}\label{eq:Qi}
    \deg Q_i \le m - 1,
    \quad
    (\partial_\sigma v)_{B_i,\psi_i} = (\partial_\sigma Q_i)_{B_i,\psi_i}
    \quad
    \text{for each } \sigma \in S \setminus S_m.
\end{equation}
Then we define $a_{\sigma,i}$ as the coefficient of $(x - x_i)^\sigma$ in the expansion of $Q_i$ about $x_i$. 
Here we recall $x_i$ is the center of $B_i$.
Let us define
\begin{equation}\label{eq:bsigmai}
    b_{\sigma,i} 
    \coloneqq \sum_{\tau \in S \setminus S_m,\tau > \sigma}\frac{\tau !}{(\tau - \sigma)!}a_{\tau,i}(x - x_i)^{\tau - \sigma}
    = \partial_\sigma Q_i - \sigma!a_{\sigma,i}.
\end{equation}
Since Lemma \ref{lem:Pi} \textbf{(i)} (with $(B,\mathcal{B},\eta,x_0)=(B_i,B_i,\psi_i,x_i)$) implies
\begin{equation}
    a_{\sigma,i} = \frac{1}{\sigma!}\left( \partial_\sigma v 
  - \sum_{\tau \in S \setminus S_m,\tau > \sigma}\frac{\tau!}{(\tau - \sigma)!}a_{\tau,i}(x - x_i)^{\tau - \sigma}  \right)_{B_i,\psi_i}
\end{equation}
for each $\sigma \in S \setminus S_m$, we have

\begin{equation}\label{eq:asigmai}
a_{\sigma,i} = \frac{1}{\sigma!}(\partial_\sigma v - b_{\sigma,i})_{B_i,\psi_i}.
\end{equation}

Similarly, for each $\sigma \in S \setminus S_m$, let
$a_{\sigma,j,i}$ be the coefficient of $(x - x_i)^\sigma$ when expanding $P_j$ about $x_i$.
For $\sigma \in S \setminus S_m $, we set
\begin{equation}\label{eq:bsigmaji}
    b_{\sigma,j,i} 
    \coloneqq \sum_{\tau \in S \setminus S_m,\tau > \sigma}\frac{\tau !}{(\tau - \sigma)!}a_{\tau,j,i}(x - x_i)^{\tau - \sigma}
    = \partial_\sigma P_j - \sigma!a_{\sigma,j,i}.
\end{equation}
By Lemma \ref{lem:Pi} \textbf{(i)} (with $(B,\mathcal{B},\eta,x_0)= \left(\frac{3}{4}B_j,B_i,\psi_j,x_i\right)$), we obtain
\begin{equation} \label{eq:asigmaji}
a_{\sigma,j,i} = \frac{1}{\sigma!}(\partial_\sigma v - b_{\sigma,j,i})_{\frac{3}{4}B_j,\psi_j}.
\end{equation}
Finally, for each $\ell \in \{0,\ldots,m\}$, we denote
\begin{equation}\label{eq:T_{ell,i}}
T_{\ell,i} \coloneqq \sup_{\widetilde{j} \in A_i}\fint_{B_{\widetilde{j}}}|D^\ell v|.
\end{equation}

With the above notation, we estimate the difference between $P_j$ and $Q_i$.
\begin{lem}\label{lem:Pioscilation}
There exists $c=c(n,m) > 0$ such that the following inequalities hold:
\begin{description}
  \item [(i)] $|a_{\sigma,j,i} - a_{\sigma,i}| \leq c r_i^{\ell - k}T_{\ell,i} $ 
  $\mathrm{for\ each}$ $\ell \in \{0,\ldots,m\},k \in \{0,\ldots,m-1\}$ $\mathrm{and}$ $\sigma \in S_k$.
\item [(ii)] $|\partial_\sigma \bar{P_j} - \partial_\sigma Q_i| \leq cr_i^{\ell - k} T_{\ell,i}$ 
$\mathrm{in}$ $\frac{3}{4}B_i$ $\mathrm{for\ each}$ $\ell \in \{0,\ldots,m\},k \in \{0,\ldots,m-1\}$ $\mathrm{and}$ $\sigma \in S_k$.
\end{description} 
\end{lem}
\begin{proof}
\noindent \textbf{(i)}:\,fix $\ell \in \{0,\ldots,m\}$ and $k \in \{0,\ldots,m-1\}$. 
Firstly, we assume $k \in \{\ell,\ldots,m - 1\}$. 
Using Lemma \ref{lem:Pi} \textbf{(ii)} (with $(B,\mathcal{B},x_0)=(\frac{3}{4}B_j,B_i,x_i)$ 
and $(B,\mathcal{B},x_0)=(B_i,B_i,x_i)$), we have
\begin{equation}
    |a_{\sigma,j,i} - a_{\sigma,i}| 
    \leq |a_{\sigma,j,i}| + |a_{\sigma,i}|
    \leq c\sum_{\mu = k}^{m - 1}r_j^{\mu - k}|(D^\mu v)_{\frac{3}{4}B_j,\psi_j}| + c\sum_{\mu = k}^{m - 1}r_i^{\mu - k}|(D^\mu v)_{B_i,\psi_i}|
\end{equation}
for any $\sigma \in S_k$.
Here we recall that $r_i$ and $r_j$ are the radii of $B_i$ and $B_j$, respectively. 
Moreover, by \textbf{(P1)}, \textbf{(P2)}, integration by parts, \textbf{(W4)} and \eqref{eq:T_{ell,i}}, we obtain
\begin{align}\label{eq:integration_by_parts}
r_j^{\mu - k}|(D^\mu v)_{\frac{3}{4}B_j,\psi_j}|
&\le cr_j^{\mu - k}|(D^\mu v \cdot \psi_j)_{\frac{3}{4}B_j}| \\
&\leq cr_j^{\mu - k}\fint_{\frac{3}{4}B_j}|D^\ell v||D^{\mu-\ell}\psi_j| \\
&\leq cr_j^{\ell - k}\fint_{\frac{3}{4}B_j}|D^\ell v| \\
&\leq cr_i^{\ell - k}\fint_{\frac{3}{4}B_j}|D^\ell v|
\le c(n,m)r_i^{\ell - k}T_{\ell,i}
\end{align}
for each $\mu \in \{k,\ldots,m-1\}$.
In the same way, we can get $r_i^{\mu - k}|(D^\mu v)_{B_i,\psi_i}| \leq c(n,m)r_i^{\ell - k}T_{\ell,i}$.
Therefore, we obtain
\begin{equation}
    |a_{\sigma,j,i} - a_{\sigma,i}| \leq cr_i^{\ell - k}T_{\ell,i}.
\end{equation}

Next, let us consider the case $k \in \{0,\ldots,\ell-1\}$.
We show by induction:\,
suppose that there exists $\ell_0 \in \{1,\ldots,\ell\}$ such that the assertion holds for every $k \in \{\ell_0,\ldots,m-1\}$. 
Then, the case $k = \ell_0 - 1$ also holds. Indeed,
fix any $\sigma \in S_{\ell_0 - 1}$.
\eqref{eq:asigmai} and \eqref{eq:asigmaji} yield
\begin{align}\label{eq:Pioscilation-I's}
    \sigma!|a_{\sigma,j,i} - a_{\sigma,i}|
    &= |(\partial_\sigma v - b_{\sigma,j,i})_{\frac{3}{4}B_j,\psi_j} - (\partial_\sigma v - b_{\sigma,i})_{B_i,\psi_i}| \\
    &\leq |(\partial_\sigma v - b_{\sigma,j,i})_{\frac{3}{4}B_j,\psi_j} - (\partial_\sigma v - b_{\sigma,j,i})_{\frac{3}{4}B_j \cap B_i}|\\
    &+ |(\partial_\sigma v - b_{\sigma,j,i})_{\frac{3}{4}B_j \cap B_i} - (\partial_\sigma v - b_{\sigma,i})_{\frac{3}{4}B_j \cap B_i}| \\
    &+ |(\partial_\sigma v - b_{\sigma,i})_{\frac{3}{4}B_j \cap B_i} - (\partial_\sigma v - b_{\sigma,i})_{B_i,\psi_i}|
    \eqcolon I_1 + I_2 + I_3.
\end{align}

Firstly, we estimate $I_2$. From \eqref{eq:bsigmai}, \eqref{eq:bsigmaji} and the induction hypothesis, it follows that
\begin{align}\label{eq:Pioscilation-I_2}
   I_2 
   &\leq \fint_{\frac{3}{4}B_j \cap B_i} |b_{\sigma,j,i} - b_{\sigma,i}| \\
   &\leq c\sum_{\tau \in S \setminus S_m,\tau > \sigma} \fint_{\frac{3}{4}B_j \cap B_i}|a_{\tau,j,i} - a_{\tau,i}||x - x_i|^{|\tau| - |\sigma|} \\
   &\leq c\sum_{\tau \in S \setminus S_m,\tau > \sigma} \fint_{\frac{3}{4}B_j \cap B_i}r_i^{\ell - |\tau|}T_{\ell,i}r_i^{|\tau| - |\sigma|}
   \leq cr_i^{\ell - \ell_0 + 1}T_{\ell,i}.
\end{align}
Note that if $\tau \in S \setminus S_m$ and $\tau > \sigma$, then $\tau \in S_\mu$ for some $\mu \in \{\ell_0,\ldots,m-1\}$. This explains why we can use the induction hypothesis.

Next, we estimate $I_1$. Using \textbf{(W7)}, we have
\begin{align}
    I_1 &\leq \fint_{\frac{3}{4}B_j \cap B_i} |\partial_\sigma v - b_{\sigma,j,i} - (\partial_\sigma v - b_{\sigma,j,i})_{\frac{3}{4}B_j,\psi_j}| \\
        &\leq c\fint_{\frac{3}{4}B_j} |\partial_\sigma v - b_{\sigma,j,i} - (\partial_\sigma v - b_{\sigma,j,i})_{\frac{3}{4}B_j,\psi_j}| 
        = c(n)\fint_{\frac{3}{4}B_j}|H|.
\end{align}
Here we set
\begin{equation}
  H \coloneqq \partial_\sigma v - b_{\sigma,j,i} - (\partial_\sigma v - b_{\sigma,j,i})_{\frac{3}{4}B_j,\psi_j}.
\end{equation}
Clearly, $(H)_{\frac{3}{4}B_j,\psi_j} = 0$.
In addition, by \eqref{eq:bsigmaji} and \eqref{eq:Pi}, we have
\begin{align}\label{eq:Pioscilation-H}
(\partial_\tau H)_{\frac{3}{4}B_j,\psi_j}
&= (\partial_\tau\partial_{\sigma}v - \partial_\tau b_{\sigma,j,i})_{\frac{3}{4}B_j,\psi_j} \\
&= (\partial_{\sigma + \tau}v - \partial_{\sigma + \tau}P_j)_{\frac{3}{4}B_j,\psi_j} = 0
\quad
\text{for any } 
\tau \in \bigcup_{\mu = 1}^{\ell - \ell_0} S_\mu.
\end{align}
Note that if $\tau \in \bigcup_{\mu = 1}^{\ell - \ell_0} S_\mu$, then $\tau + \sigma = (\tau_1+\sigma_1, \cdots, \tau_n+\sigma_n) \in S \setminus S_m$.
Then, using the classical Poincar\'e inequality, we obtain
\begin{equation}
    \fint_{\frac{3}{4}B_j}|H| \leq cr_j^{\ell - \ell_0 + 1}\fint_{\frac{3}{4}B_j}|D^{\ell - \ell_0 + 1}H|.
\end{equation}
Moreover, from \eqref{eq:bsigmaji} it follows
\begin{align}
|D^{\ell - \ell_0 + 1}H| 
&= |D^{\ell - \ell_0 + 1}(\partial_{\sigma}v - b_{\sigma,j,i})| \\
&=|D^{\ell - \ell_0 + 1}(\partial_{\sigma}v - \partial_{\sigma}P_j)|
\leq |D^{\ell}v| + |D^{\ell}P_j| \quad \text{in }\R^n.
\end{align}
Then \eqref{eq:integration by parts} (with $(B,\mathcal{B},\eta,x_0)=(\frac{3}{4}B_j,B_i,\psi_j,x_i)$) in Remark \ref{rem:integration by parts} yields 
$$ 
|D^{\ell}P_j| \leq c(n,m)\fint_{\frac{3}{4}B_j}|D^{\ell}v| \quad \text{in } \frac{3}{4}B_j.
$$
By this inequality and \eqref{eq:T_{ell,i}}, we deduce
\begin{equation}\label{eq:Pioscilation-I_1}
I_1
\le cr_i^{\ell - \ell_0 + 1}\fint_{\frac{3}{4}B_j} |D^{\ell}v|
\leq c(n,m)r_i^{\ell - \ell_0 + 1}T_{\ell,i}.
\end{equation}

We estimate $I_3$ similarly to $I_1$.
Indeed,  we have 
\begin{align}
    I_3 &\leq \fint_{\frac{3}{4}B_j \cap B_i} |\partial_\sigma v - b_{\sigma,i} - (\partial_\sigma v - b_{\sigma,i})_{B_i,\psi_i}| \\
        &\leq c\fint_{B_i} |\partial_\sigma v - b_{\sigma,i} - (\partial_\sigma v - b_{\sigma,i})_{B_i,\psi_i}| 
        = c(n)\fint_{B_i}|\widetilde{H}|
\end{align}
by \textbf{(W7)}.
Here we set
\begin{equation}
  \widetilde{H} \coloneqq \partial_\sigma v - b_{\sigma,i} - (\partial_\sigma v - b_{\sigma,i})_{B_i,\psi_i}.
\end{equation}
Observe $(\widetilde{H})_{B_i,\psi_i} = 0$.
Moreover, using \eqref{eq:bsigmai} and \eqref{eq:Qi}, we have
\begin{align}
(\partial_\tau \widetilde{H})_{B_i,\psi_i}
&= (\partial_\tau \partial_{\sigma}v - \partial_{\tau} b_{\sigma,i})_{B_i,\psi_i} \\
&= (\partial_{\sigma + \tau}v - \partial_{\sigma + \tau}Q_i)_{B_i,\psi_i} = 0
\quad
\text{for any } 
\tau \in \bigcup_{\mu = 1}^{\ell - \ell_0} S_\mu.
\end{align}
Therefore,
\begin{equation}
    \fint_{B_i}|\widetilde{H}| \leq cr_i^{\ell - \ell_0 + 1}\fint_{B_i}|D^{\ell - \ell_0 + 1}\widetilde{H}|.
\end{equation}
Moreover, \eqref{eq:bsigmai} implies
\begin{align}
|D^{\ell - \ell_0 + 1}\widetilde{H}| 
&= |D^{\ell - \ell_0 + 1}(\partial_{\sigma}v - b_{\sigma,i})| \\
&=|D^{\ell - \ell_0 + 1}(\partial_{\sigma}v - \partial_{\sigma}Q_i)|
\leq |D^{\ell}v| + |D^{\ell}Q_i|\quad \text{in }\R^n,
\end{align}
and \eqref{eq:integration by parts} ($(B,\mathcal{B},\eta,x_0)=(B_i,B_i,\psi_i,x_i)$) in Remark \ref{rem:integration by parts} implies
$$ 
|D^{\ell}Q_i| \leq c(n,m)\fint_{B_i}|D^{\ell}v| \quad \text{in } B_i.
$$
Hence,
\begin{equation}\label{eq:Pioscilation-I_3}
I_3 \leq cr_i^{\ell - \ell_0 + 1}T_{\ell,i}.
\end{equation}
From \eqref{eq:Pioscilation-I's}, \eqref{eq:Pioscilation-I_2}, \eqref{eq:Pioscilation-I_1} and \eqref{eq:Pioscilation-I_3}, the conclusion follows.

\noindent \textbf{(ii)}:\,fix $\ell \in \{0,\ldots,m\}, k \in \{0,\ldots,m-1\}$ and $\sigma \in S_k$.
Then, by \textbf{(i)} we have
\begin{align}
    |\partial_{\sigma}P_j - \partial_{\sigma}Q_i|
    &\leq c\sum_{\tau \in S \setminus S_m,\tau \geq \sigma}|a_{\tau,j,i} - a_{\tau,i}||(x - x_i)^{\tau - \sigma}|\\
    &\leq c\sum_{\tau \in S \setminus S_m,\tau \geq \sigma}r_i^{\ell - |\tau|}T_{\ell,i}r_i^{|\tau| - |\sigma|}
    \leq cr_i^{\ell - k}T_{\ell,i}
    \quad
    \text{in }
    \frac{3}{4}B_i.
\end{align}
Therefore, if $\frac{3}{4}B_j \subset B_{2R}$, the claim follows.
Now let us consider the case $\frac{3}{4}B_j \not \subset B_{2R}$.
We have 
\begin{align}
|\partial_{\sigma}\bar{P_j} - \partial_{\sigma}Q_i| 
= |\partial_{\sigma}Q_i|
\le |\partial_{\sigma}P_j| + |\partial_{\sigma}P_j - \partial_{\sigma}Q_i|
\le |\partial_{\sigma}P_j| +  cr_i^{\ell - k}T_{\ell,i}
\quad
\text{in }
\frac{3}{4}B_i.
\end{align}
Moreover, by Lemma\,\ref{lem:Pi}, we have
$|\partial_{\sigma} P_j| \leq c(m,n)\sum_{\mu = k}^{m - 1}r_i^{\mu - k}|(D^{\mu}v)_{\frac{3}{4}B_j,\psi_j}|$ in $\frac{3}{4}B_i$.
For any $\mu \in \{\ell,\ldots,m-1\}$, we obtain $r_i^{\mu - k}|(D^{\mu}v)_{\frac{3}{4}B_j,\psi_j}| \le cr_i^{\ell - k}T_{\ell,i}$ by integration by parts (see \eqref{eq:integration_by_parts}).
Next we take any $\mu \in \{k,\ldots,\ell-1\}$.
We have $\frac{3}{4}B_j \not \subset B_{2R}$, $v \equiv 0$ in $B_{2R}^c$ and $|B_j \setminus B_{2R}| \ge c(n)|B_j|$.
Therefore, Friedrichs' inequality (see \cite[lemma\,2.2]{BBK} or \cite[Lemma\,1.65]{MaZ}, for example) implies
$r_i^{\mu - k}|(D^{\mu}v)_{\frac{3}{4}B_j,\psi_j}| \le cr_i^{\ell - k}\fint_{B_j}|D^\ell v|
\le cr_i^{\ell - k}T_{\ell,i}$. 
The proof is completed.
\end{proof}

Now we estimate the derivative of the truncated function.
\begin{lem}\label{lem:derivative}
There exist $c_1=c(n,m)$ and $c_2=c(n,m,p,q,\alpha,[a]_{\alpha},\delta_0)$ such that
\begin{align}
    |D^k v_\lambda| 
    &\leq c_1 R^{\ell - k} \lambda^\frac{1}{\gamma_{p,\ell}}
    \quad
    \mathrm{in\,}
    E(\lambda)^c, \\
    a(x)^{\frac{1}{q}}|D^k v_\lambda| 
    &\leq c_2 R^{\ell - k} \lambda^\frac{1}{\gamma_{q, \ell}}
    \quad
    \mathrm{in\,}
    E(\lambda)^c \cap B_{2R}
\end{align}
for each $\ell \in \{0,\ldots,m\}$ and $k \in \{0,\ldots,\ell\}$.
\end{lem}
\begin{proof}
Fix $\ell \in \{0,\ldots,m\}$ and $k \in \{0,\ldots,\ell\}$.
By \textbf{(W1)}, it suffices to show that 
\begin{align}
|D^k v_\lambda| &\leq c(n,m) R^{\ell - k}\lambda^\frac{1}{\gamma_{p,\ell}}\  \text{in}\ \frac{3}{4}B_i, \\
a(x)^{\frac{1}{q}}|D^k v_\lambda| &\leq c(n,m,p,q,\alpha,[a]_{\alpha},\delta_0) R^{\ell - k}\lambda^\frac{1}{\gamma_{q,\ell}}\  \text{in}\ \frac{3}{4}B_i\cap B_{2R}
\end{align}
for every $i \in \mathbb{N}$. 

We fix any $i \in \mathbb{N}$.
Since $v_\lambda = \sum_{j \in A_i} \bar{P_j}\psi_j$ in $\frac{3}{4}B_i$ (see \eqref{eq:vlambda}), 
we have
\begin{equation}
    \partial_\sigma v_\lambda 
    = \sum_{j \in A_i} \sum_{\sigma_1 + \sigma_2 = \sigma} \binom{\sigma}{\sigma_1}\partial_{\sigma_1} \bar{P_j} \partial_{\sigma_2}\psi_j 
    = \sum_{\sigma_1 + \sigma_2 = \sigma}\sum_{j \in A_i}\binom{\sigma}{\sigma_1} \partial_{\sigma_1} \bar{P_j} \partial_{\sigma_2}\psi_j
\end{equation}
for each $\sigma \in S_k$. Note that we used the fact $\sharp A_i < \infty$ (see \textbf{(W6)}) to justify the interchange of differentiation and summation.
Therefore we obtain
\begin{equation}
    |\partial_\sigma v_\lambda| 
	\leq c(m,n)\sum_{\sigma_1 + \sigma_2 
    = \sigma} \left|\sum_{j \in A_i} \partial_{\sigma_1} \bar{P_j} \partial_{\sigma_2}\psi_j \right|
    \quad
    \mathrm{in\,}
    \frac{3}{4}B_i.
\end{equation}
Fix $\sigma_1$ and $\sigma_2$ such that $\sigma = \sigma_1 + \sigma_2$. We only have to show the following two inequalities:
\begin{align}\label{eq:derivative-onlyhaveto}
    \left|\sum_{j \in A_i} \partial_{\sigma_1} \bar{P_j} \partial_{\sigma_2}\psi_j \right|
    &\leq c(n,m)R^{\ell - k}\lambda^\frac{1}{\gamma_{p,\ell}}
    \quad
    \mathrm{in\,}
    \frac{3}{4}B_i, \\
    a(x)^{\frac{1}{q}}\left|\sum_{j \in A_i} \partial_{\sigma_1} \bar{P_j} \partial_{\sigma_2}\psi_j \right|
    &\leq c(n,m,p,q,\alpha,[a]_{\alpha},\delta_0)R^{\ell - k}\lambda^\frac{1}{\gamma_{q,\ell}}
    \quad
    \mathrm{in\,}
    \frac{3}{4}B_i \cap B_{2R}.
\end{align}

\noindent \textbf{Case 1 ($\sigma_2 = 0$)}:\,then $\sigma_1 = \sigma$ and it follows that
\begin{align}\label{eq:case1-1}
    \left|\sum_{j \in A_i} \partial_{\sigma_1} \bar{P_j} \partial_{\sigma_2}\psi_j\right| 
    &\le \sum_{j \in A_i} |\partial_{\sigma} \bar{P_j}| \\
    &\leq  \sharp{A_i} \sup_{j \in A_i} |D^k P_j|
    \leq c(n)\sup_{j \in A_i} |D^k P_j|
    \quad \text{in }\R^n.
\end{align} 
Here we used \textbf{(P1)} and \textbf{(W6)}. 

We show
\begin{equation}\label{eq:case1-2}
  |D^k P_j| \le c(n,m)R^{\ell-k}\fint_{\frac{3}{4}B_j}M_{B_{2R}}^{\ell+1}\left(D^\ell u\right)\chi_{B_{2R}}
  \quad
  \text{in }\frac{3}{4}B_i
\end{equation}
for any $j \in A_i$. Let us take any $j \in A_i$.
If $k=m$, \eqref{eq:case1-2} is obvious by \eqref{eq:Pi}. Therefore, we may assume $k<m$. Firstly, Lemma \ref{lem:Pi} \textbf{(iii)} with $(B,\mathcal{B}) = (\frac{3}{4}B_j,\frac{3}{4}B_i)$ implies
\begin{equation}\label{eq:case1-3}
    |D^k P_j|
    \leq c\sum_{\mu = k}^{m - 1}r_j^{\mu - k} |(D^{\mu}v)_{\frac{3}{4}B_j,\psi_j}| 
    \quad
    \text{in }\frac{3}{4}B_i.
\end{equation}
We fix $\mu \in \{k,\ldots,m - 1\}$. If $\mu \leq \ell$, then it holds that
\begin{align}\label{eq:case1-4}
    r_j^{\mu - k}|(D^{\mu}v)_{\frac{3}{4}B_j,\psi_j}|
    &= r_j^{\mu - k} \left| \frac{1}{\|\psi_j\|_{L^1 \left( \frac{3}{4}B_j \right)}}\int_{\frac{3}{4}B_j} D^{\mu}v \cdot \psi_j \right| \\
    &\le c(n)r_j^{\mu - k} \left| \fint_{\frac{3}{4}B_j} D^{\mu}v \cdot \psi_j \right| \\
    &\leq cr_j^{\mu - k} \fint_{\frac{3}{4}B_j} |D^{\mu}v| \\
    &\leq cR^{\mu - k} \fint_{\frac{3}{4}B_j} |D^{\mu}v|
    = c(n)R^{\ell - k} \fint_{\frac{3}{4}B_j} \left| \frac{D^{\mu}v}{R^{\ell - \mu}} \right|
\end{align}
by \eqref{eq:weight-mean}, \textbf{(P1)} and \textbf{(W2)}. On the other hand, if $\mu > \ell$, then we obtain 
\begin{align}\label{eq:case1-5}
    r_j^{\mu - k}|(D^{\mu}v)_{\frac{3}{4}B_j,\psi_j}|
    &\le c(n)r_j^{\mu - k}|(D^{\mu}v \cdot \psi_j)_{\frac{3}{4}B_j}| \\
    &\le  cr_j^{\mu - k} \fint_{\frac{3}{4}B_j} |D^{\ell}v||D^{\mu-\ell}\psi_j| \\
    &\leq cr_j^{\ell - k} \fint_{\frac{3}{4}B_j} |D^{\ell}v|
    \leq c(n,m)R^{\ell - k}\fint_{\frac{3}{4}B_j} |D^{\ell}v|
\end{align}
by \textbf{(P2)} and integration by parts. From \eqref{eq:case1-3}, \eqref{eq:case1-4} and \eqref{eq:case1-5}, we derive
\begin{equation}\label{eq:case1-6}
  |D^k P_j| 
  \le c(n,m)R^{\ell - k}\sum_{\mu = 0}^{\ell}\fint_{\frac{3}{4}B_j} \left| \frac{D^{\mu}v}{R^{\ell - \mu}} \right|
  \quad
  \text{in }\frac{3}{4}B_i.
\end{equation}
In addition, \eqref{eq:eta}, \eqref{eq:v} and Lemma \ref{lem:useful-formula} yield
\begin{align}\label{eq:est-T_{ell,i}}
\sum_{\mu=0}^{\ell}\left| \frac{D^{\mu}v}{R^{\ell - \mu}} \right| 
&\le c\sum_{\mu=0}^{\ell}\sum_{\mu_1=0}^{\mu}\frac{|D^{\mu_1}u - D^{\mu_1}P||D^{\mu-\mu_1}\eta|}{R^{\ell - \mu}} \\
&\le c\sum_{\mu=0}^{\ell}\sum_{\mu_1=0}^{\mu}\frac{|D^{\mu_1}u - D^{\mu_1}P|}{R^{\ell - \mu_1}}\chi_{B_{2R}} \\
&\leq c\sum_{\mu = 0}^{\ell}\left| \frac{D^\mu u - D^\mu P}{R^{\ell - \mu}} \right|\chi_{B_{2R}} \\
&\le c(n,m)M_{B_{2R}}^{\ell+1}\left(D^\ell u\right)\chi_{B_{2R}}
\quad
\text{a.e. }
x \in \R^n.
\end{align}
Therefore, \eqref{eq:case1-2} follows from \eqref{eq:case1-6} and \eqref{eq:est-T_{ell,i}}.
 
By \eqref{eq:derivative-onlyhaveto}, \eqref{eq:case1-1} and \eqref{eq:case1-2}, it suffices to show that
\begin{align}\label{eq:derivative-case1}
\fint_{\frac{3}{4}B_j}M_{B_{2R}}^{\ell+1}\left(|D^\ell u|\right)\chi_{B_{2R}}
&\leq c(n,m)\lambda^\frac{1}{\gamma_{p,\ell}}
\quad
\text{in }
\frac{3}{4}B_i,
\end{align}
\begin{align}\label{eq:derivative-case1-a}
&a(x)^{\frac{1}{q}} \fint_{\frac{3}{4}B_j}M_{B_{2R}}^{\ell+1}\left(|D^\ell u|\right)(y)\chi_{B_{2R}}(y)\, dy\\
&\leq c(n,m,p,q,\alpha,[a]_{\alpha},\delta_0)\lambda^\frac{1}{\gamma_{q,\ell}}
\quad
\text{in }
\frac{3}{4}B_i \cap B_{2R}
\end{align}
for any $j \in A_i$. 

We fix any $j \in A_i$. Note that we can take $y_j \in 16B_j \cap E(\lambda)$ by \textbf{\textbf{(W3)}}. We first show \eqref{eq:derivative-case1}. Applying Jensen's inequality iteratively, we get
\begin{align}\label{eq:derivative-case1-1}
    &\fint_{\frac{3}{4}B_j} M_{B_{2R}}^{\ell + 1}\left(|D^\ell u|\right)\chi_{B_{2R}} \\
    &\le  c\fint_{16B_j} M^{\ell+1}\left(|D^\ell u|\chi_{B_{2R}}\right)\chi_{B_{2R}} \\
    &\le  c \left( \fint_{16B_j} M^{\ell+1}\left(|D^\ell u|\chi_{B_{2R}}\right)^{\gamma_{p,\ell}\delta_0}\chi_{B_{2R}} \right)^{\frac{1}{\gamma_{p,\ell}\delta_0}}  \\
    &\le c \left( \fint_{16B_j} M^{\ell+1}\left(|D^\ell u|^{\gamma_{p,\ell}\delta_0}\chi_{B_{2R}}\right)\chi_{B_{2R}} \right)^{\frac{1}{\gamma_{p,\ell}\delta_0}}  \\
    &\le c \left( \fint_{16B_j} g \right)^{\frac{1}{\gamma_{p,\ell}\delta_0}} 
    \leq  c[M(g)(y_j)]^{\frac{1}{\gamma_{p,\ell}\delta_0}} 
    = c[G(y_j)]^{\frac{1}{\gamma_{p,\ell}}}
    \le c(n)\lambda^{\frac{1}{\gamma_{p,\ell}}}.
\end{align}
Here we used \eqref{eq:g_j}, \eqref{eq:G} and \eqref{eq:Elambda}. This implies \eqref{eq:derivative-case1}.

Next we prove \eqref{eq:derivative-case1-a}.
Fix any $x \in \frac{3}{4}B_i \cap B_{2R}$. Then we have
\begin{align}\label{eq:derivative-case1-a-J's}
&a(x)^{\frac{1}{q}} \fint_{\frac{3}{4}B_j}M_{B_{2R}}^{\ell+1}\left(|D^\ell u|\right)(y)\chi_{B_{2R}}(y)\,dy \\
&\le c\fint_{\frac{3}{4}B_j} a(y)^{\frac{1}{q}}M_{B_{2R}}^{\ell+1}\left(|D^\ell u|\right)(y)\chi_{B_{2R}}(y)\,dy  \\
&+ c\fint_{\frac{3}{4}B_j} |x-y|^{\frac{\alpha}{q}}M_{B_{2R}}^{\ell+1}\left(|D^\ell u|\right)(y)\chi_{B_{2R}}(y)\,dy 
\eqqcolon c(q,\alpha,[a]_{\alpha})(J_1 + J_2).
\end{align}

For $J_1$, note that $D^{\ell}u \in L^{\gamma_{p,\ell}\delta_0}(B_{2R})$, $0 < \beta_{\ell} < \min\left\{\frac{n}{\gamma_{p,\ell}\delta_0},\frac{\alpha}{q}\right\}$ and $2R \le 1$ (see \eqref{eq:conv-H_ell-0},\,\eqref{eq:R_0} and \eqref{eq:conv-M_beta-0}, respectively). Then Lemma \ref{lem:Hedberg-2} yields
\begin{align}\label{eq:derivative-case1-a-J_1's}
J_1
&\le c\fint_{\frac{3}{4}B_j} M_{B_{2R}}^{\ell+1}(a^{\frac{1}{q}}|D^\ell u|)(y)\chi_{B_{2R}}(y)\,dy \\
&+ c\fint_{\frac{3}{4}B_j} M_{\beta_\ell}\left(M^{\ell}\left(|D^\ell u|\chi_{B_{2R}}\right)\right)(y)\chi_{B_{2R}}(y)\,dy \\
&\le c\fint_{16B_j} M^{\ell+1}(a^{\frac{1}{q}}|D^\ell u|\chi_{B_{2R}})(y)\chi_{B_{2R}}(y)\,dy \\
&+ c\fint_{16B_j} M_{\beta_\ell}(M^{\ell}(|D^\ell u|\chi_{B_{2R}}))(y)\chi_{B_{2R}}(y)\,dy \\
&\eqqcolon c(n,m,p,q,\alpha,[a]_{\alpha},\delta_0)(J_{11} + J_{12}).
\end{align}
 
For $J_{11}$, we estimate in the same way as \eqref{eq:derivative-case1-1}. Indeed, we obtain
\begin{align}\label{eq:derivative-case1-a-J_11}
J_{11}
&\le  \left( \fint_{16B_j} M^{\ell+1}\left(a^{\frac{1}{q}}|D^\ell u|\chi_{B_{2R}}\right)^{\gamma_{q,\ell}\delta_0} (y)\chi_{B_{2R}}(y)\,dy\right)^{\frac{1}{\gamma_{q,\ell}\delta_0}}  \\
&\le \left( \fint_{16B_j} M^{\ell+1}\left(a^{\frac{\gamma_{q,\ell}\delta_0}{q}}|D^\ell u|^{\gamma_{q,\ell}\delta_0}\chi_{B_{2R}}\right)(y)\chi_{B_{2R}}(y)\,dy \right)^{\frac{1}{\gamma_{q,\ell}\delta_0}}  \\
&\le \left( \fint_{16B_j} g(y)\,dy \right)^{\frac{1}{\gamma_{q,\ell}\delta_0}} \\
&\leq  [M(g)(y_j)]^{\frac{1}{\gamma_{q,\ell}\delta_0}} 
= [G(y_j)]^{\frac{1}{\gamma_{q,\ell}}}
\le \lambda^{\frac{1}{\gamma_{q,\ell}}}
\end{align}
by \eqref{eq:g_j}, \eqref{eq:G} and \eqref{eq:Elambda}.

For $J_{12}$, using Jensen's inequality again, we get
\begin{align}\label{eq:derivative-case1-a-J_12}
J_{12}
&\le \left( \fint_{16B_j} M_{\beta_\ell}(M^{\ell}|D^\ell u|\chi_{B_{2R}})(y)^{\gamma_{q,\ell}\delta_0}\chi_{B_{2R}}(y)\,dy \right)^{\frac{1}{\gamma_{q,\ell}\delta_0}}\\
&\le \left( \fint_{16B_j} F_0(y)^{\delta_0}\chi_{B_{2R}}(y)\,dy \right)^{\frac{1}{\gamma_{q,\ell}\delta_0}}
\le \left( \fint_{16B_j} g(y)\,dy \right)^{\frac{1}{\gamma_{q,\ell}\delta_0}}\\
&\le [M(g)(y_j)]^{\frac{1}{\gamma_{q,\ell}\delta_0}} 
= [G(y_j)]^{\frac{1}{\gamma_{q,\ell}}}
\le \lambda^{\frac{1}{\gamma_{q,\ell}}}.
\end{align}
Here we used \eqref{eq:F_0}, \eqref{eq:g_j}, \eqref{eq:G} and \eqref{eq:Elambda}.
From \eqref{eq:derivative-case1-a-J_1's}, \eqref{eq:derivative-case1-a-J_11} and \eqref{eq:derivative-case1-a-J_12}, we derive
\begin{equation}\label{eq:derivative-case1-a-J_1}
J_1 \le c(n,m,p,q,\alpha,[a]_{\alpha},\delta_0)\lambda^{\frac{1}{\gamma_{q,\ell}}}.
\end{equation}
 
Next we estimate $J_{2}$.
Since $j \in A_i$, we can take $z \in \frac{3}{4}B_j \cap \frac{3}{4}B_i$. Then we have
\begin{align}
|x-y|^{\frac{\alpha}{q}}
&\le (|x-z| + |z-y|)^{\frac{\alpha}{q}} \\
&\le \left( \frac{3}{2}r_i + \frac{3}{2}r_j \right)^{\frac{\alpha}{q}} 
\le c(q,\alpha)r_i^{\frac{\alpha}{q}}
\quad
\text{for any }
y \in \frac{3}{4}B_j
\end{align}
by \textbf{(W4)}.
Therefore,
\begin{align}\label{eq:derivative-case1-a-J_2's}
J_2 
&\le cr_i^{\frac{\alpha}{q}}\fint_{\frac{3}{4}B_j} M_{B_{2R}}^{\ell+1}\left(|D^\ell u|\right)(y)\chi_{B_{2R}}(y)\,dy \\
&\le  c(q,\alpha)r_i^{\frac{\alpha}{q}}\left(\fint_{\frac{3}{4}B_j} M_{B_{2R}}^{\ell+1}\left(|D^\ell u|\right)(y)\chi_{B_{2R}}(y)\,dy \right)^{1-\frac{\gamma_{p,\ell}}{\gamma_{q,\ell}}} \\
&\times \left(\fint_{\frac{3}{4}B_j} M_{B_{2R}}^{\ell+1}\left(|D^\ell u|\right)(y)\chi_{B_{2R}}(y)\,dy \right)^{\frac{\gamma_{p,\ell}}{\gamma_{q,\ell}}}.
\end{align}
For the second factor of the right-hand side in \eqref{eq:derivative-case1-a-J_2's}, we obtain
\begin{equation}\label{eq:derivative-case1-a-J_21}
\left(\fint_{\frac{3}{4}B_j} M_{B_{2R}}^{\ell+1}\left(|D^\ell u|\right)(y)\chi_{B_{2R}}(y)\,dy \right)^{\frac{\gamma_{p,\ell}}{\gamma_{q,\ell}}} 
\le c(n)\lambda^{\frac{1}{\gamma_{q,\ell}}}
\end{equation}
by \eqref{eq:derivative-case1-1}.
Moreover, for the first factor of the right-hand side in \eqref{eq:derivative-case1-a-J_2's}, from Jensen's inequality and \textbf{(W2)}, it follows that
\begin{align}\label{eq:derivative-case1-a-J_22}
&r_i^{\frac{\alpha}{q}}\left(\fint_{\frac{3}{4}B_j} M_{B_{2R}}^{\ell+1}\left(|D^\ell u|\right)(y)\chi_{B_{2R}}(y)\,dy \right)^{1-\frac{\gamma_{p,\ell}}{\gamma_{q,\ell}}} \\
&\le r_i^{\frac{\alpha}{q}}\left(\fint_{\frac{3}{4}B_j} M^{\ell+1}\left(|D^\ell u| \chi_{B_{2R}}\right)(y)^{\gamma_{p,\ell}\delta_0}\chi_{B_{2R}}(y)\,dy \right)^{\frac{1}{\gamma_{p,\ell}\delta_0}-\frac{1}{\gamma_{q,\ell}\delta_0}} \\
&\le cr_i^{\frac{\alpha}{q}-n\left(\frac{1}{\gamma_{p,\ell}\delta_0}-\frac{1}{\gamma_{q,\ell}\delta_0}\right)}\|M^{\ell+1}(|D^\ell u|\chi_{\Omega_0})\|_{L^{\gamma_{p,\ell}\delta_0}(\Omega_0)}^{1-\frac{\gamma_{p,\ell}}{\gamma_{q,\ell}}} \\
&\le cR^{\frac{\alpha}{q}-n\left(\frac{1}{\gamma_{p,\ell}\delta_0}-\frac{1}{\gamma_{q,\ell}\delta_0}\right)}\|M^{\ell+1}(|D^\ell u|\chi_{\Omega_0})\|_{L^{\gamma_{p,\ell}\delta_0}(\Omega_0)}^{1-\frac{\gamma_{p,\ell}}{\gamma_{q,\ell}}} 
\le c(n).
\end{align}
In the last inequality, we used \eqref{eq:R_0}.
Using \eqref{eq:derivative-case1-a-J_2's}, \eqref{eq:derivative-case1-a-J_21} and \eqref{eq:derivative-case1-a-J_22}, we obtain
\begin{align}\label{eq:derivative-case1-a-J_2}
J_2 \le c(n,q,\alpha)\lambda^{\frac{1}{\gamma_{q,\ell}}}.
\end{align}
Hence, we obtain \eqref{eq:derivative-case1-a} from \eqref{eq:derivative-case1-a-J's}, \eqref{eq:derivative-case1-a-J_1}, and \eqref{eq:derivative-case1-a-J_2}. This completes the proof of \textbf{Case 1}.

\noindent \textbf{Case 2 ($\sigma_2 \neq 0$)}:\,since $1 = \sum_{j \in A_i}\psi_j$ in $\tfrac{3}{4}B_i$ by \textbf{(P3)}, we have
$0 = \sum_{j \in A_i}\partial_{\sigma_2}\psi_j$ in $\tfrac{3}{4}B_i$. We also note that $|\sigma_1| < k \le m$ since $\sigma_2 \neq 0$.
Then we have
\begin{align}\label{eq:case2-1}
    \left|\sum_{j \in A_i} \partial_{\sigma_1} \bar{P_j} \partial_{\sigma_2}\psi_j\right|
    &= \left|\sum_{j \in A_i} (\partial_{\sigma_1} \bar{P_j} - \partial_{\sigma_1} Q_i) \partial_{\sigma_2}\psi_j\right| \\
    &\leq c\sum_{j \in A_i} r_j^{-|\sigma_2|}r_i^{\ell - |\sigma_1|}T_{\ell,i} \\
    &\leq c\sup_{j \in A_i} r_j^{\ell-k}T_{\ell,i} 
    \leq c(n,m)R^{\ell - k}T_{\ell,i}
    \quad
    \text{in }
    \frac{3}{4}B_i
\end{align}
by Lemma \ref{lem:Pioscilation}, \textbf{(W4)}, \textbf{(W6)} and \textbf{(W2)}.
Then \eqref{eq:est-T_{ell,i}} yields
\begin{equation}\label{eq:case2-2}
  T_{\ell,i} 
  = \sup_{j \in A_i}\fint_{B_j}|D^\ell v|
  \le c(n,m)\sup_{j \in A_i}\fint_{B_j}M_{B_{2R}}^{\ell+1}\left(|D^\ell u|\right)\chi_{B_{2R}}.
\end{equation}
From \eqref{eq:derivative-onlyhaveto}, \eqref{eq:case2-1} and \eqref{eq:case2-2}, it is sufficient to show that
\begin{align}
\fint_{B_j}M_{B_{2R}}^{\ell+1}\left(|D^\ell u|\right)\chi_{B_{2R}} &\leq c\lambda^\frac{1}{\gamma_{p,\ell}} 
\quad
\text{in }
\frac{3}{4}B_i, \\
a(x)^{\frac{1}{q}} \fint_{B_j} M_{B_{2R}}^{\ell+1}\left(|D^\ell u|\right)(y)\chi_{B_{2R}}(y)\, dy
&\leq c\lambda^\frac{1}{\gamma_{q,\ell}}
\quad
\text{in }
\frac{3}{4}B_i \cap B_{2R}
\end{align}
for any $j \in A_i$.
These two inequalities are obtained by repeating the proofs of
\eqref{eq:derivative-case1} and \eqref{eq:derivative-case1-a},
with $B_j$ in place of $\frac{3}{4}B_j$.
The proof is completed.
\end{proof}

\subsection{Reverse H\"older's inequality}\label{subsec:rh}
The purpose of this subsection is to prove a reverse H\"older inequality. 
We first prove a Caccioppoli inequality. This inequality is obtained by substituting the truncated function $v_{\lambda}$ into \eqref{eq:main}. From this Caccioppoli inequality and the Sobolev--Poincar\'e inequality developed in Section \ref{sec:SP}, we derive the reverse H\"older inequality.

Before the proof, we recall some formulas on level sets.
Let $B \subset \mathbb{R}^n$ be a measurable set 
and $h \in L^1(B)$ be a nonnegative function.
If $r > 0$, then it holds that
\begin{equation}\label{eq:lcr-1}
    h(x)^r = r\int_{0}^{\infty}\mu^{r - 1} \chi_{\{y \in B : h(y) > \mu\}}(x)\ d\mu
    \quad
    \text{for a.e.\,}
    x \in B.
\end{equation}
On the other hand, if $r < 0$, then we have
\begin{equation}\label{eq:lcr-2}
    h(x)^r = -r\int_{0}^{\infty}\mu^{r - 1} \chi_{\{y \in B: h(y) \leq \mu\}}(x)\ d\mu
    \quad
    \text{for a.e.\,}
    x \in B
    \text{ with }
    h(x) > 0.
\end{equation}

Then we have the following Caccioppoli inequality.
\begin{lem}\label{lem:Caccioppoli}
There exist constants $c>0$ and $\delta_1 \in \left[ \frac{1+\delta_0}{2},1 \right)$ such that if $\delta \in [\delta_1,1)$, then it holds
\begin{align}\label{eq:Caccioppoli}
    &\fint_{B_{R}} H_m (x, D^{m} u)^\delta \\
    &\leq \frac{1}{2}\fint_{B_{2R}} H_m (x, D^{m} u)^\delta
    + c \sum_{\ell = 0}^{m - 1}\fint_{B_{2R}}  H_m\left(x, \frac{D^{\ell} u - D^{\ell}P}{R^{m-\ell}}\right)^{\delta}
    + c\fint_{B_{2R}} F^\delta.
\end{align}
Here $c$ and $\delta_1$ depend only on $n,m,p,q,\alpha,[a]_{\alpha},\nu$ and $\delta_0$.
\end{lem}

\begin{proof}
By Lemma\,\ref{lem:derivative} and \eqref{eq:Elambda}, 
we have $v_{\lambda} \in W^{m,\infty}_0(B_{2R})$.
Therefore, one can substitute $v_{\lambda}$
as a test function of \eqref{eq:main}.
We obtain
\begin{align}\label{eq:step1-1}
\sum_{\sigma \in S} \int_{B_{2R} \cap E(\lambda)} A_\sigma \cdot \partial_\sigma v
&= -\sum_{\sigma \in S} \int_{B_{2R} \cap E(\lambda)^c} A_\sigma \cdot \partial_\sigma v_{\lambda}.
\end{align}
Here we write $A_{\sigma} = A_\sigma(x,u,\ldots,D^m u)$ for brevity.
By \eqref{eq:hats-m} and \eqref{eq:G}, we have
\begin{gather}\label{eq:usefulful}
|D^m u| \leq G^{\frac{1}{\gamma_{p,m}}}=G^{\frac{1}{p}},\
a(x)^{\frac{1}{q}}|D^m u| \leq  G^{\frac{1}{\gamma_{q,m}}}=G^{\frac{1}{q}},\\
\ h_{r,\ell} \leq G^{\frac{1}{\widehat{t_{r,\ell}}}},\
g_{r,\ell} \leq G^{\frac{1}{\widehat{s_{r,\ell}}}}
\end{gather}
for a.e.\,$x \in B_{2R}$ and each $\ell \in \{0,\ldots,m\}$ and $r \in \{p,q\}$. Here note that we consider $s_{r,m} = \infty$ and $g_{r,m}\equiv1$ by the remark following \eqref{def:exponents}.
Then, for each $\sigma \in S_{\ell}$ and each $\ell \in \{0,\ldots,m\}$, the growth condition \eqref{eq:growth}, \eqref{eq:eta} and \eqref{eq:usefulful} imply
\begin{align}\label{eq:similar}
\left| A_\sigma \cdot \partial_\sigma v_{\lambda}\right|
&\leq c( g_{p,\ell}|D^m u|^{p - 1} + h_{p,\ell}) \left|D^{\ell}v_{\lambda} \right| \\
&+ c( g_{q,\ell}a(x)^{\frac{q-1}{q}}|D^m u|^{q-1} + h_{q,\ell} )a(x)^\frac{1}{q}\left|D^{\ell}v_{\lambda} \right| \\
&\le c( G^{\frac{1}{\widehat{s_{p,\ell}}} + \frac{1}{p'}} + G^{\frac{1}{\widehat{t_{p,\ell}}}} ) \lambda^{\frac{1}{\gamma_{p,\ell}}}
+ c(G^{\frac{1}{\widehat{s_{q,\ell}}} + \frac{1}{q'}} + G^{\frac{1}{\widehat{t_{q,\ell}}}})\lambda^{\frac{1}{\gamma_{q,\ell}}} \\
&= cG^{1-\frac{1}{\gamma_{p,\ell}}}\lambda^{\frac{1}{\gamma_{p,\ell}}}
+ cG^{1-\frac{1}{\gamma_{q,\ell}}}\lambda^{\frac{1}{\gamma_{q,\ell}}}
\quad
\text{in }B_{2R} \cap E(\lambda)^c.
\end{align}
Here we also used Lemma \ref{lem:derivative} and \eqref{eq:Holder}. Therefore, from \eqref{eq:step1-1} and \eqref{eq:similar}, we obtain
\begin{align}\label{eq:Caccippoli-last}
&\sum_{\sigma \in S} \int_{B_{2R} \cap E(\lambda)} A_\sigma \cdot \partial_\sigma v  \\
&\le c(m)\sum_{\ell=0}^{m} \int_{B_{2R} \cap E(\lambda)^c} \left( G^{1-\frac{1}{\gamma_{p,\ell}}}\lambda^{\frac{1}{\gamma_{p,\ell}}}
+ G^{1-\frac{1}{\gamma_{q,\ell}}}\lambda^{\frac{1}{\gamma_{q,\ell}}} \right)
\end{align}
for any $\lambda \in (\Lambda_0,\infty)$.

Multiplying both sides of \eqref{eq:Caccippoli-last} by $(1-\delta)\lambda^{\delta-2}$ and integrating from $\Lambda_0$ to $\infty$, we obtain
\begin{align}\label{eq:Caccioppoli-I's}
I_1
&\coloneqq \sum_{\sigma \in S}(1 - \delta)\int_{\Lambda_0}^{\infty} \lambda^{\delta - 2} \left( \int_{B_{2R}\cap E(\lambda)} A_\sigma \cdot \partial_\sigma v  \,dx\right)\ d\lambda \\
&\le c(m)\sum_{\ell=0}^{m} (1 - \delta)\int_{\Lambda_0}^{\infty} \lambda^{\delta - 2} \\
&\times \left\{ \int_{B_{2R}\cap E(\lambda)^c} \left( G^{1-\frac{1}{\gamma_{p,\ell}}}\lambda^{\frac{1}{\gamma_{p,\ell}}}
+ G^{1-\frac{1}{\gamma_{q,\ell}}}\lambda^{\frac{1}{\gamma_{q,\ell}}} \right) \,dx \right\}\ d\lambda 
\eqqcolon I_2.
\end{align}

First, we transform $I_1$.
Let us denote $G_{\Lambda_0} = \max\{G,\Lambda_0\}$.
Note that
\begin{align}\label{eq:useful-formula-2}
\left|\frac{D^{k} u - D^{k}P}{R^{\ell - k}}\right| &\leq c(n,m)G^{\frac{1}{\gamma_{p,\ell}}}, \\
\, a(x)^{\frac{1}{q}}\left|\frac{D^{k} u - D^{k}P}{R^{\ell - k}}\right| &\leq c(n,m,p,q,\alpha,[a]_{\alpha})G^{\frac{1}{\gamma_{q,\ell}}}
\end{align}
for a.e.\,$x \in B_{2R}$, each $\ell \in \{0,\ldots,m\}$ and $k \in \{0,\ldots,\ell\}$.
Indeed, fix $\ell \in \{0,\ldots,m\}$ and $k \in \{0,\ldots,\ell\}$.
By Lemma \ref{lem:useful-formula} we have
\begin{equation}
\left|\frac{D^{k} u(x) - D^{k}P(x)}{R^{\ell - k}}\right| \le c(n,m)M_{B_{2R}}^{\ell+1}(|D^\ell u|)(x)
\quad
\text{for a.e.\,}
x \in B_{2R}.
\end{equation}
Then, Jensen's inequality implies
\begin{align}
 M_{B_{2R}}^{\ell+1}(|D^\ell u|)(x)
 &\le M^{\ell+1}(|D^\ell u|\chi_{B_{2R}})(x) \\
 &\le [M^{\ell+1}(|D^\ell u|^{\gamma_{p,\ell}\delta_0}\chi_{B_{2R}})(x)]^{\frac{1}{\gamma_{p,\ell}\delta_0}} \\
 &\le g(x)^{\frac{1}{\gamma_{p,\ell}\delta_0}}
 \le [M(g)(x)]^{\frac{1}{\gamma_{p,\ell}\delta_0}}
 \le G(x)^{\frac{1}{\gamma_{p,\ell}}}
 \quad
\text{for a.e.\,}
x \in B_{2R}.
\end{align}
Here we used \eqref{eq:g_j} and \eqref{eq:G}.

Next we prove the second inequality in \eqref{eq:useful-formula-2}. By \eqref{eq:R_0} and \eqref{eq:conv-M_beta-0}, we have $2R \le 1$, $\beta_{\ell} \in \left( 0, \min\left\{\frac{n}{\gamma_{p,\ell}\delta_0},\frac{\alpha}{q}\right\} \right)$
and $D^\ell u \in L^{\gamma_{p,\ell}\delta_0}(B_{2R})$. Thus, we apply Lemma \ref{lem:Hedberg-2} to obtain
\begin{align}
&a(x)^{\frac{1}{q}}M_{B_{2R}}^{\ell+1}(|D^\ell u|)(x) \\
&\le cM^{\ell+1}(a^{\frac{1}{q}}|D^\ell u|\chi_{B_{2R}})(x)
+ cM_{\beta_\ell}(M^{\ell+1}(|D^\ell u|\chi_{B_{2R}}))(x) \\
&\le c[M^{\ell+1}(a^{\frac{\gamma_{q,\ell}\delta_0}{q}}|D^\ell u|^{\gamma_{q,\ell}\delta_0}\chi_{B_{2R}})(x)]^{\frac{1}{\gamma_{q,\ell}\delta_0}}
+ cM_{\beta_\ell}(M^{\ell+1}(|D^\ell u|\chi_{B_{2R}}))(x) \\
&\le  cg(x)^{\frac{1}{\gamma_{q,\ell}\delta_0}} \le cG(x)^{\frac{1}{\gamma_{q,\ell}}}
\quad
\text{for a.e.\,}
x \in B_{2R}.
\end{align}
We also used \eqref{eq:g_j} and \eqref{eq:G}.

By \eqref{eq:useful-formula-2}, we have
\begin{align}
|\partial_{\sigma}v| \le c(n,m)G(x)^{\frac{1}{\gamma_{p,\ell}}},\quad
a(x)^{\frac{1}{q}}|\partial_{\sigma}v| \le c(n,m,p,q,\alpha,[a]_{\alpha})G(x)^{\frac{1}{\gamma_{q,\ell}}}
\end{align}
for each $\sigma \in S_{\ell}$ and $\ell \in \{0,\ldots,m\}$.
A calculation similar to that in \eqref{eq:similar} implies $|A_\sigma\partial_\sigma v| \le cG$ for each $\sigma \in S_{\ell}$ and $\ell \in \{0,\ldots,m\}$.
It follows that $GG_{\Lambda_0}^{\delta - 1} \le G^{\delta} \in L^1(B_{2R})$ from $0 < \delta < 1$, $\Lambda_0 \ge 1$ and Lemma\,\ref{lem:G}. 
Then, using Fubini's theorem and \eqref{eq:lcr-2} (with $h = G_{\Lambda_0}$), we obtain
\begin{align}\label{eq:Caccioppoli-I_1}
&I_1 \\
&= \sum_{\sigma \in S}\int_{B_{2R}} A_\sigma \cdot \partial_\sigma v \left\{  (1-\delta)\int_{\Lambda_0}^{\infty} \lambda^{\delta - 2}\chi_{\{y \in B_{2R} : G(y) \leq \lambda\}}(x) \ d\lambda \right\}\,dx \\
    &= \sum_{\sigma \in S}\int_{B_{2R}} A_\sigma \cdot \partial_\sigma v \left\{  (1-\delta)\int_{0}^{\infty} \lambda^{\delta - 2}\chi_{\{y \in B_{2R} : G_{\Lambda_0}(y) \leq \lambda\}}(x) \ d\lambda \right\}\ dx \\
    &= \sum_{\sigma \in S}\int_{B_{2R}} A_\sigma \cdot \partial_\sigma v\cdot G_{\Lambda_0}^{\delta - 1} \,dx.
\end{align}

Next, we estimate $I_2$. Here we may assume 
\begin{equation}\label{eq:Caccioppoli-delta0}
\delta_0 - 1 +\frac{1}{\gamma_{r,\ell}} \ge 1 - \delta_0
\end{equation}
for each $\ell \in \{0,\ldots,m\}$ and $r \in \{p,q\}$.
Therefore, for each $\ell \in \{0,\ldots,m\}$, 
Fubini's Theorem and \eqref{eq:lcr-1} (with $h = G$ and $r = \delta +\frac{1}{\gamma_{p,\ell}} - 1$) yield
\begin{align}\label{eq:Caccioppoli-I_{2,p,ell}}
    I_{2,p,\ell}
    &\coloneqq (1 - \delta)\int_{\Lambda_0}^{\infty} \lambda^{\delta +\frac{1}{\gamma_{p,\ell}} - 2} \left(  \int_{B_{2R}\cap E(\lambda)^c} G^{1 - \frac{1}{\gamma_{p,\ell}}} \ dx \right)\ d\lambda \\
    &\leq (1 - \delta)\int_{0}^{\infty} \lambda^{\delta +\frac{1}{\gamma_{p,\ell}} - 2} \left(  \int_{B_{2R}\cap E(\lambda)^c} G^{1 - \frac{1}{\gamma_{p,\ell}}} \ dx \right)\ d\lambda \\
    &= \int_{B_{2R}} G^{1 - \frac{1}{\gamma_{p,\ell}}} \left\{  (1-\delta)\int_{0}^{\infty} \lambda^{\delta +\frac{1}{\gamma_{p,\ell}} - 2}\chi_{\{y \in B_{2R} : G(y) > \lambda\}} \ d\lambda \right\}\ dx \\
    &= \frac{1-\delta}{\delta - 1 +\frac{1}{\gamma_{p,\ell}}} \int_{B_{2R}} G^{1 - \frac{1}{\gamma_{p,\ell}}}\cdot G^{\delta - 1 + \frac{1}{\gamma_{p,\ell}}}
   \le c(\delta_0)(1-\delta)\int_{B_{2R}}G^{\delta}.
\end{align}
In the last inequality, we used \eqref{eq:Caccioppoli-delta0}.
Similarly, we have
\begin{align}\label{eq:Caccioppoli-I_{2,q,ell}}
 I_{2,q,\ell} 
 &\coloneqq (1 - \delta)\int_{\Lambda_0}^{\infty} \lambda^{\delta +\frac{1}{\gamma_{q,\ell}} - 2} \left(  \int_{B_{2R}\cap E(\lambda)^c} G^{1 - \frac{1}{\gamma_{q,\ell}}} \ dx \right)\ d\lambda \\
 &\le c(\delta_0)(1-\delta)\int_{B_{2R}} G^{\delta}.
\end{align}
From \eqref{eq:Caccioppoli-I's}, \eqref{eq:Caccioppoli-I_{2,p,ell}} and \eqref{eq:Caccioppoli-I_{2,q,ell}}, we obtain 
\begin{align}\label{eq:Caccioppoli-I_2}
I_2 
= c(m)\sum_{\ell=0}^{m} (I_{2,p,\ell} + I_{2,q,\ell})
\le c(m,\delta_0)(1-\delta)\sum_{\ell=0}^{m}\int_{B_{2R}} G^{\delta}.
\end{align}
Therefore, \eqref{eq:Caccioppoli-I's}, \eqref{eq:Caccioppoli-I_1} and \eqref{eq:Caccioppoli-I_2} yield that
\begin{equation}\label{eq:Caccioppoli-last-2}
\sum_{\sigma \in S} \int_{B_{2R}} A_\sigma \cdot \partial_\sigma v\cdot G_{\Lambda_0}^{\delta - 1}
\le c\cdot(1-\delta)\int_{B_{2R}} G^{\delta}.
\end{equation}

We derive the conclusion from \eqref{eq:Caccioppoli-last-2}.
Now we divide the left-hand side of \eqref{eq:Caccioppoli-last-2} in the following way:
 \begin{align}
    &\sum_{\sigma \in S} \int_{B_{2R}} A_\sigma \cdot \partial_\sigma v\cdot G_{\Lambda_0}^{\delta-1} \\
    &= \sum_{\sigma \in S_m} \int_{B_{2R}} A_\sigma\cdot \partial_\sigma u \eta\cdot G_{\Lambda_0}^{\delta-1} \\
    &+ \sum_{\sigma \in S_m} \int_{B_{2R}} A_\sigma (\partial_\sigma v - \partial_\sigma u \eta)G_{\Lambda_0}^{\delta-1}\\
    &+\sum_{\sigma \in S \setminus S_m} \int_{B_{2R}} A_\sigma \partial_\sigma v\cdot G_{\Lambda_0}^{\delta-1}
    \eqqcolon J_0 - J_1 - J_2.
 \end{align}
By this equality and \eqref{eq:Caccioppoli-last-2}, we have
\begin{equation}\label{eq:lem_substitute_J}
J_0 \le J_1 + J_2 + c \cdot (1-\delta)\int_{B_{2R}}  G^{\delta}.
\end{equation}
(In order to obtain \eqref{eq:lem_substitute_J}, the boundedness of $J_1$ and $J_2$ is needed. 
The boundedness of $J_1$ and $J_2$ follows from \eqref{eq:lem_substitute_J1} and \eqref{eq:lem_substitute_J2}.)

First we provide a lower estimate of $J_0$. The coercivity \eqref{eq:coercivity} implies
\begin{align}\label{eq:lem_substitute_J0}
   J_0
   &\geq \nu \left( \int_{B_{2R}} H_m(x,D^m u) \eta G_{\Lambda_0}^{\delta - 1} -  \int_{B_{2R}} ( f_p + a f_q )\eta G_{\Lambda_0}^{\delta-1} \right) \\
   &\geq \nu \int_{B_{R}} H_m(x,D^m u)G_{\Lambda_0}^{\delta-1} - \nu \int_{B_{2R}} (f_p + af_q)^{\delta} \\
  &\ge \nu \int_{B_{R}} H_m(x,D^m u)G_{\Lambda_0}^{\delta-1} - \nu \int_{B_{2R}} F^{\delta}.
\end{align}
Here we used \eqref{eq:eta}, $\delta - 1 < 0$, 
$f_p + a(x)f_q \le G$ and $f_p + a(x)f_q \le F$.

Next, we derive an upper estimate of $J_1$.
By \eqref{eq:v} and \eqref{eq:P}, we have 
\begin{align}\label{eq:veta-0}
    |\partial_\sigma v - \partial_\sigma u \eta| 
    &= |\partial_\sigma((u-P)\eta) - \partial_\sigma u \eta| \\
    &= \left|\sum_{\sigma_1 + \sigma_2 = \sigma}\binom{\sigma}{\sigma_1}\partial_{\sigma_1}(u-P)\partial_{\sigma_2}\eta - \partial_\sigma u \eta\right| \\
    &= \left|\sum_{\sigma_1 + \sigma_2 = \sigma, \sigma_1 < \sigma}\binom{\sigma}{\sigma_1}\partial_{\sigma_1}(u-P)\partial_{\sigma_2}\eta\right| \\
    &\leq c(n,m)\sum_{\ell = 0}^{m - 1} \left|\frac{D^{\ell} u - D^{\ell} P}{R^{m - \ell}}\right| \quad \text{in }B_{2R}
\end{align}
for each $\sigma \in S_m$ by \eqref{eq:eta} and \eqref{eq:v}.
Using \eqref{eq:veta-0} and the growth condition $\eqref{eq:growth}$, we obtain
\begin{align}\label{eq:Caccioppoli-J_1}
   J_1 
    &\leq c \sum_{\sigma \in S_m}\int_{B_{2R}} \{ (|D^m u|^{p - 1} + h_{p,m}) \\
    &+ a^{\frac{1}{q}}(a^{\frac{q-1}{q}}|D^m u|^{q - 1} + h_{q,m}) \}|\partial_\sigma v - \partial_\sigma u \eta|G_{\Lambda_0}^{\delta-1}\,dx \\
    &\le c \int_{B_{2R}} \{ (|D^m u|^{p - 1} + h_{p,m}) \\
    &+ a^{\frac{1}{q}}(a^{\frac{q-1}{q}}|D^m u|^{q - 1} + h_{q,m}) \}\sum_{\ell = 0}^{m - 1} \left|\frac{D^{\ell} u - D^{\ell} P}{R^{m - \ell}}\right| G_{\Lambda_0}^{\delta - 1}\,dx \\
    &= c(n,m) \int_{B_{2R}} I G_{\Lambda_0}^{\delta - 1}.
\end{align}
Here we set
\begin{equation}
  I \coloneqq \{ (|D^m u|^{p - 1} + h_{p,m}) 
    + a^{\frac{1}{q}}(a^{\frac{q-1}{q}}|D^m u|^{q - 1} + h_{q,m}) \}\sum_{\ell = 0}^{m - 1} \left|\frac{D^{\ell} u - D^{\ell} P}{R^{m - \ell}}\right|.
\end{equation}
\eqref{eq:usefulful} and Young's inequality yield for any $\varepsilon > 0$,
\begin{align}\label{eq:Caccioppoli-J_1-I}
I 
&= ( |D^m u|^{p - 1} + h_{p,m}) \cdot \sum_{\ell = 0}^{m-1} \left|\frac{D^{\ell} u - D^{\ell}P}{R^{m - \ell}}\right| \\
&+ (a(x)^{\frac{q-1}{q}}|D^m u|^{q-1} + h_{q,m}) \cdot \sum_{\ell = 0}^{m-1}a(x)^\frac{1}{q}\left|\frac{D^{\ell} u - D^{\ell}P}{R^{m - \ell}}\right| \\
&\leq 2G^{\frac{1}{p'}} \cdot \sum_{\ell = 0}^{m-1} \left|\frac{D^{\ell} u - D^{\ell}P}{R^{m - \ell}}\right|+ 2
G^{\frac{1}{q'}} \cdot \sum_{\ell = 0}^{m-1}a(x)^\frac{1}{q}\left|\frac{D^{\ell} u - D^{\ell}P}{R^{m - \ell}}\right| \\
&\le \varepsilon G + c(\varepsilon)\sum_{\ell = 0}^{m-1} \left( \left|\frac{D^{\ell} u - D^{\ell}P}{R^{m - \ell}}\right|^p + a(x) \left|\frac{D^{\ell} u - D^{\ell}P}{R^{m - \ell}}\right|^q  \right) \\
&= \varepsilon G + c(\varepsilon)\sum_{\ell = 0}^{m-1} H_m \left(x, \frac{D^{\ell} u - D^{\ell}P}{R^{m - \ell}} \right).
\end{align}

Combining \eqref{eq:Caccioppoli-J_1} and \eqref{eq:Caccioppoli-J_1-I}, we obtain 
\begin{align}\label{eq:lem_substitute_J1}
J_1
&\le \varepsilon \int_{B_{2R}} G \cdot G_{\Lambda_0}^{\delta-1} + c(\varepsilon)\sum_{\ell = 0}^{m-1}\int_{B_{2R}} H_m \left(x, \frac{D^{\ell} u - D^{\ell}P}{R^{m - \ell}} \right) \cdot G_{\Lambda_0}^{\delta-1} \\
&\le \varepsilon \int_{B_{2R}} G^{\delta} + c(\varepsilon)\sum_{\ell = 0}^{m-1}\int_{B_{2R}} H_m \left(x, \frac{D^{\ell} u - D^{\ell}P}{R^{m - \ell}} \right)^{\delta}.
\end{align}
In the last inequality, we used $\delta - 1 < 0$ and \eqref{eq:useful-formula-2}.

Next we evaluate $J_2$. Observe that for each $\ell \in \{0,\ldots,m\}$
\begin{equation}\label{eq:veta}
    |D^{\ell}v|
    =|D^{\ell}((u-P)\eta)|
     \leq c(n,m)\sum_{k = 0}^{\ell} \left|\frac{D^k u - D^k P}{R^{\ell - k}}\right|
\end{equation}
by \eqref{eq:v}.
Using this inequality and the growth condition \eqref{eq:growth}, we have
\begin{align}\label{eq:Caccioppoli-J_2}
    J_2
    &\leq  c\int_{B_{2R}} \sum_{\ell = 0}^{m - 1} \{ (g_{p,\ell}|D^m u|^{p - 1} + h_{p,\ell}) \\
    &+ a^{\frac{1}{q}}(g_{q,\ell}a^{\frac{q-1}{q}}|D^m u|^{q - 1} + h_{q,\ell}) \}|D^{\ell}v|G_{\Lambda_0}^{\delta-1} \\
    &\leq c\int_{B_{2R}} \sum_{\ell = 0}^{m - 1}\sum_{k = 0}^{\ell} \{ (g_{p,\ell}|D^m u|^{p - 1} + h_{p,\ell}) \\
    &+ a^{\frac{1}{q}}(g_{q,\ell}a^{\frac{q-1}{q}}|D^m u|^{q - 1} + h_{q,\ell}) \}\left|\frac{D^k u - D^k P}{R^{\ell - k}}\right|G_{\Lambda_0}^{\delta-1} \\
    &\eqqcolon c\int_{B_{2R}} \sum_{\ell = 0}^{m - 1}\sum_{k = 0}^{\ell}I_{\ell,k}G_{\Lambda_0}^{\delta-1}.
\end{align}
From \eqref{eq:Holder}, \eqref{eq:useful-formula-2} and Young's inequality, it follows that for any $\varepsilon > 0$,
\begin{align}\label{eq:Caccioppoli-J_2-I_{ell,k}}
I_{\ell,k}
&= g_{p,\ell} \cdot |D^m u|^{p - 1}\left|\frac{D^{k} u - D^{k}P}{R^{\ell - k}}\right| \\
&+ g_{q,\ell} \cdot a(x)^{\frac{q-1}{q}}|D^m u|^{q-1}a(x)^\frac{1}{q}\left|\frac{D^{k} u - D^{k}P}{R^{\ell - k}}\right| \\
    &+ h_{p,\ell} \cdot \left|\frac{D^{k} u - D^{k}P}{R^{\ell - k}}\right|
    +  h_{q,\ell} \cdot a^{\frac{1}{q}}\left|\frac{D^{k} u - D^{k}P}{R^{\ell - k}}\right| \\
    &\leq g_{p,\ell} \cdot G^{\frac{1}{p'} + \frac{1}{\gamma_{p,\ell}}} 
    + g_{q,\ell} \cdot G^{\frac{1}{q'} + \frac{1}{\gamma_{q,\ell}}}
    + h_{p,\ell} \cdot G^{\frac{1}{\gamma_{p,\ell}}}
    + h_{q,\ell} \cdot G^{\frac{1}{\gamma_{q,\ell}}} \\
    &\leq \varepsilon G + c(\varepsilon) \left( g_{p,\ell}^{\widehat{s_{p,\ell}}} + g_{q,\ell}^{\widehat{s_{q,\ell}}} + h_{p,\ell}^{\widehat{t_{p,\ell}}} + h_{q,\ell}^{\widehat{t_{q,\ell}}} \right)
\end{align}
for each $\ell \in \{0,\ldots,m-1\}$ and $k \in \{0,\ldots,\ell\}$. Therefore, \eqref{eq:Caccioppoli-J_2}, \eqref{eq:Caccioppoli-J_2-I_{ell,k}}, \eqref{eq:F_0} and \eqref{eq:F} yield
\begin{align}\label{eq:lem_substitute_J2}
	J_2 
	&\le \varepsilon \int_{B_{2R}} G \cdot G_{\Lambda_0}^{\delta-1}
    + c(\varepsilon) \int_{B_{2R}} F_0 \cdot G_{\Lambda_0}^{\delta-1}\\
	&\le \varepsilon \int_{B_{2R}} G^{\delta}
    + c(\varepsilon) \int_{B_{2R}} F_0^{\delta}
    \le \varepsilon \int_{B_{2R}} G^{\delta}
    + c(\varepsilon) \int_{B_{2R}} F^{\delta}
\end{align}
for any $\varepsilon > 0$. 
Here note that $\delta-1 < 0$, $F_0 \le G_{\Lambda_0}$ and $F_0 \le F$.

Combining \eqref{eq:lem_substitute_J}, \eqref{eq:lem_substitute_J0}, \eqref{eq:lem_substitute_J1} and \eqref{eq:lem_substitute_J2}, we obtain for any $\varepsilon > 0$,
\begin{align}\label{eq:Caccioppoli-nc}
    &\int_{B_{R}} H_m(x,D^m u)G_{\Lambda_0}^{\delta-1} \\
    &\leq c\{ \varepsilon + (1-\delta) \} \int_{B_{2R}} G^\delta \\
    &+ c(\varepsilon) \sum_{\ell = 0}^{m-1} \int_{B_{2R}}H_m \left(x, \frac{D^{\ell} u - D^{\ell}P}{R^{m - \ell}} \right)^{\delta} 
    + c(\varepsilon)\int_{B_{2R}} F^{\delta} \\
    &\le c\{ \varepsilon + (1-\delta) \} \int_{B_{2R}} G^\delta + c(\varepsilon)|B_{2R}|K,
\end{align}
where we set
\begin{equation}\label{eq:Caccioppoli-K}
  K \coloneqq \sum_{\ell = 0}^{m-1} \fint_{B_{2R}}H_m \left(x, \frac{D^{\ell} u - D^{\ell}P}{R^{m - \ell}} \right)^{\delta} + \fint_{B_{2R}} F^{\delta}.
\end{equation}
In order to estimate the left-hand side of \eqref{eq:Caccioppoli-nc}, following the proof of \cite[Lemma 3.9]{BBK},  let us define
\begin{equation}
    U_R \coloneqq \{x \in B_R : H_m(x,D^m u) \leq (1-\delta) G_{\Lambda_0}(x) \}.
\end{equation}
Then we have
\begin{align}
    \int_{B_{R}} H_m(x,D^m u)^{\delta}  
    &= \int_{U_R} H_m(x,D^m u)^{\delta} + \int_{B_{R} \setminus U_R} H_m(x,D^m u)\left[ H_m(x,D^m u) \right]^{\delta - 1} \\
    &\leq (1-\delta)^\delta \int_{U_R} G_{\Lambda_{0}}^\delta + (1-\delta)^{\delta - 1}\int_{B_{R} \setminus U_R} H_m(x,D^m u)G_{\Lambda_0}^{\delta - 1} \\
    &\leq  (1-\delta)^\delta \int_{B_{2R}} G_{\Lambda_{0}}^\delta + e^{\frac{1}{e}}\int_{B_{R}} H_m(x,D^m u)G_{\Lambda_0}^{\delta - 1}.
\end{align}
Here we used $\sup_{0<t<1}(1 - t)^{t - 1} = e^{\frac{1}{e}}$.
From this inequality, \eqref{eq:Caccioppoli-nc}, \eqref{eq:Lambda0}, Lemma \ref{lem:G} and \eqref{eq:Caccioppoli-K}, we derive for any $\varepsilon > 0$,
\begin{align}
    \fint_{B_{R}} H_m(x,D^m u)^{\delta}
    &\leq c (1-\delta)^\delta \fint_{B_{2R}} G_{\Lambda_{0}}^\delta
    + e^{\frac{1}{e}}\fint_{B_{R}} H_m(x,D^m u)G_{\Lambda_0}^{\delta - 1}  \\ 
    &\leq  c\left( (1-\delta)^\delta + \varepsilon + (1-\delta) \right)\fint_{B_{2R}} G_{\Lambda_{0}}^\delta
    + c(\varepsilon) K \\
    &\le c\left( (1-\delta)^\delta + \varepsilon + (1-\delta) \right) \left( \fint_{B_{2R}} G^\delta + \Lambda_0^{\delta} \right)
    + c(\varepsilon) K \\
    &\le c\left( (1-\delta)^\delta + \varepsilon + (1-\delta) \right)\fint_{B_{2R}} H_m(x,D^m u)^\delta
    + c(\varepsilon) K.
\end{align}
Note that $\lim_{t \rightarrow 1}(1 - t)^t = 0$.
Then we can choose $\varepsilon>0$ and $\delta_1 \in (0,1)$ such that the conclusion \eqref{eq:Caccioppoli} holds for any $\delta \in [\delta_1,1)$. The proof is completed.
\end{proof}

Using Corollary \ref{cor:sp}, we estimate the second term in the right hand side of the Caccioppoli inequality \eqref{eq:Caccioppoli}.
\begin{lem}\label{lem:SP}
There exist constants $c=c(n,m,p,q,\alpha,[a]_{\alpha})>0$ and $\widehat{\delta} = \widehat{\delta}(n,p,q) \in (0,1)$ such that
\begin{equation}
\sum_{\ell=0}^{m-1}\fint_{B_{2R}} H_m \left(x,\frac{D^\ell u - D^\ell P}{R^{m-\ell}} \right)
\le c\left(\fint_{B_{2R}} H_m(x,D^m u)^{\widehat{\delta}} \right)^{\frac{1}{\widehat{\delta}}}.
\end{equation}
\end{lem}

\begin{proof}
Set $p_* = \max \left\{1,\frac{np}{n+p}\right\}$ and $q_*  = \max \left\{1,\frac{nq}{n+q}\right\}$.
For each $\ell \in \{0,\ldots,m-1\}$, we have
\begin{equation}
D^\ell u \in L^p(\Omega_0),
\quad
a^{\frac{1}{q}}D^\ell u \in L^q(\Omega_0),
\quad
\frac{q}{p} < 1 + \frac{\alpha}{n} 
\end{equation}
by \eqref{eq:conv-H_ell-0} and \eqref{assmpt:exponents}.
Fix $\ell \in \{0,\ldots,m-1\}$. Recalling \eqref{eq:P}, we apply \eqref{eq:cor-sp-p} and the standard $(p,p)$-Poincar\'e inequality iteratively to obtain
\begin{align}\label{eq:SP-J's}
&\fint_{B_{2R}} H_m \left(x,\frac{D^\ell u - D^\ell P}{R^{m-\ell}} \right) \\
&= \fint_{B_{2R}} \left|\frac{D^\ell u -D^\ell P}{R^{m-\ell}}\right|^p
+ \fint_{B_{2R}} a\left|\frac{D^\ell u -D^\ell P}{R^{m-\ell}}\right|^q \\
&\le c\fint_{B_{2R}} \left|\frac{D^{m-1} u -D^{m-1} P}{R}\right|^p
+ c\fint_{B_{2R}} a\left|\frac{D^{m-1} u -D^{m-1} P}{R}\right|^q \\
&+ cR^{\alpha} \left(\fint_{B_{2R}} \left|\frac{D^{m-1} u -D^{m-1} P}{R}\right|^p \right)^{\frac{q}{p}} \\
&\eqqcolon c(n,m,p,q,\alpha)\left(J_1+J_2+J_3\right).
\end{align}

For convenience, we first estimate $J_2$. 
Let us define $(\widehat{p_*},\widehat{q_*})$ in order to apply Corollary \ref{cor:sp} with parameters $(p,q,a,\alpha)=(\widehat{p_*},\widehat{q_*},a^{\frac{\widehat{q_*}}{q}},\frac{\alpha \widehat{q_*}}{q})$ for $J_2$ (see Remark \ref{rem:sp-p} also). 
If $(q_* \ge)\,p_* > 1$, we simply set $(\widehat{p_*},\widehat{q_*}) = (p_*,q_*)$. 
If $q_* > 1$ and $p_*=1$, then we define $(\widehat{p_*},\widehat{q_*}) =  \left( \min\left\{ \frac{1+p}{2}, q_*\right\},q_* \right)$. 
Finally, if $q_*\,(=p_*)=1$, then we set $(\widehat{p_*},\widehat{q_*}) = \left(\frac{1+p}{2},\min\left\{ \frac{1+q}{2},\left(1+\frac{\alpha}{nq}\right)\frac{1+p}{2} \right\}\right)$. From this definition, we have
\begin{equation}\label{eq:SP-expo-1}
1 < \widehat{p_*} \le \max \left\{\frac{1+p}{2}, p_* \right\} < p,
\quad
\widehat{p_*} \le \widehat{q_*} \le \max \left\{\frac{1+q}{2}, q_* \right\} < q,
\end{equation}
and
\begin{equation}\label{eq:SP-expo-2}
\frac{\widehat{q_*}}{\widehat{p_*}} \le 1 + \frac{\alpha \widehat{q_*}}{nq}.
\end{equation}
Since \eqref{eq:SP-expo-1} is obvious, let us check the inequality \eqref{eq:SP-expo-2}. If $q_*\,(=p_*)=1$, we have
\begin{equation}
\frac{\widehat{q_*}}{\widehat{p_*}}
\le 1 + \frac{\alpha}{nq} \le 1 + \frac{\alpha \widehat{q_*}}{nq}.
\end{equation}

Next let us consider the case $(q_* \ge)\,p_* > 1$.
Recall the third condition of \eqref{assmpt:exponents}:
\begin{equation}
\frac{q}{p} < 1 + \frac{\alpha}{n}.
\end{equation}
Then observe
$\widehat{p_*}=p_* = \frac{np}{n+p},\, \widehat{q_*}=q_* = \frac{nq}{n+q}$, and  
\begin{align}
\frac{\widehat{q_*}}{\widehat{p_*}} \le 1 + \frac{\alpha \widehat{q_*}}{nq}
&\iff \frac{q_*}{p_*} \le 1 + \frac{q_{*}\alpha}{nq} \\
&\iff \frac{1}{p_*} \le \frac{1}{q_*} + \frac{\alpha}{nq} \\
&\iff \frac{n+p}{np} \le  \frac{n+q}{nq} + \frac{\alpha}{nq} \\
&\iff \frac{1}{p} \le \frac{1}{q} + \frac{\alpha}{nq}
\iff \frac{q}{p} \le 1 + \frac{\alpha}{n}.
\end{align}
 
Next, we consider the case $q_* > 1$ and $p_*=1$. Then one can see that
$\widehat{p_*} > p_* = 1,\, \widehat{q_*}=q_* = \frac{nq}{n+q}$ and 
\begin{align}
\frac{\widehat{q_*}}{\widehat{p_*}} \le 1 + \frac{\alpha \widehat{q_*}}{nq}
&\impliedby q_* \le 1 + \frac{\alpha q_*}{nq} \\
&\iff 1 \le \frac{1}{q_*} + \frac{\alpha}{nq} \\
&\iff 1 \le \frac{n+q}{nq} + \frac{\alpha}{nq} \\
&\iff q \le \frac{n+q}{n} + \frac{\alpha}{n} \\
&\iff \frac{(n-1)q}{n} \le 1 + \frac{\alpha}{n}
\impliedby \frac{q}{p} \le 1 + \frac{\alpha}{n}.
\end{align}
The last implication follows from the assumption $p \le \frac{n}{n-1} \iff p_*=1$. The proof of \eqref{eq:SP-expo-2} is completed.
 
By \eqref{eq:SP-expo-1}, we may assume
\begin{equation}\label{eq:SP-delta_0}
\widehat{p_*} \le p \delta_0
\quad
\text{and}
\quad
\widehat{q_*} \le q \delta_0.
\end{equation}
From \eqref{eq:SP-delta_0} and the fact $u$ is a very weak solution with $\delta\,(> \delta_0)$, it follows
\begin{equation}
\int_{B_{2R}} \left( |D^m u|^{\widehat{p_*}} + a^{\frac{\widehat{q_*}}{q}} |D^m u|^{\widehat{q_*}} \right)< \infty.
\end{equation}
By this inequality, \eqref{eq:SP-expo-1} and $q \le (q_*)^* \le (\widehat{q_*})^*$, we can apply Corollary \ref{cor:sp} with $(p,q,a,\alpha,r)=(\widehat{p_*},\widehat{q_*},a^{\frac{\widehat{q_*}}{q}},\frac{\alpha \widehat{q_*}}{q},q)$ and obtain
\begin{align}\label{eq:SP-J_2's}
J_2
&= \fint_{B_{2R}} [a^{\frac{\widehat{q_*}}{q}}]^{\frac{q}{\widehat{q_*}}}\left|\frac{D^{m-1} u -D^{m-1} P}{R}\right|^q  \\
&\le c\left(\fint_{B_{2R}} a^{\frac{\widehat{q_*}}{q}}|D^m u|^{\widehat{q_*}} \right)^{\frac{q}{\widehat{q_*}}} 
+ cR^{\alpha}\left(\fint_{B_{2R}} |D^m u|^{\widehat{p_*}} \right)^{\frac{q}{\widehat{p_*}}} \\
&\eqqcolon c(J_{21} + J_{22}),
\end{align}
where $c=c(n,m,p,q,\alpha,[a]_{\alpha})$.
Here we used $[a^{\frac{\widehat{q_*}}{q}}]_{\frac{\alpha \widehat{q_*}}{q}} \le [a]_{\alpha}$.
 
We first estimate $J_{21}$. 
Set $\widehat{\delta} = \max \left\{ \frac{\widehat{p_*}}{p}, \frac{\widehat{q_*}}{q} \right\} \in (0,1)$.
Then \eqref{def:H_ell} and Jensen's inequality imply
\begin{align}\label{eq:SP-J_21}
J_{21}
&= \left(\fint_{B_{2R}} \left( a|D^m u|^q \right)^{\frac{\widehat{q_*}}{q}} \right)^{\frac{q}{\widehat{q_*}} } \\
&\le \left(\fint_{B_{2R}} \left( a|D^m u|^q \right)^{\widehat{\delta}} \right)^{\frac{1}{\widehat{\delta}}} 
\le \left(\fint_{B_{2R}} H_m(x,D^m u)^{\widehat{\delta}} \right)^{\frac{1}{\widehat{\delta}}}.
\end{align}
 
On the other hand, for $J_{22}$, we have 
\begin{align}\label{eq:SP-J_22}
J_{22}
&= R^{\alpha}\left(\fint_{B_{2R}} |D^m u|^{\widehat{p_*}} \right)^{\frac{q-p}{\widehat{p_*}}}\left(\fint_{B_{2R}} |D^m u|^{\widehat{p_*}} \right)^{\frac{p}{\widehat{p_*}}}.
\end{align}
For the second factor of \eqref{eq:SP-J_22}, similarly to \eqref{eq:SP-J_21}, we get
\begin{equation}\label{eq:SP-J_22-0}
\left(\fint_{B_{2R}} |D^m u|^{\widehat{p_*}} \right)^{\frac{p}{\widehat{p_*}}}
= \left(\fint_{B_{2R}} (|D^m u|^p)^{\frac{\widehat{p_*}}{p}} \right)^{\frac{p}{\widehat{p_*}}}
\le \left( \fint_{B_{2R}} H_m(x,D^m u)^{\widehat{\delta}} \right)^{\frac{1}{\widehat{\delta}}}.
\end{equation}
Moreover, for the first factor of \eqref{eq:SP-J_22}, using $\widehat{p_*} \le p \delta_0$ (see \eqref{eq:SP-delta_0}), Jensen's inequality and \eqref{eq:R_0}, it follows
\begin{align}\label{eq:SP-J_22-1}
R^{\alpha}\left(\fint_{B_{2R}} |D^m u|^{\widehat{p_*}} \right)^{\frac{q-p}{\widehat{p_*}}}
&\le R^{\alpha}\left(\fint_{B_{2R}} |D^m u|^{p\delta_0} \right)^{\frac{q-p}{p\delta_0}} \\
&\le cR^{\alpha-n\frac{q-p}{p\delta_0}}\|D^m u\|_{L^{p\delta_0}(\Omega_0)}^{q-p} \\
&= c\left( R^{\frac{\alpha}{q}-n\left(\frac{1}{p\delta_0}-\frac{1}{q\delta_0}\right)}\|D^m u\|_{L^{p\delta_0}(\Omega_0)}^{1-\frac{p}{q}} \right)^q
\le c(n,p,q).
\end{align}
Therefore, \eqref{eq:SP-J_22}, \eqref{eq:SP-J_22-0} and \eqref{eq:SP-J_22-1} yield
\begin{equation}\label{eq:SP-J_22-2}
J_{22}
\le c(n,p,q)\left(\fint_{B_{2R}} H_m(x,D^m u)^{\widehat{\delta}} \right)^{\frac{1}{\widehat{\delta}}}.
\end{equation}

From \eqref{eq:SP-J_2's}, \eqref{eq:SP-J_21}, \eqref{eq:SP-J_22-2} and Jensen's inequality, we derive
\begin{equation}\label{eq:SP-J_2}
J_2
\le c(n,m,p,q,\alpha,[a]_{\alpha})\left(\fint_{B_{2R}} H_m(x,D^m u)^{\widehat{\delta}} \right)^{\frac{1}{\widehat{\delta}}}.
\end{equation}
The estimate of $J_2$ is completed.

For $J_1$, from $p_* \le \widehat{p_*}$ (see \eqref{eq:SP-expo-1}), the classical Sobolev--Poincar\'e inequality  and \eqref{eq:SP-J_22-0}, it follows that
\begin{equation}\label{eq:SP-J_1}
J_1
\le c\left(\fint_{B_{2R}} |D^m u|^{\widehat{p_*}} \right)^{\frac{p}{\widehat{p_*}}} 
\le \left(\fint_{B_{2R}} H_m(x,D^m u)^{\widehat{\delta}} \right)^{\frac{1}{\widehat{\delta}}}.
\end{equation}
 
Finally, for $J_3$, applying the classical Sobolev--Poincar\'e inequality and \eqref{eq:SP-J_22-2}, we obtain
\begin{align}\label{eq:SP-J_3}
J_3 
&\le cR^{\alpha}\left(\fint_{B_{2R}} |D^m u|^{\widehat{p_*}} \right)^{\frac{q}{\widehat{p_*}}} \\
&= c(n,m,p,q)J_{22}
\le c(n,m,p,q)\left(\fint_{B_{2R}} H_m(x,D^m u)^{\widehat{\delta}} \right)^{\frac{1}{\widehat{\delta}}}.
\end{align}
Hence, from \eqref{eq:SP-J's}, \eqref{eq:SP-J_1}, \eqref{eq:SP-J_2} and \eqref{eq:SP-J_3}, we conclude
\begin{equation}
\sum_{\ell=0}^{m-1}\fint_{B_{2R}} H_m \left(x,\frac{D^\ell u - D^\ell P}{R^{m-\ell}} \right)
\le c\left(\fint_{B_{2R}} H_m(x,D^m u)^{\widehat{\delta}} \right)^{\frac{1}{\widehat{\delta}}}.
\end{equation}
The proof is completed.
\end{proof}

\begin{rem}
In Lemma \ref{lem:SP}, we used Corollary \ref{cor:sp}. 
Alternatively, we could employ \cite[Corollary 7.2.6]{HH}, which is derived within the framework of Sobolev--Musielak spaces.
\end{rem}

Combining the last two lemmas, we immediately obtain the following reverse H\"older inequality.
\begin{lem}\label{lem:Reverse-Holder}
There exists $c=c(n,m,p,q,\alpha,[a]_{\alpha},\nu,\delta_0)>0$ such that if $\delta \in [\delta_1,1)$, then
\begin{align}\label{eq:RH-c}
    &\fint_{B_{R}} H_m (x, D^{m} u)^\delta \\
    &\leq c\left( \fint_{B_{2R}}   H_m\left(x, D^{m} u \right)^{\widehat{\delta}} \right)^{\frac{\delta}{\widehat{\delta}}}
    + c\fint_{B_{2R}} F^\delta + \frac{1}{2}\fint_{B_{2R}} H_m (x, D^{m} u)^\delta.
\end{align}
Here $\delta_1$ is the constant in Lemma \ref{lem:Caccioppoli} and $\widehat{\delta}$ is the constant in Lemma \ref{lem:SP}.
\end{lem}

\begin{proof}
By Lemma \ref{lem:Caccioppoli}, we have
\begin{align}\label{eq:rh-2-1}
    \fint_{B_{R}} H_m (x, D^{m} u)^\delta 
    &\leq \frac{1}{2}\fint_{B_{2R}} H_m (x, D^{m} u)^\delta \\
    &+ c \sum_{\ell = 0}^{m - 1}\fint_{B_{2R}}  H_m\left(x, \frac{D^{\ell} u - D^{\ell}P}{R^{m-\ell}}\right)^{\delta}
    + c\fint_{B_{2R}} F^\delta,
\end{align}
where $c=c(n,m,p,q,\alpha,[a]_{\alpha},\nu,\delta_0)$.

It remains to estimate the second term in the right-hand side. 
Using Jensen's inequality and Lemma \ref{lem:SP}, we obtain
\begin{align}\label{eq:rh-2-2}
&\sum_{\ell=0}^{m-1}\fint_{B_{2R}}  H_m\left(x, \frac{D^{\ell} u - D^{\ell}P}{R^{m-\ell}}\right)^{\delta} \\
&\le \sum_{\ell=0}^{m-1}\left(\fint_{B_{2R}}  H_m\left(x, \frac{D^{\ell} u - D^{\ell}P}{R^{m-\ell}}\right) \right)^{\delta} \\
&\le c(n,m,p,q,\alpha,[a]_{\alpha})\left(\fint_{B_{2R}}  H_m\left(x, D^m u \right)^{\widehat{\delta}} \right)^{\frac{\delta}{\widehat{\delta}}}.
\end{align}

Combining \eqref{eq:rh-2-1} and \eqref{eq:rh-2-2}, we arrive at the conclusion.
\end{proof}

\subsection{Conclusion via Gehring's lemma}
\begin{proof}[Proof of Theorem \ref{thm:main}]
We apply Lemma \ref{lem:Gehring} to the reverse H\"older inequality \eqref{eq:RH-c}. Let $\delta \in [\delta_1,1)$, $f_1 = H_m(x,D^m u)^{\delta}$ and $f_2 = F^\delta$. By \eqref{eq:F}, note that $f_1 \in L^{1}_{loc}(\Omega_0)$ and $f_2 \in L^{1 + \varepsilon_0}_{loc}(\Omega_0)$, where $\varepsilon_0 = \frac{1}{\delta_0}-1$. By Lemma \ref{lem:Reverse-Holder}, there exists $c=c(n,m,p,q,\alpha,[a]_{\alpha},\nu,\delta_0) > 0$ such that
\begin{equation}
     \fint_{B_{R}} f_1
     \leq c \left( \fint_{B_{2R}} f_1^\frac{\widehat{\delta}}{\delta}\right)^\frac{\delta}{\widehat{\delta}}
    + c\fint_{B_{2R}} f_2 + \frac{1}{2}\fint_{B_{2R}} f_1
\end{equation}
for any $B_{2R} \subset \subset \Omega_0$ with $R \le R_0$. 
Therefore, Lemma \ref{lem:Gehring} implies that there exists $c^*=c^*(n,m,p,q,\alpha,[a]_{\alpha},\nu,\delta_0)$ such that  $f_1$ belongs to $L^{1 + \varepsilon}_\text{loc}(\Omega_0)$ for any $0 < \varepsilon \leq \min \left\{ \frac{1-\frac{\widehat{\delta}}{\delta}}{c^*}, \varepsilon_0 \right\}$.
We have $\delta \geq \delta_1 \geq \frac{1+\delta_0}{2}$ and we may assume $(1 >)\,\delta_0 \ge \widehat{\delta}$. Then we notice that
\begin{equation}
\frac{1-\frac{\widehat{\delta}}{\delta}}{c^*} \geq \frac{1-\frac{2\widehat{\delta}}{1+\delta_0}}{c^*} > 0.
\end{equation}
If we choose $\delta_2 \in [\delta_1,1)$ such that 
$$
\min \left\{\frac{1-\frac{2\widehat{\delta}}{1+\delta_0}}{c^*}, \varepsilon_0 \right\} \ge \frac{1}{\delta \delta_2}-1
\quad
\text{for any }
\delta \in [\delta_2,1),
$$
then $f_1 \in L^{\frac{1}{\delta \delta_2}}_\text{loc}(\Omega_0)$, that is, $H_m(x,D^m u) \in  L^{\frac{1}{\delta_2}}_\text{loc}(\Omega_0)$.
The proof of Theorem \ref{thm:main} is completed.
\end{proof}

\begin{proof}[Proof of Corollary \ref{cor:main}]
Under the assumption of Corollary \ref{cor:main}, we can conclude Lemma \ref{lem:Reverse-Holder}. However, we notice that $F \in L^1_{loc}(\Omega_0)$ since $f_p + a(x)f_q \in L^1_{loc}(\Omega)$ in this setting. Let $\delta \in [\delta_1,1)$, $f_1 = H_m(x,D^m u)^{\delta}$ and $f_2 = F^\delta$. Note that $f_1 \in L^{1}_{loc}(\Omega_0)$ and $f_2 \in L^{1 + \widehat{\varepsilon_0}}_{loc}(\Omega_0)$, where $\widehat{\varepsilon_0} = \frac{1}{\delta}-1$. By $\delta > \delta_0$, we have $\widehat{\varepsilon_0} < \frac{1}{\delta_0}-1$. Together with this, Lemma \ref{lem:Gehring} implies there exists $\widehat{c^*}=\widehat{c^*}(n,m,p,q,\alpha,[a]_{\alpha},\nu,\delta_0)$ such that
$f_1 \in L^{1 + \varepsilon}_\text{loc}(\Omega_0)$ for any $0 < \varepsilon \leq \min \left\{ \frac{1-\frac{\widehat{\delta}}{\delta}}{\widehat{c^*}}, \widehat{\varepsilon_0} \right\}$.
Note that $\delta \geq \delta_1 \geq \frac{1+\delta_0}{2}$ and we may assume $(1 >)\,\delta_0 \ge \widehat{\delta}$. Then one can see that
\begin{equation}
\frac{1-\frac{\widehat{\delta}}{\delta}}{\widehat{c^*}} \geq \frac{1-\frac{2\widehat{\delta}}{1+\delta_0}}{\widehat{c^*}} > 0.
\end{equation}
If we choose $\delta_3 \in [\delta_1,1)$ such that 
$$
\frac{1-\frac{2\widehat{\delta}}{1+\delta_0}}{\widehat{c^*}} \ge \widehat{\varepsilon_0} = \frac{1}{\delta}-1
\quad
\text{for any }
\delta \in [\delta_3,1),
$$
then we conclude $f_1 \in L^{1+\widehat{\varepsilon_0}}_\text{loc}(\Omega_0)$ and $H_m(x,D^m u) \in  L^{1}_\text{loc}(\Omega_0)$. The proof of Corollary \ref{cor:main} is completed.
\end{proof}

\section*{Acknowledgements}
The author is grateful to Professor Kentaro Fujie for his helpful comments and encouragement throughout this work.

\section*{Declarations}
The author has no relevant financial or non-financial interests to disclose.

\bibliographystyle{abbrv}  
\bibliography{main}      

\appendix
\section{Gehring's lemma}
As mentioned in Subsection \ref{subsec:Ge}, we provide an elementary proof of Lemma \ref{lem:Gehring} in this appendix.

Let us begin with the following classical iteration lemma:
\begin{lem} \label{lem:iteration}
Let $0 < R_0 < R_1 < \infty$  
and $h:[R_0,R_1] \rightarrow \mathbb{R}$ be a nonnegative and bounded function.
Moreover, let $\tau \in (0,1)$ and $C_1,\,C_2,\,\gamma \ge 0$ be fixed constants and 
assume that 
\begin{equation}
    h(s) \leq \tau h(t) + C_1 + \frac{C_2}{(t - s)^\gamma}
\end{equation}
holds for all $R_0 \leq s < t \leq R_1$.
Then there exists a positive constant $c = c(\tau, \gamma)$ such that
\begin{equation}
    h(R_0) \leq \frac{C_1}{1-\tau} + \frac{cC_2}{(R_1 - R_0)^\gamma}.
\end{equation}
Moreover, we can take $c(\tau,\gamma) = \sum_{i=0}^{\infty} \tau^{i} \left\{ (i+1)(i+2) \right \}^{\gamma}$.
\end{lem}
\begin{proof}
Although this lemma is well known (see \cite[Lemma 6.1]{Gi}, for example), we provide a proof here to make the constants appearing in the conclusion more explicit.
For each $k \in \N \cup \{0\}$, we set
$$
r_k \coloneqq R_0 + \left(1-\frac{1}{k+1}\right)(R_1-R_0).
$$ 
Then we have 
\begin{equation}\label{eq:iteration-3}
  h(R_0) \le \tau^kh(r_k) + C_1\sum_{j=0}^{k-1}\tau^j + C_2\sum_{j=0}^{k-1}\frac{\tau^j}{(r_{j+1}-r_j)^{\gamma}}
\end{equation}
for each $k \in \N$ by a simple iteration. Since 
\begin{align}
r_{j+1} - r_j 
&= R_0 + \left(1-\frac{1}{j+2}\right)(R_1-R_0) \\
&- \left\{ R_0 + \left(1-\frac{1}{j+1}\right)(R_1-R_0) \right\} \\
&= \left(\frac{1}{j+1} - \frac{1}{j+2}\right)(R_1 - R_0) = \frac{R_1 - R_0}{(j+1)(j+2)}
\quad
\text{for each }
j \in \N \cup \{0\},
\end{align}
letting $k \rightarrow \infty$ in \eqref{eq:iteration-3}, we obtain the desired inequality.
\end{proof}

Using Vitali's covering lemma, the layer-cake representation formula and Lemma \ref{lem:iteration},
we obtain the following
\begin{lem}\label{lem:m_Gehring}
Let $\Omega \subset \mathbb{R}^n$ be an open set, and let
$f \in L^1_{loc}(\Omega)$ and $g \in L^{1 + \varepsilon_0}_{loc}(\Omega)$ for some $\varepsilon_0 > 0$
, where $f$ and $g$ are nonnegative.
In addition, assume there exist $\kappa \in (0,1)$, $A > 0$ and $R_0 > 0$ such that
\begin{equation}\label{eq:Gehring_asuumpt}
\fint_{B_{R}} f
\le A\left( \fint_{B_{2R}} f^\kappa  \right)^{\frac{1}{\kappa}} + \fint_{B_{2R}} g,
\end{equation}
for any ball $B_{2R} \subset \subset \Omega$ such that $R \le R_0$ and
\begin{equation}
\fint_{B_{2R}} f \leq \fint_{B_{R}} f.
\end{equation}
Then there exists $c^* = c^*(n,A,\varepsilon_0)>0$ 
such that 
if $0 < \varepsilon \leq \min \left\{ \frac{1-\kappa}{c^*}, \varepsilon_0 \right\}$, 
$f \in L^{1 + \varepsilon}_{loc}(\Omega)$. 
Moreover, there exists $c = c(n,A,\varepsilon_0)$ such that for any $0 < \varepsilon \leq \min \left\{ \frac{1-\kappa}{c^*}, \varepsilon_0 \right\}$, it holds
\begin{equation}
    \left( \fint_{B_{R}} f^{1 + \varepsilon} \right)^\frac{1}{1 + \varepsilon} 
    \leq c\fint_{B_{2R}} f 
    +    c\left( \fint_{B_{2R}} g^{1 + \varepsilon} \right)^\frac{1}{1 + \varepsilon} 
\end{equation}
whenever $B_{2R}\subset \subset \Omega$ and $R \le R_0$. 
Furthermore, the constant $c^*$ can be chosen to be monotone increasing with respect to $\varepsilon_0$.
\end{lem}

\begin{proof}[Proof of \rm{Lemma \ref{lem:m_Gehring}}]
We first note that Lemma\,\ref{lem:Gehring} immediately follows from Lemma\,\ref{lem:m_Gehring}.
The proof of Lemma\,\ref{lem:m_Gehring} is based on \cite[Section 4]{BBK}. 
Although the argument therein is carried out under a restricted setting, we note that the proof is general enough to establish Gehring's lemma.

Fix any $B_{2R} \subset \subset \Omega$ such that $R \le R_0$. Take $r_1, r_2 > 0$ 
such that $R \leq r_1 < r_2 \leq 2R$. 
In addition, $B_{r_1}$ and $B_{r_2}$ have the same center as $B_{2R}$.
Observe that
\begin{equation}\label{eq:rho}
\frac{r_2-r_1}{10} \le R \le R_0.
\end{equation}
For any $x \in B_{r_1}$ and any $\rho \in (0, r_2 - r_1)$, 
we define  
\begin{equation}
\Psi(x,\rho) \coloneqq \fint_{B(x,\rho)} f.
\end{equation}
Then, for any $x \in B_{r_1}$ 
and any $\rho \in [\frac{r_2-r_1}{10},r_2 - r_1)$, we have
\begin{equation}\label{eq:Gehring_lambda0}
\Psi(x,\rho) 
= \frac{1}{|B(x,\rho)|}\int_{B(x,\rho)} f \\
\leq \frac{10^n}{\omega_n (r_2 - r_1)^n}\int_{B_{2R}} f.
\end{equation}
Let us define
\begin{equation}\label{eq:Gehring_lambda0_2}
\lambda_0 \coloneqq \frac{10^n}{\omega_n (r_2 - r_1)^n}\int_{B_{2R}} f
\end{equation}
and fix any $\lambda > 0$ such that
\begin{equation}\label{eq:Gehring_lambda}
\lambda > \lambda_0.
\end{equation}
Here we set for $r \in \{r_1, r_2\}$ and $\mu > 0$,
\begin{equation}
E(r,\mu) \coloneq \{ x \in B_r : f(x) > \mu\}.
\end{equation}
Then we can take a null set $N$ such that
for each $x \in E(r_1,\lambda) \setminus N$ there exists $\rho_x \in \left( 0,\frac{r_2-r_1}{10} \right)$ satisfying
\begin{equation}\label{eq:Gehring_exit}
    \Psi(x,\rho_x) = \lambda
\end{equation}
and
\begin{equation}\label{eq:Gehring_exit_2}
    \Psi(x,\rho) 
    \leq \lambda
    \quad
    \text{for any }
    \rho \in [\rho_x, r_2 - r_1).
\end{equation}

Indeed, Lebesgue's differentiation theorem and the definition of $E(r_1, \lambda)$ yield that
\begin{equation}\label{eq:Gehring_1}
\lim_{\rho \rightarrow 0} \Psi(x,\rho) = f(x) > \lambda
\quad
\text{for a.e.\,}
x \in E(r_1,\lambda).
\end{equation}
Here, we fix any $x \in E(r_1,\lambda)$ such that the last inequality holds.
From \eqref{eq:Gehring_lambda0}, \eqref{eq:Gehring_lambda0_2} and \eqref{eq:Gehring_lambda}, it follows that
\begin{equation}\label{eq:Gehring_2}
\Psi(x,\rho) \le \lambda_0 < \lambda
\quad
\text{for any }
\rho \in \left[\frac{r_2-r_1}{10},r_2 - r_1 \right).
\end{equation}
Since the function $\rho \mapsto \Psi(x,\rho)$ is continuous in $(0,r_2 - r_1)$,
\begin{equation}
\left\{\rho \in \left( 0,\frac{r_2-r_1}{10} \right) : \Psi(x,\rho) = \lambda \right\}
\end{equation}
is not an empty set by \eqref{eq:Gehring_1}, \eqref{eq:Gehring_2} and the intermediate value theorem.
Let us define
\begin{equation}
\rho_x \coloneqq \sup \left\{\rho \in \left( 0,\frac{r_2-r_1}{10} \right) : \Psi(x,\rho) = \lambda \right\}
\end{equation}
and this $\rho_x$ satisfies the desired properties.

Since the family $\{B(x,2\rho_x)\}_{x \in E(r_1, \lambda) \setminus N}$ 
covers $E(r_1, \lambda) \setminus N$, by Vitali's covering lemma, there exists a countable subfamily
$\{2B_i\} = \{B(x_i,2\rho_{x_i})\}_{i \in \mathbb{N}}$ satisfying
\begin{equation}\label{eq:Gehring_covering}
 \{2B_i\}\text{ is disjoint}
\quad
\text{ and }
\quad
    E(r_1, \lambda) \setminus N \subset \bigcup_{i \in \mathbb{N}} 10B_i.
\end{equation}


Fix any $i \in N$. By \eqref{eq:Gehring_exit} and \eqref{eq:Gehring_exit_2}, we have
\begin{equation}\label{eq:Gehring_exit-i}
    \Psi(x_i,\rho_{x_i}) = \fint_{B_i}f = \lambda
\end{equation}
and
\begin{equation}\label{eq:Gehring_exit_2-i}
    \Psi(x_i,\rho) 
    = \fint_{B(x_i,\rho)}f
    \leq \lambda
    \quad
    \text{for any }
    \rho \in [\rho_{x_i}, r_2 - r_1).
\end{equation}
In particular, it follows that
\begin{equation}
  \fint_{2B_i} f \le \lambda = \fint_{B_i} f.
\end{equation}
Therefore, we apply \eqref{eq:Gehring_asuumpt} to obtain
\begin{equation}\label{eq:Gehring_{f,g}}
    \fint_{B_i} f 
    \leq A \left( \fint_{2B_i} f^{\kappa} \right)^\frac{1}{\kappa} 
    + \fint_{2B_i} g.
\end{equation}
Let us define
\begin{equation}\label{eq:Gehring-d}
d = \frac{1+\kappa}{2}.
\end{equation}
Since $\kappa \le d \le 1$, Jensen's inequality yields that
\begin{align}\label{eq:Gehring_f}
\left( \fint_{2B_i} f^{\kappa} \right)^\frac{1}{\kappa}
&\leq \left( \fint_{2B_i} f^{d} \right)^\frac{1}{d} \\
&= \left( \fint_{2B_i} f^d \right)^{\frac{1}{d} - 1} \cdot \left( \fint_{2B_i} f^d \right) \\
&\leq \left( \fint_{2B_i} f \right)^{1 - d} \cdot \left( \fint_{2B_i} f^d \right) \\
&= \Psi(x_i,2\rho_{x_i})^{1 - d} \cdot \left( \fint_{2B_i} f^d \right)
\leq \lambda^{1 - d}\fint_{2B_i} f^d.
\end{align}
In the last inequality, we used \eqref{eq:Gehring_exit_2-i}.
In addition, we have
\begin{align}\label{eq:Gehring_f_2}
    \fint_{2B_i} f^d
    &= \frac{1}{|2B_i|}\int_{2B_i\cap E(r_2, \theta \lambda)^c} f^d 
    + \frac{1}{|2B_i|}\int_{2B_i \cap E(r_2, \theta \lambda)} f^d \\
    &\leq \frac{|2B_i\cap E(r_2, \theta \lambda)^c|}{|2B_i|}(\theta \lambda)^d 
    + \frac{1}{|2B_i|}\int_{2B_i \cap E(r_2, \theta \lambda)} f^d \\
    &\leq \theta^{d} \lambda^d + \frac{1}{|2B_i|}\int_{2B_i \cap E(r_2, \theta \lambda)} f^d
\end{align}
for any $\theta > 0$.
Therefore, combining \eqref{eq:Gehring_f} and \eqref{eq:Gehring_f_2}, we derive
\begin{equation}\label{eq:Gehring_f_3}
\left( \fint_{2B_i} f^{\kappa} \right)^\frac{1}{\kappa}
\leq \lambda^{1 - d}\fint_{2B_i} f^d \leq \theta^d \lambda 
+ \frac{\lambda^{1 -d}}{|2B_i|}\int_{2B_i \cap E(r_2, \theta \lambda)} f^d.
\end{equation}

Next, we estimate the second integral in \eqref{eq:Gehring_{f,g}}. 
We set
\begin{equation}
\mathcal{E}(r,\mu) \coloneqq \{ x \in B_r : g(x) > \mu\}
\end{equation}
for $r \in \{r_1, r_2\}$ and $\mu > 0$.
A calculation similar to that in \eqref{eq:Gehring_f_2} yields that
\begin{align}\label{eq:Gehring_g}
\fint_{2B_i} g 
&= \frac{1}{|2B_i|}\int_{2B_i\cap \mathcal{E} \left(r_2, \frac{\lambda}{4} \right)^c} g + \frac{1}{|2B_i|}\int_{2B_i \cap \mathcal{E} \left(r_2, \frac{\lambda}{4} \right)} g \\
&\leq \frac{|2B_i\cap \mathcal{E} \left(r_2, \frac{\lambda}{4} \right)^c|}{|2B_i|}\frac{\lambda}{4} + \frac{1}{|2B_i|}\int_{2B_i \cap \mathcal{E} \left(r_2, \frac{\lambda}{4} \right)} g \\
&\leq \frac{\lambda}{4} + \frac{1}{|2B_i|}\int_{2B_i \cap \mathcal{E} \left(r_2, \frac{\lambda}{4} \right)} g.
\end{align}

Combining \eqref{eq:Gehring_exit-i}, \eqref{eq:Gehring_{f,g}}, \eqref{eq:Gehring_f_3} and \eqref{eq:Gehring_g}, we obtain
\begin{align}\label{eq:Ge-lambda}
\lambda 
&= \Psi(x_i,\rho_{x_i}) \\
&= \fint_{B_i} f \\
&\leq  A\left(\fint_{2B_i} f^{\kappa} \right)^\frac{1}{\kappa} + \fint_{2B_i} g \\
&\leq \left( A\theta^{d} + \frac{1}{4} \right)\lambda + \frac{A\lambda^{1 -d}}{|2B_i|}\int_{2B_i \cap E(r_2, \theta \lambda)} f^d 
+ \frac{1}{|2B_i|}\int_{2B_i \cap \mathcal{E} \left(r_2, \frac{\lambda}{4} \right)} g
\end{align}
for any $\theta > 0$.
Let us define 
\begin{equation}\label{eq:Gehring-theta}
\theta \coloneqq \left( \frac{1}{4A+1} \right)^{\frac{1}{d}}.
\end{equation}
Since $A\theta^{d} + \frac{1}{4} \le \frac{1}{2}$, we have
\begin{equation}\label{eq:Ge-lambda-2}
\lambda \le \frac{2A\lambda^{1 -d}}{|2B_i|}\int_{2B_i \cap E(r_2, \theta \lambda)} f^d 
+ \frac{2}{|2B_i|}\int_{2B_i \cap \mathcal{E} \left(r_2, \frac{\lambda}{4} \right)} g.
\end{equation}
Noting that $\fint_{10B_i} f = \Psi(x_i,10\rho_{x_i}) \le \lambda$ by \eqref{eq:Gehring_exit_2-i}, \eqref{eq:Ge-lambda-2} implies
\begin{equation}\label{eq:Gehring_sum}
\int_{10B_i} f \le 2 \cdot 5^n A\lambda^{1 -d}\int_{2B_i \cap E(r_2, \theta \lambda)} f^d 
+ 2 \cdot 5^n \int_{2B_i \cap \mathcal{E} \left(r_2, \frac{\lambda}{4} \right)} g.
\end{equation}

Here we set $c_1 \coloneqq 2 \cdot 5^nA$ and $c_2 \coloneqq 2 \cdot 5^n$.
Summing \eqref{eq:Gehring_sum} over $i \in \mathbb{N}$ and using \eqref{eq:Gehring_covering}, we obtain 
\begin{align}\label{eq:Gehring_sum_2}
    \int_{E(r_1, \lambda)} f
    &= \int_{E(r_1, \lambda) \setminus N} f \\
    &\leq \sum_{i \in \mathbb{N}} \int_{10B_i} f \\
    &\leq \sum_{i \in \mathbb{N}} \left( c_1\lambda^{1 -d}\int_{2B_i \cap E(r_2, \theta \lambda)} f^d 
    + c_2\int_{2B_i \cap \mathcal{E} \left(r_2, \frac{\lambda}{4} \right)} g  \right) \\
    &\leq   c_1\lambda^{1 -d}\int_{E(r_2, \theta \lambda)} f^d 
    + c_2\int_{\mathcal{E} \left(r_2, \frac{\lambda}{4} \right)} g.
\end{align} 

Fix any $t > \lambda_0$ and $0 < \varepsilon \le \varepsilon_0$. 
Multiplying both sides of \eqref{eq:Gehring_sum_2} by $\varepsilon \lambda^{\varepsilon-1}$ and integrating from $\lambda_0$ to $t$, we have
\begin{align}\label{eq:Ge-J's}
J_0 
&\coloneqq \varepsilon \int_{\lambda_0}^{t} \lambda^{\varepsilon - 1} \left( \int_{E(r_1,\lambda)} f(x)\ dx \right)\ d\lambda  \\
&\leq c_1\varepsilon \int_{\lambda_0}^{t} \lambda^{-d+\varepsilon} \left( \int_{E(r_2,\theta\lambda)} f(x)^d\ dx \right)\ d\lambda \\
&+ c_2\varepsilon \int_{\lambda_0}^{t} \lambda^{\varepsilon-1} \left( \int_{\mathcal{E} \left(r_2, \frac{\lambda}{4} \right)} g(x)\ dx \right)\ d\lambda 
\eqqcolon c_1 J_1 + c_2 J_2.
\end{align}

Firstly, we estimate $J_0$. We divide it into $J_{01}$ and $J_{02}$:
\begin{align}\label{eq:Ge-J_0's}
J_0 
&= \varepsilon \int_{0}^{t} \lambda^{\varepsilon - 1} \left( \int_{E(r_1,\lambda)} f(x)\ dx \right)\ d\lambda \\
    &- \varepsilon\int_{0}^{\lambda_0} \lambda^{\varepsilon - 1} \left( \int_{E(r_1,\lambda)} f(x)\ dx \right)\ d\lambda 
\eqqcolon J_{01} - J_{02}.
\end{align}

For $J_{01}$, Fubini's theorem and \eqref{eq:lcr-1} imply
\begin{align}
J_{01} 
&= \int_{B_{r_1}} f(x) \left(\varepsilon \int_{0}^{t} \lambda^{\varepsilon - 1}\chi_{ \left\{ B_{r_1}:f > \lambda \right\} }(x) \ d\lambda  \right)\ dx \\
&= \int_{B_{r_1}} f(x) \left(\varepsilon \int_{0}^{\infty} \lambda^{\varepsilon - 1}\chi_{ \left\{ B_{r_1}:\min\{ f, t\}> \lambda \right\}}(x) \ d\lambda  \right)\ dx \\
&=  \int_{B_{r_1}} f(x) \min \left\{f(x), t \right\}^\varepsilon dx.
\end{align}
Moreover, since $\theta \le 1$, we have
\begin{equation}\label{eq:Ge-J_01}
J_{01} 
= \int_{B_{r_1}} f \min \left\{f, t \right\}^\varepsilon 
\ge \int_{B_{r_1}} f \min\left\{f, \theta t \right\}^\varepsilon 
= h(r_1),
\end{equation}
where
\begin{equation}\label{eq:Ge-h}
h(r) \coloneqq \int_{B_{r}} f \min\left\{f, \theta t \right\}^\varepsilon
\quad
\text{for }
r \in [R,2R].
\end{equation}

On the other hand, for $J_{02}$, recall \eqref{eq:Gehring_lambda0} and we obtain
\begin{align}\label{eq:Ge-J_02}
J_{02} 
&\leq \varepsilon\int_{0}^{\lambda_0} \lambda^{\varepsilon - 1} \left( \int_{B_{2R}} f(x)\ dx \right)\ d\lambda \\
&= \lambda_0^\varepsilon \int_{B_{2R}} f(x)\ dx \\
&= \left( \frac{10^n}{\omega_n (r_2 - r_1)^n}\int_{B_{2R}} f \right)^{\varepsilon}\int_{B_{2R}} f(x)\ dx \\
&= \frac{10^{n\varepsilon}}{\omega_n^{\varepsilon} (r_2 - r_1)^{n\varepsilon}}\left( \int_{B_{2R}} f(x)\ dx \right)^{1+\varepsilon}
= \frac{M}{(r_2 - r_1)^{n\varepsilon}},
\end{align}
where
\begin{equation}
  M \coloneqq \frac{10^{n\varepsilon}}{\omega_n^{\varepsilon}}\left( \int_{B_{2R}} f(x)\ dx \right)^{1+\varepsilon}.
\end{equation}

Therefore, \eqref{eq:Ge-J_0's}, \eqref{eq:Ge-J_01} and \eqref{eq:Ge-J_02} yield that
\begin{equation}\label{eq:Gehring-J_0}
J_0 
\geq  h(r_1) - \frac{M}{(r_2 - r_1)^{n\varepsilon}}.
\end{equation}
The estimate of $J_0$ is completed.

We now estimate $J_1$. Using Fubini's theorem and \eqref{eq:lcr-1} (with $r=1-d+\varepsilon$), we obtain
\begin{align}\label{eq:Ge-J_1-0}
J_1 &= \varepsilon \int_{\lambda_0}^{t} \lambda^{-d+\varepsilon} \left( \int_{E(r_2,\theta\lambda)} f(x)^d\ dx \right)\ d\lambda  \\
&\leq \varepsilon \int_{0}^{t} \lambda^{-d+\varepsilon} \left( \int_{E(r_2,\theta\lambda)} f(x)^d\ dx \right)\ d\lambda  \\
&= \int_{B_{r_2}} f(x)^d \left(\varepsilon \int_{0}^{t} \lambda^{-d+\varepsilon}\chi_{\{ B_{r_2} : f > \theta \lambda \}}(x) \ d\lambda  \right)\ dx \\
&= \int_{B_{r_2}} f(x)^d \left(\varepsilon \int_{0}^{\infty} \lambda^{(1-d+\varepsilon)-1}\chi_{ \left\{ B_{r_2} \min\{\theta^{-1}f, t\}> \lambda \right\} }(x) \ d\lambda  \right)\ dx \\
&=  \frac{\varepsilon}{1-d+\varepsilon}\int_{B_{r_2}} f(x)^d \min\left\{\theta^{-1}f(x), t \right\}^{1-d+\varepsilon} dx.
\end{align}
By $d < 1$ we have
\begin{equation}
f^d \min\left\{\theta^{-1}f, t \right\}^{1-d+\varepsilon}
= \theta^{-1+d-\varepsilon}f^d\min\left\{f, \theta t \right\}^{1-d+\varepsilon}
\le \theta^{-1+d-\varepsilon} f \min\left\{f, \theta t \right\}^{\varepsilon}.
\end{equation}
Therefore, \eqref{eq:Ge-J_1-0} implies
\begin{equation}\label{eq:Ge-J_1}
J_1 
\le \frac{\theta^{-1+d-\varepsilon}\varepsilon}{1-d+\varepsilon}\int_{B_{r_2}}  f \min\left\{f, \theta t \right\}^{\varepsilon}
= \frac{\theta^{-1+d-\varepsilon}\varepsilon}{1-d+\varepsilon}h(r_2).
\end{equation}
This completes the estimate of $J_1$.

Finally, we evaluate $J_2$. As for $J_1$, we get 
\begin{align}\label{eq:Gehring-J_2}
J_2 
&\leq \varepsilon \int_{0}^{\infty} \lambda^{\varepsilon - 1} \left( \int_{\mathcal{E} \left(r_2, \frac{\lambda}{4} \right)} g(x)\ dx \right)\ d\lambda \\
&= \int_{B_{r_2}} g(x) \left(\varepsilon \int_{0}^{\infty} \lambda^{\varepsilon - 1}\chi_{ \left\{ B_{r_2}:g > \frac{\lambda}{4} \right\} }(x) \ d\lambda  \right)\ dx \\
&\leq 4^{\varepsilon}\int_{B_{2R}} g(x)^{1 + \varepsilon}\ dx
\eqqcolon  L.
\end{align}

Combining \eqref{eq:Ge-J's}, \eqref{eq:Gehring-J_0}, 
\eqref{eq:Ge-J_1} and \eqref{eq:Gehring-J_2}, we obtain
\begin{equation}\label{eq:Ge-nc}
h(r_1)
\le \frac{c_1\theta^{-1+d-\varepsilon}\varepsilon}{1-d+\varepsilon}h(r_2)
+ c_2L + \frac{M}{(r_2 - r_1)^{n\varepsilon}}
\end{equation}
for any $t > \lambda_0$.
For $0 < t \le \lambda_0$, we have
\begin{align}
h(r_1) 
\le \int_{B_{r_1}} f (\theta t)^\varepsilon 
\le \lambda_0^{\varepsilon} \int_{B_{2R}} f
\le \frac{M}{(r_2 - r_1)^{n\varepsilon}}.
\end{align}
Therefore, \eqref{eq:Ge-nc} holds for every $t \in (0,\infty)$.
In order to apply Lemma \ref{lem:iteration} with this inequality, we now estimate the coefficient of $h(r_2)$.
By \eqref{eq:Gehring-d} and \eqref{eq:Gehring-theta}, we have
\begin{equation}
\frac{1}{1-d} = \frac{2}{1-\kappa}
\end{equation}
and
\begin{equation}
\theta^{-1+d-\varepsilon} 
= (4A+1)^{\frac{1+\varepsilon}{d}-1} 
\le (4A+1)^{2(1+\varepsilon)-1}
\le (4A+1)^{1+2\varepsilon_0}.
\end{equation}
Therefore,
\begin{equation}
\frac{c_1\theta^{-1+d-\varepsilon}\varepsilon}{1-d+\varepsilon}
\le \frac{2c_1(4A+1)^{1+2\varepsilon_0}\varepsilon}{1-\kappa}.
\end{equation}
Moreover, we choose $\varepsilon > 0$ such that
\begin{equation}
\varepsilon \le \frac{1-\kappa}{4c_1(4A+1)^{1+2\varepsilon_0}}
\ 
\left( \iff \frac{2c_1(4A+1)^{1+2\varepsilon_0}\varepsilon}{1-\kappa} \le \frac{1}{2} \right).
\end{equation}
Note that, as the proof below shows, we may take $c^* = 4c_1(4A+1)^{1+2\varepsilon_0}$. In this case, $c^*$ is monotonically increasing with respect to $\varepsilon_0$.

Then by \eqref{eq:Ge-nc} we have
\begin{equation}
h(r_1)
\le \frac{1}{2}h(r_2)
+ c_2L + \frac{M}{(r_2 - r_1)^{n\varepsilon}}.
\end{equation}
Since the function $h$ is bounded, applying Lemma \ref{lem:iteration} with this inequality, we deduce
\begin{equation}
h(R) \le 2c_2L + \frac{\widehat{c}(n,\varepsilon)M}{R^{n\varepsilon}}.
\end{equation}
That is,
\begin{equation}
\int_{B_{R}} f \min\left\{f, \theta t \right\}^\varepsilon
\le \frac{20^{n\varepsilon}\widehat{c}(n,\varepsilon)}{|B_{2R}|^{\varepsilon}}\left( \int_{B_{2R}} f \right)^{1+\varepsilon}
+ 2\cdot4^{\varepsilon}c_2\int_{B_{2R}} g^{1+\varepsilon},
\end{equation}
where $\widehat{c}(n,\varepsilon) = \sum_{i = 0}^{\infty} 2^{-i}\{(i+1)(i+2)\}^{n\varepsilon}$.
First, letting $t \rightarrow \infty$, we obtain
\begin{equation}
  \int_{B_{R}} f^{1 + \varepsilon}
\le \frac{20^{n\varepsilon}\widehat{c}(n,\varepsilon)}{|B_{2R}|^{\varepsilon}}\left( \int_{B_{2R}} f \right)^{1+\varepsilon}
+ 2\cdot4^{\varepsilon}c_2\int_{B_{2R}} g^{1+\varepsilon}.
\end{equation}
Next, multiplying both sides by $\frac{1}{|B_R|} = \frac{2^n}{|B_{2R}|}$, it follows that
\begin{equation}
\fint_{B_{R}} f^{1+\varepsilon}
\le  2^n 20^{n\varepsilon}\widehat{c}(n,\varepsilon)\left( \fint_{B_{2R}} f \right)^{1+\varepsilon}
+ 2^{n+1}\cdot4^{\varepsilon}c_2\fint_{B_{2R}} g^{1+\varepsilon}.
\end{equation}
Finally, raising both sides to the power $\frac{1}{1+\varepsilon}$, we derive
\begin{equation}
\left( \fint_{B_{R}} f^{1+\varepsilon} \right)^{\frac{1}{1 + \varepsilon}}
\le
 \widehat{c}(n,\varepsilon_0)\fint_{B_{2R}} f +
c(n)\left( \fint_{B_{2R}} g^{1 + \varepsilon} \right)^{\frac{1}{1+\varepsilon}}.
\end{equation}
Here we used
$$
\left( \widehat{c}(n,\varepsilon)  \right)^{\frac{1}{1+\varepsilon}}
\le \widehat{c}(n,\varepsilon) 
\le \widehat{c}(n,\varepsilon_0).
$$
The proof is completed.
\end{proof}

\end{document}